\documentclass[12pt]{article}
\usepackage{amsthm}

\usepackage[ansinew]{inputenc}
\usepackage{graphicx}
\usepackage{color}
\usepackage{soul}
\usepackage[colorlinks]{hyperref}

\usepackage{enumerate,latexsym}
\usepackage{latexsym}
\usepackage{amsmath,amssymb}
\usepackage{graphicx}
\usepackage{times}

\newfont{\bb}{msbm10 at 11pt}
\newfont{\bbsmall}{msbm8 at 8pt}

\def\R{\mathbb{R}}

\def\N{\mathbb{N}}
\def\Z{\mathbb{Z}}
\def\S{\mathbb{S}}
\def\C{\mathbb{C}}

\def\Hip{\mathbb{H}}

\def\D{\mathbb{D}}

\def\esf{\mathbb{S}}

\newcommand{\la}{\looparrowright}

\newcommand{\ben}{\begin{enumerate}}
\newcommand{\bit}{\begin{itemize}}
\newcommand{\een}{\end{enumerate}}
\newcommand{\eit}{\end{itemize}}
\newcommand{\wh}{\widehat}
\newcommand{\Int}{\mbox{\rm Int}}

\newcommand{\Div}{\mbox{\rm div}}

\newcommand{\wt}{\widetilde}

\newcommand{\ed}{\end{document}}
\newcommand{\ov}{\overline}

\def\a{{\alpha}}

\def\t{{\theta}}

\def\g{{\gamma}}
\def\G{{\Gamma}}
\def\l{{\lambda}}

\def\de{{\delta}}
\def\be{{\beta}}
\def\ve{{\varepsilon}}

\def\cU{\mathcal{U}}

\def\cB{\mathcal{B}}

\def\cM{\mathcal{M}}

\newcommand{\chM}{\mathcal{M}^1}

\def\cF{\mathcal{F}}

\let\8=\infty \let\0=\emptyset  
\let\hat=\widehat

\let\landa=\lambda
\let\alfa=\alpha

\let\parc=\partial

\def\ep{\varepsilon}

\def\landa{\lambda}

\def\flecha{\rightarrow}
\def\esiz{\langle}
\def\esde{\rangle}

\def\cte.{\mathop{\rm cte.}\nolimits}

\def\det{\mathop{\rm det}\nolimits}

\def\cosh{\mathop{\rm cosh }\nolimits}

\def\N{\mathbb{N}}
\def\E{\mathbb{E}}

\def\R{\mathbb{R}}
\def\Z{\mathbb{Z}}
\def\C{\mathbb{C}}

\def\D{\mathbb{D}}

\def\H{\mathbb{H}}
\def\S{\mathbb{S}}

\def\psl{{\rm PSL}(2,\mathbb{R})}

\newtheorem{theorem}{Theorem}[section]
\newtheorem{lemma}[theorem]{Lemma}
\newtheorem{proposition}[theorem]{Proposition}

\newtheorem{remark}[theorem]{Remark}
\newtheorem{corollary}[theorem]{Corollary}
\newtheorem{definition}[theorem]{Definition}

\newtheorem{assertion}[theorem]{Assertion}

\newtheorem{claim}[theorem]{Claim}

\newcommand{\calf}{{\mathcal F}}

\newcommand{\su}{{\rm SU}(2)}

\renewcommand{\sl}{\wt{\mathrm{SL}}(2,\R)}

\numberwithin{equation}{section}

\usepackage{fullpage}

\usepackage{pdfsync}
\definecolor{pp}{rgb}{.5,0,.7}

\begin{document}

\begin{title}
{Constant mean curvature spheres in homogeneous three-manifolds}
\end{title}

\begin{author}
{William H. Meeks III\thanks{This material  is based upon
 work for the NSF under Award No. DMS -
  1309236. Any opinions, findings, and conclusions or recommendations
 expressed in this publication are those of the authors and do not
 necessarily reflect the views of the NSF.}
\and Pablo Mira\thanks{The second author was partially supported
by MINECO/FEDER, Grant no. MTM2016-80313-P, and Programa
de  Apoyo a la Investigacion, Fundacion Seneca-Agencia de
Ciencia y Tecnologia Region de
Murcia, 19461/PI/14.}
 \and Joaqu\'\i n P\' erez
\and Antonio Ros\thanks{Research of the third and fourth authors partially supported by
a MINECO/FEDER grant no. MTM2014-52368-P. } }
\end{author}
\maketitle
\vspace{-.6cm}

\begin{abstract}
We prove that two spheres of the same constant mean curvature
in an arbitrary homogeneous three-manifold only differ by an ambient isometry,
and we determine the values of the mean curvature for which such spheres exist.
This gives a complete classification of immersed constant
mean curvature spheres in three-dimensional homogeneous manifolds.

\vspace{.3cm}


\end{abstract}
\maketitle

\section{Introduction}  \label{sec:introduction}
In this manuscript we  solve the fundamental problem of classifying constant mean curvature spheres in
an arbitrary homogeneous three-manifold,
where by a \emph{sphere}, we mean a closed immersed surface of genus zero:

\begin{theorem}\label{main}
Let $M$ be a Riemannian homogeneous three-manifold, $X$ denote its
Riemannian universal cover, and let ${\rm Ch}(X)$ denote the Cheeger constant of $X$. Then,
any two spheres in $M$ of the same constant mean curvature differ by an isometry of $M$. Moreover:
 \begin{enumerate}[(1)]
 \item
If $X$ is not diffeomorphic to $\R^3$, then, for every $H\in \R$, there
exists a sphere of constant mean curvature $H$ in $M$.
 \item
If $X$ is diffeomorphic to $\R^3$, then the values $H\in \R$\, for which there exists a
sphere of constant mean curvature $H$ in $M$ are exactly those with $|H|>{\rm Ch}(X)/2$.
 \end{enumerate}
\end{theorem}

Our study of  constant
mean curvature spheres in homogeneous three-manifolds provides a natural parameterization
of their moduli space
and  fundamental information about their geometry and symmetry, as explained in the next theorem:

\begin{theorem}
\label{main2}
Let $M$ be a Riemannian homogeneous three-manifold, $X$ denote its
Riemannian universal cover and
let $S_H$ be a sphere of constant mean curvature $H\in \R$ in $M$. Then:
\begin{enumerate}[(1)]
\item
$S_H$ is \emph{maximally symmetric}; that is, there exists a point $p\in M$, which we call
the \emph{center of symmetry} of $S_H$, with the property that any isometry $\Phi$ of $M$ that
fixes $p$ also satisfies $\Phi(S_H)=S_H$. In particular,
constant mean curvature spheres are totally umbilical if $M$ has constant sectional curvature,
and are spheres of revolution if $M$ is rotationally symmetric.
 \item
If $H=0$ and $X$ is a Riemannian product  $\esf^2(\kappa)\times \R$, where
$\esf^2(\kappa)$ is a sphere
of constant curvature $\kappa>0$, then it is well known that $S_H$ 
is totally geodesic, stable and has nullity $1$ for its stability
operator.  Otherwise,
$S_H$ has index $1$ and nullity $3$ for its stability operator and the immersion of $S_H$
into $M$ extends as the boundary of an isometric immersion $F\colon B \to M$ of a Riemannian
 3-ball $B$ which is mean convex\footnote{That the ball $B$ is mean convex means that
 the mean curvature of the boundary $\partial B$ with respect to its inward pointing normal vector
 is nonnegative.}.
\end{enumerate}
Moreover, if $\cM_M(p)$ denotes the space of spheres of nonnegative constant mean curvature
in $M$ that have a base point $p\in M$ as a center of symmetry, then
the map $S_H\in \cM_M(p)\mapsto H\in \R$ that assigns to each sphere $S_H$ its mean curvature
is a homeomorphism 
between: (i) $\cM_M(p)$ and $[0,\8)$ if $X$ is
not diffeomorphic to $\R^3$, or (ii) $\cM_M(p)$ and
$({\rm Ch}(X)/2,\8)$ if $X$ is diffeomorphic to $\R^3$.
This homeomorphism is real analytic except at $H=0$ when
$X=\S^2(\kappa)\times \R$.
\end{theorem}

The classification of constant mean curvature spheres is an old problem. In the 19th century, Jellet~\cite{je1}
proved that any constant mean curvature sphere in $\R^3$ that is star-shaped with respect to
some point is a round sphere. In 1951, Hopf~\cite{hf1} introduced a holomorphic quadratic
differential for any constant mean curvature surface in $\R^3$ and then used it to prove that
constant mean curvature spheres in $\R^3$ are round. His proof also worked in
the other simply-connected three-dimensional spaces of constant sectional curvature,
which shows that these spheres are, again, totally umbilical (hence, they are the
boundary spheres of geodesic balls of the space); see e.g.~\cite{Che5,che6}.

In 2004,  Abresch and Rosenberg~\cite{AbRo1,AbRo2} proved that any constant mean curvature
sphere in the product spaces $\H^2(\kappa)\times \R$ and $\S^2(\kappa)\times \R$, 
where $\esf^2(\kappa)$ (resp. $\H^2(\kappa)$)
denotes the two-dimensional sphere (resp. the hyperbolic plane) of
constant Gaussian curvature $\kappa\neq 0$, and more generally, in any
simply connected Riemannian homogeneous three-manifold with an isometry group of dimension four,
is a rotational sphere. Their theorem settled an old problem posed by Hsiang-Hsiang~\cite{hshs1},
and reduced the classification of constant mean curvature spheres in these homogeneous spaces to
an ODE analysis. For their proof, Abresch and Rosenberg introduced a \emph{perturbed} Hopf
differential, which turned out to be holomorphic for constant mean curvature surfaces in these spaces.
The Abresch-Rosenberg theorem was
the starting point for the development of the theory of constant mean curvature surfaces in
rotationally symmetric homogeneous three-manifolds; see e.g.,~\cite{dhm1,fm4} for a survey
on the beginnings of this theory.

It is important to observe here that a generic homogeneous three-manifold $M$ has an isometry
 group of dimension three, and that the techniques used by Abresch and Rosenberg in the rotationally
symmetric case do not work to classify constant mean curvature spheres in such
an $M$.

In 2013, Daniel and Mira~\cite{dm2} introduced a new method for studying constant mean curvature
spheres in the homogeneous manifold ${\rm Sol}_3$ with its usual Thurston geometry. Using this
method, they classified constant mean curvature  spheres in ${\rm Sol}_3$ for values $H$ of the
mean curvature greater than ${\large \frac{1}{\sqrt{3}}}$,  and reduced the general classification problem to the
obtention of area estimates for the family of spheres of  constant mean curvature greater
than any given positive number. These
crucial area estimates were subsequently proved by  Meeks~\cite{me34}, which completed the
classification of constant mean curvature spheres in  ${\rm Sol}_3$: \emph{For any $H>0$ there is a
unique constant mean curvature sphere $S_H$ in ${\rm Sol}_3$ with mean curvature $H$; moreover,
$S_H$ is maximally symmetric, embedded, and has index one}  \cite{dm2,me34}.

In~\cite{mmpr4} we extended the Daniel-Mira theory in~\cite{dm2} to arbitrary simply connected homogeneous
three-manifolds $X$ that cannot be expressed as a Riemannian product of
a two-sphere of constant curvature with a line. However, the area estimates problem in this general setting
cannot be solved following Meeks' approach in ${\rm Sol}_3$, since the proof by Meeks uses in an
essential way special properties of ${\rm Sol}_3$ that are not shared by a general $X$.

In order to solve the area estimates problem in any homogeneous three-manifold,
the authors have developed in previous works an extensive theoretical background for the
study of constant mean curvature surfaces in homogeneous
three-manifolds, see \cite{mpe11,mmpr4,mmpr2,mmp1,mmp2}.
Specifically, in~\cite{mpe11} there is a
detailed presentation of the geometry of \emph{metric Lie groups}, i.e.,
simply connected homogeneous three-manifolds given by a Lie group endowed with a left
invariant metric. In~\cite{mmpr4,mpe11} we described the basic theory of constant mean
curvature surfaces in metric Lie groups. Of special importance for our study here is~\cite{mmpr4},
where we extended the Daniel-Mira theory~\cite{dm2} to arbitrary metric Lie groups $X$,
and we proved Theorems~\ref{main} and \ref{main2} in the case that $M$ is
a homogeneous three-sphere. In~\cite{mmp2} we studied the geometry of spheres in metric Lie groups whose left invariant
Gauss maps are diffeomorphisms. In~\cite{mmpr2} we proved that any metric
Lie group $X$ diffeomorphic to $\R^3$ admits a foliation by leaves of
constant mean curvature of value ${\rm Ch}(X)/2$. In~\cite{mmp1} we established uniform radius estimates for
certain stable minimal surfaces in metric Lie groups. All these works are used in the
present manuscript to solve the area estimates problem.

We note that the uniqueness statement in Theorem~\ref{main} does not follow
from the recent G\'alvez-Mira uniqueness theory for immersed spheres in
three-manifolds in~\cite{gm1}. Indeed, their results imply in our context
(see Step~2 in the proof of Theorem~4.1 in~\cite{mmpr4}) that two spheres
$\Sigma_1,\Sigma_2$ of the same constant mean curvature in a metric Lie
group $X$ are congruent provided that the left invariant Gauss map of
one of them is a diffeomorphism (see Definition~\ref{defG} for the notion of
left invariant Gauss map); Theorem~\ref{main} does not assume this additional
hypothesis.

The main results of this paper, Theorems~\ref{main} and~\ref{main2},
can be reduced to demonstrating Theorem~\ref{topR3case} below; this reduction
is explained in Section~\ref{sec:outline}. To state Theorem~\ref{topR3case},
we need the notion of entire Killing graph given in the next definition, where $X$ is a simply
connected homogeneous three-manifold.

\begin{definition}
\label{defKgraph}
{\em
Given a nowhere zero Killing field $K$ on $X$, we say that an
immersed surface $\Sigma\subset X$ is a {\em Killing graph}
with respect to $K$ if whenever $\Sigma$ intersects an integral curve of $K$, this intersection
is transversal and consists of a single point. If additionally $\Sigma $ intersects
every integral curve of $K$, then we say that $\Sigma $ is an {\em entire} Killing graph.
}
\end{definition}

 By classification, any homogeneous manifold $X$ diffeomorphic to $\R^3$ is
isometric to a metric Lie group.
In this way, any nonzero vector field $K$ on $X$ that is right invariant for such Lie group
structure is a nowhere-zero Killing field on $X$. We note that any integral curve of $K$ is a
properly embedded curve in $X$ diffeomorphic to $\R$.

\begin{theorem}
\label{topR3case}
Let $X$ be a homogeneous manifold diffeomorphic to $\R^3$, and let $S_n$
 be a sequence of
constant mean curvature spheres in $X$ with $\mbox{\rm Area}(S_n)\to \8$.
Then, there exist isometries $\Phi_n$ of $X$ and compact domains $\Omega_n\subset \Phi_n(S_n)$
with the property that a subsequence of the $\Omega_n$ converges uniformly
on compact subsets of $X$ to an entire Killing graph
with respect to some nowhere zero Killing vector field on $X$, which in fact is right
invariant with respect to some Lie group structure on $X$.
\end{theorem}

\section{Organization of the paper.}

We next explain the organization of the paper and  outline of the strategy of the proof of
the main theorems. In our study in~\cite{mmpr4} of the space of constant mean curvature spheres of
index one in metric Lie groups $X$, we proved that when  $X$ is diffeomorphic to $\esf^3$,
Theorems~\ref{main} and \ref{main2} hold for $M=X$. In that same paper
we proved that if $X$ is a metric Lie group not diffeomorphic
to $\S^3$ (and hence, diffeomorphic to $\R^3$),
then there exists a constant $h_0(X)\geq 0$ such that: (i) for any $H>h_0(X)$
there is an index-one sphere $S_H$ of constant mean curvature
$H$ in $X$; (ii) any sphere of constant mean curvature $H>h_0(X)$ is a left
translation of $S_H$, and (iii) if $H_n\searrow h_0(X)$, then the areas of the spheres $S_{H_n}$ diverge to $\8$
as $n\to \8$; furthermore, the spheres $S_n$ bound isometrically immersed mean convex
Riemannian three-balls $\wt{f}_n\colon B_n\to X$.

In Section~\ref{sec:outline} we explain how Theorems~\ref{main} and \ref{main2} can be
deduced from Theorem~\ref{topR3case} and from the result stated in the last paragraph.

Section~\ref{sec:gauss} is an introductory section; in it we prove some basic properties of
constant mean curvature surfaces $f\colon\Sigma \la X$ in a metric Lie
group $X$ that are \emph{invariant}, i.e., such that $f(\Sigma)$ is
everywhere tangent to a nonzero right invariant Killing vector field $K_\Sigma$ on $X$ (equivalently,
$f(\Sigma )$ is invariant under the flow of $K_\Sigma$),
and we explain in more detail the geometry of metric Lie groups diffeomorphic to $\R^3$.
For this analysis, we divide these metric Lie groups into two categories: \emph{metric semidirect products}
and those of the form $(\sl,\esiz,\esde)$, where $\sl$ is the universal cover of the special
linear group ${\rm SL}(2,\R)$ and $\esiz,\esde$ is a left invariant metric; here the metric semidirect products $X$
under consideration are the
semidirect product of a normal subgroup $\R^2$ with $\R$
and we refer to the cosets in $X$ of $\R^2$ as being horizontal
planes.

In Section~\ref{sec:limits} we prove that if $\{S_n\}_n$ is a sequence of index-one spheres with
constant mean curvatures of values $H_n\searrow h_0(X)$ in a metric Lie group $X$ diffeomorphic
to $\R^3$, then, after taking limits of an appropriately chosen subsequence of left translations of
these spheres, there exists an invariant
surface $f\colon\Sigma \la X$ of constant mean curvature $h_0(X)$ that is a limit of compact
domains of the spheres $S_n$. Moreover, $f$ is complete, stable,
$\Sigma $ has the topology of a plane or a cylinder,
and if $f(\Sigma)$ is not a coset of a two-dimensional subgroup of $X$,
then the closure $\overline{\g_f}$ of its left invariant Gauss map image $\g_f$ in $\S^2$
(see Definition~\ref{defG} for this notion)
has the structure of a one-dimensional lamination.
Another key property  proved in Section~\ref{sec:limits} is that
whenever $f(\Sigma)$ is tangent to some coset of a two-dimensional subgroup of $X$,
then $f(\Sigma)$ lies in one of the two closed \emph{half-spaces} bounded by this coset;
see Corollary~\ref{cor:unique}.

In Section~\ref{sec:periodic} we prove in
Theorem~\ref{g_f closed} that the invariant limit
surface $f(\Sigma)$ can be chosen so that its left invariant Gauss map image $\g_f\subset \S^2$
is a point or a closed embedded regular curve. In the case $\g _f$ is a point, Theorem~\ref{topR3case}
follows easily from the fact that $f(\Sigma )$ is a coset of a two-dimensional subgroup of $X$. In the case that
$\g _f$ is a closed curve, one can choose $f$ so that one of the following exclusive possibilities occurs:
 \begin{enumerate}
 \item
$\Sigma$ is diffeomorphic to an annulus.
 \item
$\Sigma$ is diffeomorphic to a plane, and $f(\Sigma)$ is an immersed annulus in $X$.
 \item
$\Sigma$ is diffeomorphic to a plane and there exists an element $a\in X$ such that the left
translation by $a$ in $X$ leaves $f(\Sigma)$ invariant,
but this left translation does not lie in the $1$-parameter subgroup of isometries generated by 
the nonzero right invariant Killing vector field $K_{\Sigma}$ that is everywhere
tangent to $f(\Sigma)$.
 \end{enumerate}

In Section~\ref{sec:semidirect} we prove Theorem~\ref{topR3case} in the case that $X$ is a
metric semidirect product. For that, it suffices to prove that $f(\Sigma)$ is an entire
Killing graph for some right invariant Killing vector field $V$ in $X$. This is proved as follows.
First, we show that if case 3 above holds, then the Killing field $K_{\Sigma}$ which
is everywhere tangent to $f(\Sigma)$ is horizontal with respect to the semidirect product structure of $X$,
and we prove that $f(\Sigma)$ is an entire
graph with respect to any other horizontal Killing field $V$ linearly independent from
$K_{\Sigma}$. So, it suffices to rule out cases 1 and 2 above. We prove that case 2 is
impossible by constructing in that situation geodesic balls of a certain fixed radius
$R^*>0$ in the abstract Riemannian three-balls $B_n$ that the Alexandrov-embedded index-one
spheres $S_n$ bound, and whose volumes tend to infinity as $n\to \8$.
This unbounded volume result eventually provides a contradiction with Bishop's theorem.
Finally, we show that case 1 is impossible by constructing an abstract
three-dimensional cylinder bounded by $\Sigma$ that submerses
isometrically into $X$ with boundary
$f(\Sigma)$, and then proving that a certain CMC flux
of $\Sigma$ in this abstract cylinder is
different from zero. This gives a contradiction with the homological invariance of the
CMC flux and the fact that $f(\Sigma)$ is a limit of the (homologically trivial)
Alexandrov-embedded constant mean curvature spheres $S_n$.

In Section~\ref{sec:sl2R} we prove Theorem~\ref{topR3case} in the remaining case
where $X$ is isomorphic to $\sl$. The arguments and the basic strategy of the proof in
this situation follow closely those from the previous Section~\ref{sec:semidirect}.
However, several of these arguments are by necessity different from those in
Section~\ref{sec:semidirect}, as many geometric properties of metric semidirect products
do not have analogous counterparts in $\sl$.

The paper finishes with an Appendix that can be read independently of the rest of the manuscript.
In it, we prove a general nonvanishing result for the CMC flux that is used in
Sections~\ref{sec:semidirect} and \ref{sec:sl2R} for ruling out case 1 above,
as previously explained.

\section{Reduction of Theorems~\ref{main}
and \ref{main2} to Theorem~\ref{topR3case}}
\label{sec:outline}

In order to show that Theorems~\ref{main} and \ref{main2} are implied by Theorem~\ref{topR3case},
we will prove the next two assertions in Sections~\ref{sec3.1} and \ref{sec3.2}, respectively.

\begin{assertion}\label{ass31}
Assume that Theorem~\ref{topR3case} holds. Then Theorems~\ref{main} and \ref{main2} hold for the
particular case that the homogeneous three-manifold $M$ is simply connected.
\end{assertion}

\begin{assertion}\label{ass32}
Theorems~\ref{main} and \ref{main2} hold provided they hold for the particular case that
the homogeneous three-manifold $M$ is simply connected.
\end{assertion}

In what follows, $M$ will denote a homogeneous three-manifold and $X$ will denote
its universal Riemannian cover; if $M$ is simply connected, we will identify $M=X$.
By an $H$-\emph{surface} (or $H$-\emph{sphere}) in $M$ or $X$ we will mean an
immersed surface (resp. sphere) of constant mean curvature of value $H\in \R$ immersed in $M$ or $X$.

The \emph{stability operator} of an $H$-surface $f\colon \Sigma \la M$ is the Schr\"{o}dinger operator 
$L=\Delta + q$, where $\Delta$ stands for the Laplacian with respect to the induced metric in $\Sigma$, 
and $q$ is the smooth function given by $$q=|\sigma|^2 + {\rm Ric}(N),$$ where $|\sigma|^2$ 
denotes the squared norm of the second fundamental form of $f$, and $N$ is the unit normal vector field to $f$. 
The immersion $f$ is said to be \emph{stable} if $-L$ is a nonnegative operator. 
When $\Sigma$ is closed, the \emph{index} of $f$ is defined as the number of negative eigenvalues 
of $-L$, and the \emph{nullity} of $f$ is the dimension of the kernel of $L$. We say that a function 
$u\in C^{\8}(\Sigma)$ is a \emph{Jacobi function} if $Lu=0$. Functions of the type $u=\esiz N,F\esde$ 
where $F$ is a Killing vector field on the ambient space $M$ are always Jacobi functions on $\Sigma$.

\subsection{Proof of Assertion~\ref{ass31}}\label{sec3.1}

Let $M=X$ denote a simply connected homogeneous three-manifold. It is well known that if $X$ is not
isometric to a Riemannian product space $\S^2(\kappa )\times \R$, where $\S^2(\kappa )$ denotes a sphere of
constant Gaussian curvature $\kappa >0$, then $X$ is diffeomorphic to $\R^3$ or to $\S^3$.

When $X$ is isometric to $\S^2(\kappa )\times \R$, the statements contained in
Theorems~\ref{main} and \ref{main2} follow from the Abresch-Rosenberg theorem that
states that constant mean curvature spheres in $\S^2(\kappa)\times \R$ are rotational spheres, and from
an ODE analysis of the profile curves
of these spheres. More specifically, an explicit expression of the rotational $H$-spheres
in $\S^2(\kappa)\times \R$ can be found in Pedrosa and Ritor\'e~\cite[Lemma~1.3]{pri1},
or in~\cite{AbRo1}.
From this expression it follows that if $H\neq 0$, then each of these
rotational $H$-spheres is embedded, and bounds a unique compact
subdomain in $X$ which is a mean convex ball that is invariant under all the
isometries of $X$ that fix the mid point in the ball of the revolution axis of
the sphere (this point can be defined as the \emph{center of symmetry} of the sphere).
If $H=0$, the corresponding rotational sphere is a slice $\S^2(\kappa)\times \{t_0\}$,
that is trivially stable, and hence has index zero and nullity one. Moreover, any point
of $\S^2(\kappa)\times \{t_0\}$ can be defined as the center of symmetry of the sphere
in this case. That the rotational $H$-spheres in  $\S^2(\kappa )\times \R$
for $H\neq 0$ have index one and nullity
three is shown in Souam~\cite[proof of Theorem 2.2]{so3}.
Also, the rotational $H$-spheres for $H> 0$
with a fixed center of symmetry $p_0$ converge as $H\to 0$
to a double cover of the slice $\S^2(\kappa)\times \{t_0\}$
that contains $p_0$.
So, from all this information
and the Abresch-Rosenberg theorem, it follows that
Theorems~\ref{main} and \ref{main2} hold
when $X$ is isometric to $\S^2(\kappa)\times \R$.

When $X$ is diffeomorphic to $\S^3$, the statements in Theorem~\ref{main}
and \ref{main2} were proven in our previous work~\cite{mmpr4}. 

So, from now on in this Section~\ref{sec3.1},
$X$ will denote a homogeneous manifold diffeomorphic to
$\R^3$. It is well known that $X$ is isometric to a \emph{metric Lie group},
i.e., a three-dimensional Lie group endowed with a left invariant Riemannian metric.
In the sequel, $X$ will be regarded as a metric Lie group, and $e$ will denote its identity element.
Given $x\in X$, we will denote by $l_x,r_x\colon X\to X$ the  left and right translations
by $x$, respectively given by $l_x(y)=xy$, $r_x(y)=yx$ for all $y\in X$.
Thus, $l_x$ is an isometry of $X$ for every $x\in X$.

\begin{definition}
\label{defG}
{\rm Given an immersed oriented surface $f\colon \Sigma
\looparrowright X$ with unit normal vector field $N\colon \Sigma \to
TX$ (here $TX$ refers to the tangent bundle of $X$), we define the
{\it left invariant Gauss map} of $f$ to be the map
$G\colon \Sigma \to \esf^2\subset T_eX$ that assigns to each $p\in
\Sigma $, the unit tangent vector $G(p)$ to $X$ at the identity element $e$
given by $(dl_{f(p)})_e(G(p))=N_p$. }
\end{definition}

The following proposition follows from a rearrangement of ideas taken from~\cite{mmpr4}, but for
the sake of clarity, we will include a proof here.

\begin{proposition}
\label{propsu2}
Let $X$ be any metric Lie group $X$ diffeomorphic to $\R^3$. Then, there exists a number $h_0(X)\geq 0$ such that:
\begin{enumerate}[(1)]
\item
For every $|H|>h_0(X)$, there exists an $H$-sphere $S_H$ in $X$
with index one and nullity three for its stability operator. Furthermore, $S_H$ is Alexandrov embedded,
i.e., $S_H$ is the boundary of a (unique) mean convex immersed ball $\wt{f}_n\colon B_n\to X$ in $X$.
 \item
If $\Sigma$ is an $H$-sphere in $X$ for some $|H|>h_0(X)$, then $\Sigma$ is a left translation of
$S_H$.
 \item
The left invariant Gauss map $G\colon S_H\flecha \S^2$ of $S_H$ 
is an orientation-preserving diffeomorphism.
 \item
Each $S_H$ lies inside a real-analytic family $\{\Sigma_{H'}\,  | \, H'\in (H-\ep,H+\ep)\}$ of
index-one spheres in $X$ for some $\ep>0$, where $\Sigma_H=S_H$ and $\Sigma_{H'}$ has
constant mean curvature  of value $H'$.
In particular, there exists a unique component $\mathcal{C}$ of the space of index-one spheres with
constant mean curvature in $X$ such that $S_H\in \mathcal{C}$ for all $H>h_0(X)$.
 \item
For every $H_0>h_0(X)$ there exists some positive constant $C=C(X,H_0)$ such that
if $|H| \in (h_0(X),H_0]$, then the norm of the second fundamental form of $S_H$ is at most $C$.
 \item
For every $H_0>h_0(X)$ there exists some positive constant $D=D(X,H_0)$ such that
if $|H|>H_0$, then the area of $S_H$ is at most $D$.
 \item
If $S_n$ is a sequence of constant mean curvature spheres in $X$ whose
mean curvatures $H_n> h_0(X)$ satisfy $H_n\to h_0(X)$,
then ${\rm Area} (S_n)\to \8$.
\end{enumerate}
\end{proposition}

\begin{proof}
After a change of orientation if necessary, it clearly suffices to prove all statements for the case where $H\geq 0$.
By~\cite[item~3 of Theorem~4.1]{mmpr4},
there exists a unique component $\mathcal{C}$ of the space of index-one spheres with constant
mean curvature in $X$ such that the values of the mean curvatures
of the spheres in $\mathcal{C}$ are not bounded from above.
As $X$ is diffeomorphic to $\R^3$, then~\cite[item~6 of Theorem~4.1]{mmpr4} ensures that the map
$\mathcal{H}\colon \mathcal{C}\to \R$ that assigns to each
 sphere in $\mathcal{C}$ its mean curvature is not surjective. By item~5 in the same theorem,
 there exists $h_0(X) \geq 0$ such that for every $H>h_0(X)$, there is an index-one sphere
$S_H\in \mathcal{C}$, and the following property holds:
\begin{enumerate}[(7)']
\item The areas of any sequence $\{ S_{H_n} \} _n \subset \mathcal{C}$ with $H\to h_0(X)$
satisfy $\lim _{n\to \infty }\mbox{\rm Area}(S_{H_n})=\infty $.
\end{enumerate}
Since the sphere $S_H$ has index one, then Cheng~\cite[Theorem~3.4]{cheng1} gives that the
nullity of $S_H$ is three. Alexandrov embeddedness of $S_H$ follows from~\cite[Corollary~4.4]{mmpr4}.
Now item~1 of Proposition~\ref{propsu2} is proved.
\newline
\mbox{}\qquad
Once this item 1 has been proved, items 2, 3, 4, 5 of Proposition~\ref{propsu2} are proved in items 1, 2, 3, 4
of~\cite[Theorem~4.1]{mmpr4}. Item 6 of Proposition~\ref{propsu2} is a direct consequence of the existence,
uniqueness and analyticity properties of the spheres $S_H$ stated in items 1,2,4,
and of the fact also proved in~\cite{mmpr4} that,
for $H$ large enough, the spheres $S_H$ bound small isoperimetric regions in $X$.
Finally, item~7 of Proposition~\ref{propsu2} follows from property (7)' above and
from the uniqueness given by the already proven item~2 of Proposition~\ref{propsu2}. Now the proof
is complete.
\end{proof}

A key notion in what follows is the {\it critical mean curvature of $X$.}
\begin{definition}
\label{defH(X)}
{\rm
Let $X$ be a homogeneous manifold diffeomorphic to $\R^3$, let $\mathcal{A}(X)$
be the collection of all closed, orientable immersed  surfaces in $X$,
and given a surface $\Sigma \in \mathcal{A}(X)$, let $|H_{\Sigma }|\colon \Sigma \to [0,\infty )$
 stand for the absolute mean curvature function of $\Sigma $. The \emph{critical mean curvature} of
$X$ is defined as
\[
H(X) =\inf \left\{ \max_{\Sigma }|H_{\Sigma }|\ : \ \Sigma \in \mathcal{A}(X)\right\} .
\]
}
\end{definition}

Item~2 of Theorem 1.4 in~\cite{mmpr2} gives that Ch$(X)=2H(X)$.
By definition of $H(X)$, every compact $H$-surface in $X$ satisfies $|H|\geq H(X)$.
Observe that Proposition~\ref{propsu2} gives that for every $H>h_0(X)$ there is an $H$-sphere
$S_H$ in $X$. Thus,
\begin{equation}\label{hoche}
h_0(X)\geq H(X)={\rm Ch}(X)/2.
\end{equation}

In order to prove Assertion~\ref{ass31}, {\bf in the remainder of this section we will assume that Theorem~\ref{topR3case} holds.} 

We will prove next that $h_0(X)=H(X)$.
If $S_n$ are constant mean curvature spheres satisfying item~7
of Proposition~\ref{propsu2},
 then by Theorem~\ref{topR3case} we obtain the existence of a surface $\Sigma_0$ in $X$ of
 constant mean curvature $h_0(X)$ that is an entire Killing graph with respect to some nonzero
 right invariant Killing field $K$. Let
$\Gamma=\{ \Gamma(t)\ | \ t\in \R \} $ be the $1$-parameter subgroup of $X$ given by $\Gamma(0)=e$, $\Gamma'(0)=K(e)$.
As left translations are isometries of $X$, it follows that
$\{ l_{a} (\Sigma_0 )\ | \ a\in \Gamma \} $ defines a
foliation of $X$ by congruent surfaces of constant mean curvature $h_0(X)$,
and this foliation is topologically a product foliation.
A standard application
of the mean curvature comparison principle shows then that there are no closed surfaces
$\Sigma$ in $X$ with $\max_{\Sigma }|H|< h_0(X)$ (otherwise, we left translate $\Sigma_0$ until $\Sigma$ is
contained in the region of $X$ on the mean convex side of $\Sigma _0$,
 and then start left translating $\Sigma_0$ towards its
mean convex side until it reaches a first contact point with $\Sigma$; this provides a contradiction
with the mean curvature comparison principle). Therefore, we conclude that
$h_0(X)=H(X)$ 
and that, by item 1 of Proposition~\ref{propsu2}, the values $H\in \R$ for which there exists a sphere of constant
mean curvature $H$ in $X$ are exactly those with $|H|>H(X)$. This fact together with item~2 of
Proposition~\ref{propsu2} proves item~2 of Theorem~\ref{main}, and thus
it completes the proof of Theorem~\ref{main} in the case $X$ is diffeomorphic to $\R^3$ (assuming that Theorem~\ref{topR3case} holds).

Regarding the proof of Theorem~\ref{main2} when $X$
is diffeomorphic to $\R^3$, similar arguments as the one in the preceding paragraph prove item~2 of Theorem~\ref{main2}.
To conclude the proof of Theorem~\ref{main2} (and hence of Assertion~\ref{ass31})
we need to show that all spheres $S_H$ in $X$ are maximally symmetric,
something that we will prove next. We note that the `\emph{Moreover}' part in the statement of Theorem~\ref{main2}
follows directly from the existence of a center of symmetry of each sphere $S_H$, and the analyticity properties in item 4
of Proposition~\ref{propsu2}.

\vspace{0.1cm}

In the case that the isometry group $I(X)$ of $X$ is of dimension 6, all constant mean
curvature spheres in $X$ are totally umbilical; in particular they are maximally symmetric,
and the definition of the center of symmetry of the sphere is clear. So,
in  the remainder of this section we will let $X$ be a metric Lie group diffeomorphic to $\R^3$ whose
isometry group $I(X)$ has dimension $3$ or $4$. Let $\mbox{Stab}_e(X)$ denote the group of
isometries $\psi$ of $X$ with $\psi(e)=e$, and let $\mbox{Stab}^+_e(X) $ denote the
index-two subgroup of orientation-preserving isometries in $\mbox{Stab}_e(X)$.
The next proposition gives some basic properties of the elements of
$\mbox{Stab}_e(X)$ that will be needed for proving the maximal symmetry of $H$-spheres in $X$.
Its proof follows from the analysis of metric Lie groups in~\cite{mpe11}.

\begin{proposition}
\label{prop:mass}
There exists a $1$-parameter subgroup $\G $ of $X$ such that:
 \begin{enumerate}[(1)]
 \item
There exists an isometry $\phi\in \mbox{Stab}_e^+(X)$ of order two in $X$ that
is a group automorphism of $X$, and whose fixed point set is
$\Gamma$, i.e., $\Gamma=\{x\in X : \phi(x)=x\}$.
\item
Every $\psi\in \mbox{Stab}_e(X)$ leaves $\Gamma$ invariant, in the sense that $\psi(\Gamma)=\Gamma$.
\end{enumerate}
Moreover, if there exists another $1$-parameter subgroup $\hat{\Gamma}\neq \Gamma$ of $X$ that
satisfies the previous two properties, then $X$ has a three-dimensional isometry group,
and $\phi(\hat{\Gamma})=\hat{\Gamma}$, where $\phi$ is given with respect to $\Gamma$ by
item (1) above.
\end{proposition}
\begin{proof}
First suppose that $I(X)$ has dimension 4. By~\cite[item~2 of Proposition~2.21]{mpe11}, there exists a
unique principal Ricci eigendirection $w\in T_eX$ whose associated eigenvalue is simple, and Stab$^+_e(X)$
contains an $\esf^1$-subgroup $A$, all whose elements have differentials that fix $w$. Let $\G $ be the
1-parameter subgroup of $X$ generated by $w$, and let $\phi $ be the orientation-preserving automorphism of
$X$ whose differential satisfies $d\phi _e(w)=w$, $(d\phi _e)|_{\langle w\rangle ^{\perp}}=-1_{\langle w
\rangle ^{\perp }}$ (see~\cite[proof of Proposition~2.21]{mpe11} for the construction of $\phi $). Clearly
$\phi $ satisfies item~1 of Proposition~\ref{prop:mass}. Regarding item~2, observe that
Stab$^+_e(X)=A\cup B$, where $B$ is the set of $\pi $-rotations about the geodesics $\g$
of $X$ such that $\g (0)=e$ and $\g'(0)$ is orthogonal to $w$. Therefore, every $\psi \in \mbox{Stab}^+_e(X)$ satisfies $\psi
(\G )=\G $. If Stab$^+_e(X)= \mbox{Stab}_e(X)$, then item~2 of Proposition~\ref{prop:mass} holds.
Otherwise,  Stab$_e(X)-\mbox{Stab}^+_e(X)\neq \mbox{\O }$ and in this case, item~2 of
Proposition~2.24 of~\cite{mpe11} ensures that $X$ is isomorphic and homothetic to $\Hip ^2\times \R $
endowed with its standard product metric (hence $w$ is a nonzero vertical vector and $\G $ is the vertical
line passing through $e$), and every $\psi \in \mbox{Stab}_e(X)-\mbox{Stab}^+_e(X)$
is either a reflectional symmetry with respect to a vertical plane or it is the composition of a reflectional symmetry
with respect to $\Hip ^2\times \{ 0\} $ with a rotation about $\G$. Thus, $\psi $ leaves $\G $ invariant
and item~2 of Proposition~\ref{prop:mass} holds. Note that the moreover part of the proposition cannot hold in
this case of $I(X)$ being four-dimensional, because given any 1-parameter subgroup of $X$ different from $\G $,
there exists $\psi \in B$ such that $\psi (\G )\neq \G $.

Next suppose that $I(X)$ has dimension 3. First suppose that
the underlying Lie group structure $Y$ of $X$ is not
unimodular\footnote{A Lie group with Lie algebra $\mathfrak{g}$
is called unimodular if for all $V\in \mathfrak{g}$, the Lie algebra endomorphism
ad$_V\colon \mathfrak{g}\to \mathfrak{g}$ given by ad$_V(W)=[V,W]$ has trace zero.}.
Hence $Y$ is isomorphic to some semidirect product $\R^2\rtimes _A\R$ for some matrix $A\in {\cal M}_2(\R)$
with trace$(A)\neq 0$ and $\det(A)\neq0$ since the dimension of $I(X)=3$,
see Section~\ref{subsecsemdirprod} for this notion of nonunimodular metric Lie group.
Then~\cite[item~4 of Proposition~2.21]{mpe11} gives
that $\mbox{Stab}^+_e(X)=\{ 1_X,\phi \} $ where $\phi (x,y,z)=(-x,-y,z)$. Therefore, item~1 of
Proposition~\ref{prop:mass} holds with the choice $\G =\{ (0,0,z)\ | \ z\in \R\} $ and the above $\phi$.
 Regarding item~2, we divide the argument into two cases in this nonunimodular case for $Y$.
Given $b\in \R-\{0\}$, consider the matrix
\begin{equation}
\label{ASol}
A(b)=\left( \begin{array}{cc}
1 & 0\\
0 & b
\end{array}\right) .
\end{equation}
\begin{enumerate}[({A}1)]
\item If $X$ is not isomorphic and homothetic to any $\R^2\rtimes _{A(b)}\R $ with $A$ given by (\ref{ASol}) with
$b\neq -1$, then~\cite[item~3 of Proposition~2.24]{mpe11} gives that
 $ \mbox{Stab}_e(X)=\mbox{Stab}^+_e(X)$.
 \item   If $X$ is isomorphic and homothetic to $\R^2\rtimes _{A(b)}\R $
for some $b\neq -1,0$, then~\cite[item~(3a) of Proposition~2.24]{mpe11} implies that
then $\mbox{Stab}_e(X)-\mbox{Stab}^+_e(X)=\{ \psi _1,\psi _2\}$ where $\psi _1(x,y,z)=(-x,y,z)$ and $\psi _2(x,y,z)=(x,-y,z)$.
In the case that $b=0$, then $\det(A(b))=0$, which is not the case presently under consideration as the isometry group
of $X$ for $b=0$ has dimension four.
\end{enumerate}
In both cases (A1) and (A2), item~2 of Proposition~\ref{prop:mass}  holds with $\G $ being the
$z$-axis. The moreover part of Proposition~\ref{prop:mass} also holds, because item~1 only holds for
$\G=\{ (0,0,z)\ | \ z\in \R \} $.

If dim$(I(X))=3$ and $Y$ is unimodular,
then there exists a frame of left invariant vector fields $E_1,E_2,E_3$ on $Y$
that are eigenfields of the Ricci tensor of $X$, independently of the
left invariant metric on $X$, see~\cite[Section~2.6]{mpe11}.
By the proof of Proposition~2.21 of~\cite{mpe11}, for each $i=1,2,3$
there exists an orientation-preserving automorphism $\phi_i\colon X\to X$ whose differential
satisfies $(\phi_i)_*(E_i)=E_i$ and $(\phi_i)_*(E_j)=-E_j$ whenever $j\neq i$. Furthermore,
$\phi $ is an isometry of every left invariant metric on $Y$ (in particular, of the metric of $X$).
Also, item~1 of Proposition~2.24 in~\cite{mpe11} gives that
$\mbox{Stab}_e^+(X)$ is the dihedral group $\{ 1_X,\phi _1,\phi _2,\phi _3\}$.
Therefore, item~1 of Proposition~\ref{prop:mass} holds with any of the choices
$\phi =\phi _i$ and $\G =\G_i$, for $i=1,2,3$,
where $\G_i$ is the $1$-parameter subgroup of $X$ generated by $(E_i)_e$. Regarding item~2, we again
divide the argument into two cases in this unimodular case for $Y$.
\begin{enumerate}[({A}1)']
\item If $X$ is not isomorphic and homothetic to $\R^2\rtimes _{A(-1)}\R $
(recall that $\R^2\rtimes _{A(-1)}\R $ is Sol$_3$ with its standard metric),
then~\cite[item~3 of Proposition~2.24]{mpe11} ensures that $ \mbox{Stab}_e(X)=\mbox{Stab}^+_e(X)$.
Thus item~2 of Proposition~\ref{prop:mass} holds for every choice of the form
$\G=\G_i$ and $\phi=\phi _i$ with $i=1,2,3$. Once here, and since it is immediate that $\phi_i(\Gamma_j)=\Gamma_j$
for any $i,j\in \{1,2,3\}$, we conclude that the moreover part of Proposition~\ref{prop:mass} holds in this case.
\item
If $X$ is isomorphic and homothetic to $\R^2\rtimes _{A(-1)}\R $,
then~\cite[item~3b of Proposition~2.24]{mpe11} gives that
$\mbox{Stab}_e(X)-\mbox{Stab}^+_e(X)=\{ \psi _1,\psi _2,\psi _3,\psi _4\} $, where
\[
\begin{array}{ll}
\psi _1(x,y,z)=(-x,y,z), & \psi _2(x,y,z)=(x,-y,z),\\
\psi _3(x,y,z)=(y,-x,-z), & \psi _4(x,y,z)=(-y,x,-z).
\end{array}
\]
Therefore, item~2 of Proposition~\ref{prop:mass}
only holds for the choice $\G=\{ (0,0,z) \ | \ z\in \R \} $ and $\phi(x,y,z)=(-x,-y,z)$
(in particular, the moreover part of Proposition~\ref{prop:mass} also holds).

\end{enumerate}
Now the proof is complete.
\end{proof}

We are now ready to prove the existence of the center of symmetry of any $H$-sphere in $X$.
Let $f_H\colon S_H\la X$ denote an $H$-sphere, with left invariant Gauss map $G
\colon S_H\to \esf^2$. By item~3 of Proposition~\ref{propsu2} and $h_0(X)=H(X)$, 
we know $G$ is an orientation preserving diffeomorphism. Without loss of generality we can assume that
$f_H(q_H)= e$, where $q_H\in S_H$ is the unique point such that $G(q_H)=v:=\Gamma'(0)$
and $\G \colon \R \to X$ is an arc length parameterization of the 1-parameter subgroup of $X$ that
appears in Proposition~\ref{prop:mass}.
Let $\phi \in \mbox{Stab}_e^+(X)$ be an isometry associated to $\G $ in the conditions of item~1
of Proposition~\ref{prop:mass}. Clearly $d\phi _e(v)=v$, hence $f_H(S_H)$
and $(\phi\circ f_H)(S_H)$ are two $H$-spheres in $X$ passing through $e$ with the same left invariant
Gauss map image at $e$. By the uniqueness of $H$-spheres up to left translations in item~2 of
Proposition~\ref{propsu2}, we conclude that $f_H(S_H)=(\phi\circ f_H)(S_H)$.
Let $q_H^*$ be the unique point in $S_H$ such that $G(q_H^*)=-v$.
Since $d\phi_e (-v)=-v$, we see that the Gauss map image of $\phi\circ f_H\colon S_H\la X$ at
$q_H^*$ is also $-v$. Since $f_H(S_H)=(\phi\circ f_H)(S_H)$, this implies
(by uniqueness of the point $q_H^*$) that $f_H(q_H^*)$ is a fixed point of $\phi$,
which by item 1 of Proposition~\ref{prop:mass} shows that $f_H(q_H^*)\in \Gamma$.

Let $x_H\in \Gamma$ be the midpoint of the connected arc of $\Gamma$ that connects $e$
with $f_H(q_H^*)$, and define $\hat{f}_H:=l_{x_H^{-1}}\circ f_H\colon S_H\la X$.
As the family $\{f_H\colon S_H\la X : H>H(X)\}$ is real analytic with respect to $H$ (by items 2 and 4
of Proposition~\ref{propsu2}), then the family $\{\hat{f}_H\colon S_H\la  X \ | \ H>H(X)\}$
is also real analytic in terms of $H$. Also, note that the points $p_H:=\hat{f}_H(q_H)$ and $p_H^*:=\hat{f}_H(q_H^*)$
both lie in $\Gamma$, and are equidistant from $e$ along $\Gamma$.
This property and item~2 of Proposition~\ref{prop:mass} imply that every
$\psi\in \mbox{Stab}_e (X)$ satisfies $\psi(\{ p_H,p_H^*\})=\{ p_H,p_H^*\} $, and hence either
$\psi (p_H)=p_H$ and $\psi(p_H^*)=p_H^*$, or alternatively $\psi(p_H)=p_H^*$ and $\psi(p_H^*)=p_H$.

Take $\psi \in \mbox{Stab}_e(X)$. We claim that $\psi$ leaves $\hat{f}_H(S_H)$ invariant.
To see this, we will distinguish four cases.
\begin{enumerate}[({B}1)]
\item
Suppose that $\psi$ preserves orientation on $X$ and satisfies $\psi(p_H)=p_H$, $\psi(p_H^*)=p_H^*$.
The first condition implies that $(\psi\circ \hat{f}_H)(S_H)$ is an $H$-sphere
in $X$.  So, by items 2 and 3 of Proposition~\ref{propsu2} we have  $l_a(\hat{f}_H(S_H))=(\psi\circ \hat{f}_H)(S_H)$
for some $a\in X$, and that $l_a(p_H)$ corresponds to the unique point
of $(\psi\circ \hat{f}_H)(S_H)$ where its Gauss map takes the value $v$. But now,
from the second condition we obtain that $\psi $ restricts to $\G $ as the identity map,
which implies that the value of the left invariant Gauss map image of $(\psi\circ \hat{f}_H)(S_H)$ at $p_H$ is $v$.
Thus, $l_a(p_H)=p_H$, hence $a=e$ and $\hat{f}_H(S_H)=(\psi\circ \hat{f}_H)(S_H)$,
that is, $\psi$ leaves $\hat{f}_H(S_H)$ invariant.
\item
Suppose that $\psi$ reverses orientation on $X$ and satisfies $\psi(p_H)=p_H^*$, $\psi(p_H^*)=p_H$.
The first condition implies that after changing the orientation of
$(\psi\circ \hat{f}_H)(S_H)$, this last surface is an $H$-sphere in $X$.
Since by the second condition $\psi $ restricts to
$\G $ as minus the identity map, then the left invariant Gauss map image of
$(\psi\circ \hat{f}_H)(S_H)$ (with the reversed orientation) at $p_H$ is $v$. Now we
deduce as in case (B1) that $\psi$ leaves $\hat{f}_H(S_H)$ invariant.
\item
Suppose that $\psi$ reverses orientation on $X$ and satisfies $\psi(p_H)=p_H$,
 $\psi(p_H^*)=p_H^*$. As in case (B2), we change the orientation on
$(\psi\circ \hat{f}_H)(S_H)$ so it becomes an $H$-sphere in $X$. Item~2 of
Proposition~\ref{propsu2} then gives that there exists $a\in X$ such that
$(\psi\circ \hat{f}_H)(S_H)=l_a(\hat{f}_H(S_H))$.
As $p_H^*\in \hat{f}_H(S_H)\cap \G$ and $\psi $ fixes $\G$ pointwise, then
$p_H^*\in (\psi\circ \hat{f}_H)(S_H)$. Moreover, the value of the left invariant Gauss map
image of $(\psi\circ \hat{f}_H)(S_H)$ (with the reversed orientation) at $p_H^*$ is $v$.
Since the only point of $l_a(\hat{f}_H(S_H))$ where its left invariant Gauss map image
(with the original orientation) takes the value $v$ is $l_a(p_H)$, we deduce that
$l_a(p_H)=p_H^*$. Similarly, $p_H=l_a(p_H^*)$. This is only possible if $a=a^{-1}$,
which, in our setting that $X$ is diffeomorphic to $\R^3$, implies that $a=e$
(note that this does not happen when $X$ is diffeomorphic to the three-sphere, because
then $X$ is isomorphic to $\su$, which has one element of order 2). The claim then
holds in this case.
\item
Finally, assume that $\psi$ preserves the orientation on $X$ and satisfies $\psi(p_H)=p_H^*$,
 $\psi(p_H^*)=p_H$. As in case (B1), $(\psi\circ \hat{f}_H)(S_H)$ is
 an $H$-sphere in $X$ and so there exists $a\in X$ such that $(\psi\circ \hat{f}_H)(S_H)=l_a(\hat{f}_H(S_H))$.
As $p_H\in \hat{f}_H(S_H)$, then $p_H^*=\psi(p_H)\in (\psi\circ \hat{f}_H)(S_H)$.
Since $\psi $ restricts
to $\G $ as minus the identity map, then the value of the left invariant Gauss map
image of $(\psi\circ \hat{f}_H)(S_H)$  at $p_H^*$ is $v$.  Since the only
point of $l_a(\hat{f}_H(S_H))$ where its left invariant Gauss map image
takes the value $v$ is $l_a(p_H)$, we deduce that $l_a(p_H)=p_H^*$. Now we finish as in case (B3).
\end{enumerate}

To sum up, we have proved that any $\psi\in \mbox{Stab}_e(X)$ leaves $\hat{f}_H(S_H)$ invariant;
this proves that the $H$-sphere $\hat{f}_H(S_H)$ is maximally symmetric with respect to the
identity element $e$, in the sense of item~1 of Theorem~\ref{main2}.
So, we may define $e$ to be the \emph{center of symmetry} of $\hat{f}_H(S_H)$.
We next show that the above definition of center of symmetry 
does not depend on the choice of the
 1-parameter subgroup $\G $ of $X$ that appears in Proposition~\ref{prop:mass}. Suppose
 $\wh{\G}$ is another $1$-parameter subgroup of $X$ satisfying the conclusions of
Proposition~\ref{prop:mass}.
As we explained in the paragraph just after  Proposition~\ref{prop:mass},
 $\wh{\G}$ determines two points $\wh{q}, \wh{q}^*\in S_H$ with $\wh{f}_H(\wh{q}),\wh{f}_H(\wh{q}^*)\in \wh{\G }$,
 such that $G(\wh{q})=\wh{\G }'(0)$, $G(\wh{q}^*)=-\wh{\G }'(0)$, where $G$ is the
 left invariant Gauss map of $\hat{f}_H$. To show that the above definition of center of symmetry
 of $\hat{f}_H(S_H)$ does
 not depend on $\G $, we must prove that  the mid point of the arc of $\wh{\G}$ between
 $\hat{f}_H(\wh{q})$ an  $\hat{f}_H(\wh{q}^*)$ is $e$. To do this,
let $\phi \in \mbox{\rm Stab}^+_e(X)$ be an order-two isometry
 that satisfies item~1 of Proposition~\ref{prop:mass} with respect to $\G $. By Proposition~\ref{prop:mass},
 $\phi(\hat{\G})= \hat{\G}$.
 As  $\phi $ leaves $\hat{f}_H(S_H)$ invariant and $G$ is a diffeomorphism, then
 $\phi (\{ \hat{f}_H(\wh{q}),  \hat{f}_H(\wh{q}^*) \} )=\{ \hat{f}_H(\wh{q}),\hat{f}_H
 (\wh{q}^*)\} $. Since the set of fixed points of $\phi $ is $\G $ and $\wh{f}_H(\wh{q}),\wh{f}_H(\wh{q}^*)\notin \G $
 (because two different 1-parameter subgroups in $X$ only intersect at $e$, here we are
 using again that $X$ is diffeomorphic to $\R^3$), then $\phi (\hat{f}_H(\wh{q}))=
 \hat{f}_H(\wh{q}^*)$. As $\phi$ is an isometry with $\phi(e)=e$
 and $\phi(\hat{\Gamma})=\hat{\Gamma}$, we deduce that the lengths of the
 arcs of $\hat{\G}$ that join $e$ to $\hat{f}_H(\wh{q})$,
 and $e$ to $\hat{f}_H(\wh{q}^*)$, coincide. This implies that the mid point of the arc of $\wh{\G}$
 between $\hat{f}_H(\wh{q}))$ an $\hat{f}_H(\wh{q}^*)$ is $e$, as desired.
As every $H$-sphere in $X$ is a left translation of $\hat{f}_H(S_H)$,
we conclude that all $H$-spheres in $X$ are maximally symmetric
with respect to some point $p\in X$, which is obtained by the
corresponding left translation of the identity element $e$,
and thus can be defined as the center of symmetry 
of the sphere. This completes the proof of Theorem~\ref{main2} (assuming that
Theorem~\ref{topR3case} holds) for the case that $X$ is diffeomorphic to $\R^3$.
Thus, the proof of Assertion~\ref{ass31} is complete.

The next remark gives a direct definition of the center of symmetry of any $H$-sphere in $X$:

\begin{remark}[Definition of the center of symmetry] \label{rem:center}{\em
Let $f\colon S_H\la X$ be an $H$-sphere in a homogeneous manifold $X$
diffeomorphic to $\R^3$ with an isometry group of dimension three or four.
Let $q,q^*\in S_H$ be given by $G(q)=v$, $G(q^*)=-v$,
where $G\colon S_H\flecha \S^2$ denotes the left invariant Gauss map of $S_H$
and $v:=\Gamma'(0)$, with $\Gamma(t)$ a $1$-parameter subgroup satisfying the conditions
in Proposition~\ref{prop:mass}.
Then, by our previous discussion, the left coset $\alfa:=l_{f(q)} (\Gamma)$ passes through both $f(q)$ and $f(q^*)$,
and we define the \emph{center of symmetry} of $S_H$ as the
midpoint $p\in \alfa$ of the subarc of $\alfa$ that joins $f(q)$ with $f(q^*)$.

It is worth mentioning that this definition gives a nonambiguous definition of the center of symmetry
of an $H$-sphere; however, when $X$ has an isometry group of dimension three and is nonunimodular,
there are many points $p\in X$ besides this center of symmetry such that the $H$-sphere
is invariant under all isometries of $X$ that fix $p$.  Nevertheless,
in order to make sense of the analyticity properties of the
family of constant mean curvature spheres with a fixed center of symmetry (see Theorem~\ref{main2}),
we need to make the definition of center of symmetry nonambiguous.}
\end{remark}

\subsection{Proof of Assertion~\ref{ass32}}\label{sec3.2}
Let $M$ be  a homogeneous
three-manifold with universal covering space $\Pi\colon X\to M$.
To prove Assertion~\ref{ass32}, {\bf we will assume in this Section~\ref{sec3.2}
that Theorems~\ref{main} and \ref{main2} hold for $X$}.

We first prove Theorem~\ref{main} in $M$. Since every
constant mean curvature  sphere  $f\colon S \la M$ is the projection via $\Pi$ of some lift
$\wt{f}\colon S\la X$ of $f$ with the same mean curvature, then clearly
the first and second items in Theorem~\ref{main} hold in $M$, since they are true by hypothesis in $X$.
In order to prove the uniqueness statement in Theorem~\ref{main}, let $f_1\colon S_1\la M$,
$f_2\colon S_2\la M$ be two spheres with the same constant mean curvature in $M$. For $i=1,2$,
choose $\wt{f}_i$ respective lifts of $f_i$, let $\wt{p}_i$ be the centers of symmetry of $\wt{f}_i$,
denote $p_i=\Pi(\wt{p}_i)$, and let  $I\colon M\to M$ be an isometry with $I(p_1)=p_2$. Let $\wt{I}\colon X\to X$
be the lift of $I$ that takes $\wt{p}_1$ to $\wt{p}_2$. It follows that the $H$-spheres $\wt{I}\circ \wt{f}_1$
and $\wt{f}_2$ have the same center of symmetry in $X$ and so these immersions have the same images. In particular,
it follows that $I\circ f_1$ and $f_2$ also have the same images, which completes the proof of Theorem~\ref{main} in $M$.

We next prove Theorem~\ref{main2} in $M$. Suppose $f\colon S \la M$ is an oriented $H$-sphere, and
let  $\wt{f}\colon S\la X$ denote some lift of $f$. Since the stability operators of $f$ and $\wt{f}$
are the same, the index and nullity of $f$ and $\wt{f}$ agree. Also, note that if the
immersion $\wt{f}$ extends to an isometric immersion $F\colon B\flecha X$ of a mean convex Riemannian
three-ball $B$ into $X$, then $\Pi\circ F$ is an isometric immersion of $B$ into $M$ that extends $f$.
These two trivial observations prove that item 2 of Theorem~\ref{main2} holds in $M$.

If $M$ is covered by $\esf^2(\kappa)\times \R$ and $f(S)$ is totally geodesic,
then any point of the image sphere satisfies the property of being a center of symmetry of $f$.
This observation follows by the classification of homogeneous three-manifolds
covered by $\esf^2(\kappa)\times \R$ as being those spaces
isometric to  $\esf^2(\kappa)\times \R$,
$\esf^2(\kappa)\times \esf^1(R)$, $\mathbb{P}^2(\kappa)\times \R$ or $\mathbb{P}^2(\kappa)\times\esf^1(R)$,
where $\esf^1(R)$ is a circle of circumference $R$ for some $R>0$ and $\mathbb{P}^2(\kappa) $ is the projective plane of
constant Gaussian curvature $\kappa>0$.

Suppose now that
$M$ is not covered by $\esf^2(\kappa)\times \R$ with $f(S)$ being totally geodesic.
In order to complete the proof of item~1 of Theorem~\ref{main2}, we next define the center of symmetry of $f$ in $M$
as $p:=\Pi(\wt{p})$, where $\wt{p}$ is the center of symmetry of a lift $\wt{f}$ in $X$.
Since the center of symmetry of the lift $\wt{f}$
is uniquely defined by $\wt{f}$ and any other lift is equal to the composition of $\wt{f}$ with
an isometry $\sigma \colon X\to X$ that is a covering transformation, then by uniqueness of the center of
symmetry in $X$,
$\sigma$ maps the center of symmetry of $\wt{f} $  to the center of symmetry of the composed oriented
immersion $\sigma \circ \wt{f}$; hence $p$ is independent of the choice of the lift of $f$ to $X$.
We now show  that $f(S)$ is
invariant under all isometries of $M$ that fix $p$. Let $I\colon M\to M$
be any such isometry, and let $\wt{I}\colon X\to X$
be the lift of $I$ such that $\wt{I}(\wt{p})=\wt{p}$. As $\wt{p}$ is a center of symmetry of
$\wt{f}$, then, $\wt{I}$ induces an isometry of $\wt{f}\colon S\la X$.
It follows that $I$ induces an isometry of the $H$-sphere $f$, which proves item~1 of Theorem~\ref{main2}
in $M$.
Once here, the homeomorphism and analyticity properties in the last statement of Theorem~\ref{main2} follow trivially
from the validity of the corresponding statement in $X$. This completes the proof of Theorem~\ref{main2} in $M$,
and thus also the proof of Assertion~\ref{ass32}.

\subsection{Constant mean curvature spheres in complete locally homogeneous three-manifolds}
\label{sec3.3}
As a consequence of Theorems~\ref{main} and \ref{main2} and of the discussion in 
Sections \ref{sec3.1} and \ref{sec3.2} we can provide a description of the space of constant 
mean curvature spheres in complete, locally homogeneous three-manifolds. Specifically, 
let $Y$ be a complete, connected, locally homogeneous Riemannian three-manifold. 
Then, the universal Riemannian covering space $\Pi\colon X\to Y$ is a simply connected homogeneous three-manifold. 
Let ${\rm Ch}(X)$ be the Cheeger constant of $X$. Remark~\ref{rem:center},
the proof of Proposition~\ref{prop:mass} and the results on the existence and uniqueness of
the center of symmetry for any $H$-sphere in $X$
imply that associated to any $H$-sphere $S_{H} $ in $Y$ with $H> 0$,
there is a unique point $p({S_{H}})\in Y$ that is the image by $\Pi$
of the center of symmetry of any lift of $S_{H}$ to $X$; that $p(S_H)$ does not depend on the lift of $S_H$ 
can be proved following the arguments used in Section \ref{sec3.2}.
The following properties of spheres with nonzero constant mean 
curvature in $Y$ are then easy to check, and provide a parameterization of the space of such constant mean curvature spheres.
\begin{enumerate}
\item
For any $H>{\rm Ch}(X)/2$ and any $p_0\in Y$ there exists a unique $H$-sphere $S_H$ in $Y$ with $p({S_{H}})=p_0$.
 \item
If ${\rm Ch}(X)>0$, there are no $H$-spheres in $Y$ for $0<H\leq {\rm Ch}(X)/2$.
\item The space $\cM_H(Y)$ of all immersed $H$-spheres in $Y$ for a given $H>{\rm Ch}(X)/2$
is naturally diffeomorphic to $Y$ (by the map that sends every $S_{H}$ into $p(S_{H})$), 
where $H$-spheres in $Y$ with distinct images are considered to be different elements in $\cM_H(Y)$.
 \item
Any $H$-sphere in $Y$ for $H>0$ has index one and nullity three for its stability operator.
 \item
Any $H$-sphere $S_H$ in $Y$ for $H>0$ is invariant under every isometry of $Y$ that fixes $p(S_H)$.
\item Given any two $H$-spheres $S^0_H,S^1_H\in \cM_H(Y)$ for some $H>0$, there exist an isometry 
$\phi\colon  S^0_H\flecha S^1_H$ that also preserves the second fundamental forms of $S^0_H$ and $S^1_H$ 
at corresponding points. Moreover, given any path $\a\colon [0,1]\to Y$ joining
$p({S^0_H})$ to $p({S^1_H})$, there is an associated family of isometric $H$-spheres $S_H({t})$, $t\in [0,1]$, in $Y$
with  $S^0_H =S_H(0)$ and $S^1_H=S_H(1)$, such that $p(S_H({t}))=\a(t)$ for every $t$.
\end{enumerate}

We observe that item 6 above is the natural generalization to complete locally homogeneous 
three-manifolds of the uniqueness result in Theorem \ref{main} that any two spheres of the 
same constant mean curvature in a homogeneous three-manifold differ by an ambient isometry.

\section{Background on constant mean curvature surface theory in metric Lie groups}
\label{sec:gauss}

In this section we will collect some basic material about metric Lie groups diffeomorphic to
$\R^3$ and immersed surfaces of constant mean curvature in these spaces, that will be needed
in later sections. For the purposes of this paper, it is worth dividing
metric Lie groups $X$ diffeomorphic to $\R^3$ into two classes
depending on whether the underlying group structure of $X$ is that of a
semidirect product, or it is the one of $\sl $, that can be naturally identified as
the universal cover of the group $\psl$ of orientation-preserving isometries of the hyperbolic plane.

We  divide this section into three parts: in Section~\ref{backgroundCMC}
we describe some geometric aspects of constant mean curvature surfaces in a metric
Lie group~$X$ that are invariant under the flow of a nonzero right invariant vector field on $X$.
Sections~\ref{subsecsemdirprod} and~\ref{subsecsl}
describe the ambient geometry and particular aspects of surface theory in a metric Lie group
isomorphic to a semidirect product or to $\sl $, respectively. We will use the contents of Section~\ref{backgroundCMC} in
Section~\ref{sec:limits}, but Sections~\ref{subsecsemdirprod} and~\ref{subsecsl} will not be needed
until Section~\ref{sec:periodic}.
The basic references for this material are Milnor~\cite{mil2}, the notes~\cite{mpe11} and the paper~\cite{mmpr4}.

\subsection{Invariant constant mean curvature surfaces}
\label{backgroundCMC}
Given a unit vector $v\in T_e X$ in the Lie algebra of $X$, we can extend $v$ by left translation to a
left-invariant unit vector field $V$ globally defined on $X$. By using right translations in $X$ in a
similar way, we can extend $v$ to a right-invariant vector field $K$ in $X$. The vector field $K$ will
not be, in general, of unit length anymore, but it is a  nowhere-zero Killing field on $X$.

By~\cite[Corollary~3.17]{mpe11},  left cosets of two-dimensional subgroups
of $X$ are characterized by their left invariant Gauss map
being constant, and they all have constant mean
curvature with absolute value at most $H(X)$.
As every right coset $\Delta x$ of a two-dimensional subgroup $\Delta $ of $X$ is the left coset of
a conjugate subgroup of $\Delta $ (namely $\Delta x=x\Delta _1$ where $\Delta _1=
x^{-1}\Delta x$), then we deduce:
\begin{enumerate}[(C)]
\item The left invariant
Gauss map of a left or right coset of a two-dimensional subgroup of
$X$ is constant.
\end{enumerate}

The next proposition is a reformulation of Corollary~3.8 in~\cite{mmpr4}.
We remark that Corollary~3.8 in~\cite{mmpr4} was written in terms of the $H$-potential of $X$, a
concept that we will not introduce here, but one can easily translate that formulation
into the one below by using~\cite[Corollary~3.21]{mpe11}. 

\begin{proposition}
\label{propos3.1}
Let $f\colon \Sigma \looparrowright X $ be a complete immersed $H$-surface in $X$
and let $G\colon \Sigma \to \esf^2\subset T_eX$ be its left invariant Gauss map.
Assume that there are no two-dimensional subgroups in $X$ with mean curvature $H$ whose
(constant) left invariant Gauss map lies in the Gauss map image $G(\Sigma)$. Then:
\begin{enumerate}[(1)]
\item The differential $dG$ of $G$ has rank at most 1 everywhere on $\Sigma $ if and only if $f$ is invariant
under the flow of a nonzero right invariant vector field $K$ on $X$.
\item If $f$ is invariant under the flow of a nonzero right invariant
vector field of $X$, then rank$(dG)=~1$ everywhere on
$\Sigma $, and $G(\Sigma)$ is a regular curve in $\S^2$.
\end{enumerate}
\end{proposition}
\begin{remark}
\label{rem7.1}
{\em If $K$ is a nonzero right invariant vector field on a metric Lie group $X$
diffeomorphic to $\R^3$,
 then
each integral curve of $K$ is diffeomorphic to $\R$ and the quotient space $X/K$ of
integral curves of $K$ inherits a two-dimensional differentiable structure that
makes the natural projection $\Pi _K\colon X\to X/K$ a submersion with $\mbox{\rm ker}
[d(\Pi_K)_x]=\mbox{Span}\{ K(x)\} $ for all $x\in X$. 
Therefore, when $f\colon \Sigma \looparrowright X $ satisfies the hypothesis in item~2 of
Proposition~\ref{propos3.1}, we can consider $f(\Sigma )$ as the surface in $X$ obtained
by pulling back via $\Pi _{K_{\Sigma }}$ an immersed curve $\be $ contained in
$X/K_{\Sigma }$.
Moreover, after identifying in the standard way $X/K_{\Sigma}$ with any entire Killing graph
$S_0\subset X$ with respect to $K_{\Sigma}$, the invariant surface $f\colon\Sigma\looparrowright X$
can be parameterized locally (and even globally if $\Sigma$ is simply connected) as
\[
f(s,t)=l_{\Gamma(t)} (\beta(s)), \hspace{1cm} (s,t)\in I\times \R,
\]
where $\beta(s)\colon I\subset \R\flecha X/K_{\Sigma}\equiv S_0$ is the immersed curve above,
and $\Gamma=\Gamma(t)$ is the $1$-parameter subgroup of $X$ given by $\Gamma'(0)=K_{\Sigma}(e)$.}
\end{remark}

Let $f\colon\Sigma\la X$ be an immersed $H$-surface in $X$ whose left invariant Gauss map $G\colon\Sigma\flecha\S^2$
has rank one at every point. It is shown in~\cite[proof of Corollary 3.8]{mmpr4} that around any $z_0\in \Sigma$
there exist conformal parameters $(s,t)$ on $\Sigma$ such that $G$ does not depend on $t$, so we can write $G=G(s)$.
With respect to these coordinates, the Gauss map $G(s)$ satisfies a second order autonomous ODE of the form
\begin{equation}
\label{gaussecu}
 G''=\Psi (G,G'),
 \end{equation}
where $\Psi\colon\S^2\times \R^3\flecha \R^3$ is real analytic; this follows directly from
the ODE (4.3) in~\cite{mmpr4}. The special form of this ODE has the following trivial consequence:

\begin{enumerate}[(D)]
\item
If $\hat{g}(s)$ is a solution of the ODE (4.3) in~\cite{mmpr4},
then $\hat{g}(\delta_1 s+ \delta_2)$, $\delta_1\neq 0$, is also a solution of the same ODE.
Consequently, the same property holds for a solution $G(s)$ of \eqref{gaussecu}.
\end{enumerate}

\begin{lemma}\label{tangin}
Let $f_i\colon\Sigma_i\la X$, $i=1,2$ denote two complete immersed $H$-surfaces in $X$ whose Gauss map images
$\gamma_i:=G_i(\Sigma_i)$ are regular curves in $\S^2$. Assume that there exist points $p_i\in \Sigma_i$
with $f_1(p_1)=f_2(p_2)$  and $G_1(p_1)=G_2(p_2)=:v\in \S^2$, so that $\gamma_1,\gamma_2$ intersect tangentially at
$v$. Then $f_1(\Sigma_1)=f_2(\Sigma_2)$.
\end{lemma}
\begin{proof}
By the previous comments, we can view $f_1,f_2$ locally around $p_1,p_2$ as conformal immersions
$f_i(s,t)$ into $X$, and $G_1,G_2$ as regular parameterized curves $G_i(s)\colon J_i\flecha \S^2$,
both of them satisfying the ODE \eqref{gaussecu} on an open interval  $J_i\subset \R$.
The conditions in the statement of the lemma imply that
$G_1(s_1)=G_2(s_2)$ and $(G_1)'(s_1)
=\pm (G_2)'(s_2)$ for some $s_i\in J_i$. But once here, property (D) above and the uniqueness of
solution to the Cauchy problem for ODEs imply
that $G_1(s)=G_2(\delta_1 s+ \delta_2)$ for adequate constants $\delta_1\neq 0$ and
$\delta_2$. In particular, the two Gauss maps coincide after conformal reparameterization of $f_1$ or $f_2$.
 Recall now that, by~\cite[Theorem 3.7]{mmpr4}, the Gauss map of a conformally parameterized
$H$-surface in $X$ determines the surface up to left translation in $X$. Since $f_1(p_1)=f_2(p_2)$, this left translation
is trivial in our case, and so we conclude that $f_1(\Sigma_1)=f_2(\Sigma_2)$.
\end{proof}

\subsection{Metric semidirect products}
\label{subsecsemdirprod}
Given a  real $2\times2$ matrix $A\in \cM_2(\R)$, the semidirect product
$\R^2\rtimes_A \R$ is the Lie group $(\R^3\equiv \R^2\times \R,*)$ endowed with
the group operation
$({\bf p}_1,z_1)*({\bf p}_2,z_2)=({\bf p}_1+ e^{z_1 A}\  {\bf
 p}_2,z_1+z_2);$
here $e^B=\sum _{k=0}^{\infty }\frac{1}{k!}B^k$ denotes the usual
exponentiation of a matrix $B\in \cM_2(\R )$. Let

\begin{equation}
\label{equationgenA}
A=\left(
\begin{array}{cr}
a & b \\
c & d \end{array}\right) ,\qquad
 e^{zA}=\left(
\begin{array}{cr}
a_{11}(z) & a_{12}(z) \\
a_{21}(z) & a_{22}(z)
\end{array}\right) .
\end{equation}
Then, a left invariant frame $\{ E_1,E_2,E_3\} $ of $X$ is given by
\begin{equation}
\label{eq:6*}
 E_1(x,y,z)=a_{11}(z)\partial _x+a_{21}(z)\partial _y,\quad
E_2(x,y,z)=a_{12}(z)\partial _x+a_{22}(z)\partial _y,\quad
 E_3=\partial _z.
\end{equation}
Observe that $\{ E_1,E_2,E_3\} $ is the left invariant extension of the
canonical basis $(\partial_x)_e$, $(\partial _y)_e$, $(\partial _z)_e$ of the tangent
space $T_eX$ at the identity element $e=(0,0,0)$.
The right invariant extensions on $X$ of the same vectors of $T_eX$ define the
frame $F_1,F_2,F_3$ where
\begin{equation}
\label{eq:6}
 F_1=\partial _x,\quad F_2=\partial _y,\quad F_3(x,y,z)=
(ax+by)\partial _x+(cx+dy)\partial _y+\partial _z.
\end{equation}
In terms of $A$, the Lie bracket relations are:
\[
\label{eq:8a} [E_1,E_2]=0, \quad [E_3,E_1]=aE_1+cE_2, \quad
 [E_3,E_2]=bE_1+dE_2.
\]

\begin{definition}
 \label{def2.1}
 {\rm
We define the {\it canonical left invariant metric} on $\R^2\rtimes _A\R $ to be that one for which the
left invariant frame $\{ E_1,E_2,E_3\} $ given by (\ref{eq:6*}) is
orthonormal. Equivalently, it is the left invariant extension to
$X=\R^2\rtimes _A\R $ of the inner product on
$T_eX$ that makes $(\partial _x)_e,(\partial _y)_e,(\partial _z)_e$ an orthonormal basis.
}
\end{definition}

Some basic properties of the canonical left
invariant metric $\langle ,\rangle $ on $\R^2\rtimes_A \R$ are the following ones.

\begin{enumerate}[(E1)]
\item The vector fields $F_1,F_2,F_3$ are Killing.
\item The mean curvature of each leaf of the foliation $\mathcal{F}=
\{ \R^2\rtimes _A\{ z\} \mid z\in \R \}$ with respect to the unit
normal vector field $E_3$ is the constant $H=\mbox{trace}(A)/2$. All
the leaves of the foliation $\mathcal{F}$ are intrinsically flat.
Moreover, the Gauss map image of each leaf of $\cF$ with respect to the orientation
given by $E_3$ (resp. $-E_3$) is the North (resp. South) pole of $\S^2$ with
respect to the left invariant frame $\{E_1,E_2,E_3\}$.
\item The change (\ref{eq:6*})  from the orthonormal basis $\{ E_1,E_2,E_3\} $ to the
basis $\{ \partial _x,\partial _y,\partial _z\} $
produces the following expression of $\langle ,\rangle $ in the $x,y,z$ coordinates of $X$:
\begin{equation}
\label{eq:13}
 \left.
\begin{array}{rcl}
\langle ,\rangle & =&
 \left[ a_{11}(-z)^2+a_{21}(-z)^2\right] dx^2+
\left[ a_{12}(-z)^2+a_{22}(-z)^2\right] dy^2 +dz^2 \\
& + & \rule{0cm}{.5cm} \left[
a_{11}(-z)a_{12}(-z)+a_{21}(-z)a_{22}(-z)\right] \left( dx\otimes
dy+dy\otimes dx\right) .
\end{array}
\right.
\end{equation}

\item The $\pi$-rotation about any of the integral curves of $\partial _z=E_3$
(vertical lines in the $(x,y,z)$-coordinates) is an
order-two orientation preserving isometry.
\end{enumerate}

\begin{lemma}
\label{lem:asin}
If $A$ is a singular matrix, then $X=\R^2\rtimes_A \R$ is isometric to $\R^3$ or to an
$\E(\kappa ,\tau)$-space
with $\kappa \leq 0$ and $\tau \in  \R$.
\end{lemma}
\begin{proof}
If $A=0$, then $X$ is the standard $\R^3$. If $A\neq 0$ and
trace$(A)=0$, then~\cite[Theorem~2.15]{mpe11} gives that either $\det(A)\neq 0$
(which contradicts our hypothesis) or $X$ is
isometric to the Heisenberg space Nil$_3$ (recall that Nil$_3$ admits a $1$-parameter family of homogeneous
metrics, all of which are homothetic).
Finally, if trace$(A)\neq0$ then $X$ is isometric to some
$\E (\kappa ,\tau )$ with $\kappa <0$ and $\tau \in \R $
as explained in~\cite[Section~2.8]{mpe11}.
\end{proof}

A simple consequence of equations (\ref{eq:6}) and (\ref{eq:13}) is that every
horizontal right invariant vector field (i.e., every linear combination of $\partial _x,\partial _y$
with constant coefficients) is bounded in $\R^2\rtimes _A[z_1,z_2]$, for all $z_1,z_2\in \R$
with $z_1\leq z_2$.
The following lemma guarantees the existence of at least one nonzero
horizontal right invariant vector field in $\R^2\rtimes _A\R$ which is bounded in a
horizontal halfspace.

\begin{lemma}
\label{cor:bounded}
Let $X$ be a semidirect product $\R^2\rtimes _A\R $ endowed with its
canonical metric, where $A\in \mathcal{M}_2(\R )$. 
Then, there exists a nonzero horizontal right
invariant vector field which is bounded in $\R^2\rtimes _A[0,\infty
)$.
\end{lemma}
\begin{proof}
If trace$(A)\neq 0$, the lemma follows from~\cite[item 2 of Proposition~6.1]{mmp1}. If
trace$(A)=0$, then~\cite[item 2 of Theorem 2.15]{mpe11} implies that
there are three possible cases for $A$ up to rescaling the metric. We will prove the lemma
by studying each case separately.
\begin{itemize}
\item $A=\left( \begin{array}{cc}
0 & -c \\
1/c & 0 \end{array}
\right) $ for some $1\leq c<\infty $. This case corresponds to the group structure of
$\widetilde{\rm E}(2)$, the universal cover of the group of orientation-preserving rigid motions of the
Euclidean plane. A direct computation using  (\ref{eq:6}) and (\ref{eq:13}) gives
$|F_1|^2=\cos ^2z+c^{-2}\sin ^2z$ and $|F_2|^2=\cos ^2z+c^2\sin ^2z$, both bounded.

\item $A=\left( \begin{array}{cc}
0 & c \\
1/c & 0  \end{array}
\right) $ for some $1\leq c<\infty $. In this case, $X$ is isomorphic to Sol$_3$.
A similar reasoning as in the preceding case gives that
the basis of horizontal right invariant vector fields $\widehat{F}_1=-c\partial _x+\partial _y$,
$\widehat{F}_2=c\partial_x+\partial _y$ satisfies
$|\widehat{F}_1|=\sqrt{1+c^2}e^z$, $|\wh{F}_2|=\sqrt{1+c^2}e^{-z}$.
\item $A=\left( \begin{array}{cc}
0 & 1 \\
0 & 0 \end{array}
\right) $ and $X$ is isomorphic to the Heisenberg space Nil$_3$. Using again
(\ref{eq:6}) and (\ref{eq:13}) we have that $|F_1|=1$.
\end{itemize}
\par
\vspace{-.7cm}
\end{proof}

\subsection{Metric Lie groups isomorphic to $\sl$}
\label{subsecsl}

The Lie group $\sl $ is the universal cover of the special linear group
$\mathrm{SL}(2,\R )=\{ A\in \mathcal{M}_2(\R )\ | \ \det A=1\} $,
and of the projective special linear group
$\mathrm{PSL}(2,\R)=\mathrm{SL}(2,\R )/\{ \pm I_2\} $.
The Lie algebra of any of the groups $\sl$, $\mathrm{SL}(2,\R )$, and $\mathrm{PSL}(2,\R)$ is
$\mathfrak{sl}(2,\R ) = \{ B\in \mathcal{M}_2(\R )\ | \ \mbox{trace}(B)=0\} $.
It is worth recalling that
$\mathrm{PSL}(2,\R)$ has several isomorphic models, that we will use in the sequel:
\begin{enumerate}[(F1)]
\item The group of orientation-preserving isometries of the hyperbolic plane,
\[
\mbox{\hspace{1cm}} z\in \Hip ^2\equiv (\R^2)^+ \mapsto
\frac{az+b}{cz+d} \in (\R^2)^+\qquad (a,b,c,d\in \R ,\ ad-bc=1).
\]
\item The group of conformal automorphisms of the unit disc
$\D:=\{z\in \C: |z|<1\}$; these automorphisms are the
M\"{o}bius transformations of the type
$\phi(z)=\xi \frac{z+a}{1+\overline{a} z},$ with $a\in \D$
and $|\xi|=1$. The universal cover map $\theta\mapsto \xi:=
e^{i\theta}\colon \R\flecha \S^1$ gives $\sl$ a structure of an
$\R$-bundle over $\D$.
\item The unit tangent bundle of $\Hip^2$.
This interpretation of PSL$(2,\R )$ as an $\esf^1$-bundle over $\Hip ^2$ (and of
$\sl$ as an $\R $-bundle over $\Hip ^2$) produces a submersion from
$\sl $ to $\Hip^2$.
\end{enumerate}
There are three types of 1-parameter subgroups of $\sl$:
\begin{itemize}
\item {\em Elliptic subgroups.} Elements of any of these subgroups correspond to liftings to
$\sl $ of rotations around
any fixed point in $\Hip ^2$. These groups of rotations
fix no points in the boundary at infinity
$\partial _{\infty }\Hip ^2$.
Any two elliptic 1-parameter subgroups are conjugate.

\item {\em Hyperbolic subgroups.} These subgroups correspond to liftings to $\sl $ of
translations along any fixed geodesic 
in $\Hip ^2$. In the Poincar\'e disk model of $\H^2$, the hyperbolic
subgroup associated to a geodesic $\G $ fixes the two points at
infinity given by the end points of $\G $.
As in the elliptic case, any two 1-parameter hyperbolic subgroups are conjugate.

\item {\em Parabolic subgroups.}
They correspond to liftings to $\sl $ of rotations about any fixed point
$\theta \in \partial_{\infty }\Hip ^2$. Any of these rotation subgroups only fixes the point $\theta $ at infinity,
and leaves invariant the 1-parameter family of horocycles based at
$\theta $. As in the previous cases, 1-parameter parabolic subgroups are all
conjugate by elliptic rotations of the Poincar\'e disk.
\end{itemize}

The \emph{character} of a $1$-parameter subgroup $\Gamma$ of
$\sl$ refers to the property of $\Gamma$ being elliptic,
parabolic or hyperbolic. The character is invariant by
conjugation, i.e., if $\Gamma$ is a $1$-parameter subgroup of
$\sl$ and we define $\Gamma_a := a \Gamma a^{-1} =
l_a (r_{a^{-1}} (\Gamma))$ for some $a\in \sl$, then $\Gamma_a$
has the same character as $\Gamma$.

Nonzero right invariant vector fields $K$ are in one-to-one correspondence with
1-parameter subgroups $\G $ and with tangent vectors at the identity element $e$ of $\sl$
by the formula $K_e=\G'(0)$. We will say that a nonzero right invariant vector field $K$ on $\sl $
is elliptic (resp. hyperbolic, parabolic) when the related 1-parameter subgroup $\G $ in $\sl $ is
elliptic (resp. hyperbolic, parabolic).

The matrices in the Lie algebra $\mathfrak{sl}(2,\R )$ given by
\begin{equation}
\label{eq:lframesl}
(E_1)_{e}=\left(
\begin{array}{rr}
1 & 0 \\
0 & -1\end{array}\right) ,\quad
(E_2)_{e}=\left( \begin{array}{cc}
0 & 1 \\
1 & 0\end{array}\right) ,\quad
(E_3)_{e}=\left(
\begin{array}{rr}
0 & -1 \\
1 & 0\end{array}\right)
\end{equation}
define a left invariant frame $\{ E_1,E_2,E_3\}$ on $\sl $ with the
property that
$ [E_1,E_2]=-2E_3$,
$[E_2,E_3]=2E_1$,
$[E_3,E_1]=2E_2$.
The fields $E_1,E_2,E_3$ are respectively generated by the
1-parameter subgroups of $\mathrm{SL}(2,\R )$ given by
\begin{equation}\label{1parsu}
\left(
\begin{array}{cc}
e^t & 0 \\
0 & e^{-t}\end{array}\right) ,
\quad
\left(
\begin{array}{cc}
\cosh t & \sinh t \\
\sinh t & \cosh t\end{array}\right) ,
\quad
\left(
\begin{array}{cc}
\cos t & -\sin t \\
\sin t & \cos t\end{array}\right) ,\quad t\in \R ,
\end{equation}
where the first two are hyperbolic subgroups and the third one
is an elliptic subgroup.

The center $Z$ of $\sl$ is an infinite cyclic subgroup (hence isomorphic to $\Z$) whose elements correspond to
liftings to $\sl $ of rotations by angles that are multiples of $2\pi $;
 if $\Lambda\colon \sl\flecha {\rm PSL}(2,\R)$ denotes the universal covering map,
then $Z=\Lambda^{-1}(e)$ where $e$ is the identity element of ${\rm PSL}(2,\R)$.
In this way, $Z$ is contained in the integral curve $\Gamma_3\subset \sl$
of the left invariant vector field $E_3$ given by \eqref{1parsu}
(indeed, $Z$ is contained in every elliptic subgroup of $\sl$).

By declaring the left invariant frame $\{
E_1,E_2,E_3\} $ to be orthonormal, we define a left invariant metric
$\langle ,\rangle $ on $\sl $ such that the metric Lie group
$(\widetilde{\mbox{\rm SL}}(2,\R ),\langle ,\rangle )$ is isometric
to the $\E (\kappa ,\tau )$-space with base curvature $\kappa =-4$
and bundle curvature $\tau ^2=1$. Let
\begin{equation}
\label{eq:Ekt}
\Pi_0\colon \E (-4,1)\to \Hip ^2(-4)
\end{equation}
 be the associated Riemannian
submersion onto the hyperbolic plane endowed with the metric of
constant curvature $-4$ (this submersion can be naturally identified with any
of the submersions described in properties (F2) and (F3) above).
Let $\phi \colon (\sl ,\langle ,
\rangle )\to \E (-4,1)$ be a Riemannian isometry.
Then, the image by $\phi $ of every integral curve of $E_3$ is a $\Pi _0$-vertical
geodesic. The composition
\begin{equation}
\label{eq:PI}
\Pi =\Pi _0\circ \phi \colon \sl \to \Hip ^2(-4)
\end{equation}
defines a submersion which is Riemannian with respect to the particular
left invariant metric $\langle ,\rangle $ defined above.

The next result follows from~\cite[Section~2.7]{mpe11} and describes the geometry of each 1-parameter subgroup of
$\sl$ in terms of the coordinates of
its velocity vector at the identity element $e$ with respect to the basis $\{
E_1,E_2,E_3\} $.

\begin{lemma}\label{caracsl}
Consider a vector
$\G'(0)\in T_e\sl $ with coordinates $(a,b,c)$ with respect to the basis
given in (\ref{eq:lframesl}), and let $\Gamma$ denote the $1$-parameter
subgroup generated by $\Gamma'(0)$. Then:
\begin{enumerate}
\item
$\Gamma$ is elliptic if and only if $a^2+b^2<c^2$. If $a=b=0$, then $\G$ is the lift to
$\widetilde{\mbox{\rm SL}}(2,\R ) $ of the elliptic subgroup of
rotations of $\Hip ^2(-4)$ around the point $\Pi (e)$.  If
$0<a^2+b^2<c^2$, then $\Pi (\G)$ is a constant geodesic curvature
circle passing through $\Pi (e)$ and completely contained in $\Hip
^2(-4)$.
 \item
$\Gamma$ is hyperbolic if and only if $a^2+b^2>c^2$. In this case, $\Pi (\G)$ is a
constant geodesic curvature arc passing through
$\Pi (e)$ with two end points in the boundary at infinity of $\Hip
^2(-4)$.
\item
$\Gamma$ is parabolic if and only if $a^2+b^2=c^2$. In this case, then $\Pi (\G)$ is a horocycle in $\Hip ^2(-4)$.
\end{enumerate}
\end{lemma}

Every left invariant metric in $\sl $
can be obtained by choosing numbers $\l _1,\l _2,\l _3>0$ and declaring
the length of the left invariant vector field $E_i$ to be $\l _i$, $1,2,3$,
while keeping them orthogonal; see~\cite[Proposition 4.2]{mmpr2}.
For the remainder of this section, we will suppose
that the metric Lie group $X$ is isometric and isomorphic to $\sl $ endowed with the
left invariant metric given by an (arbitrary but fixed) choice of $\l _i>0$, $i=1,2,3$.

\begin{lemma}
\label{lemma4.1}
Let $\g \subset \Hip ^2(-4)$ be a geodesic. Denote by $\Sigma (\g )=
\Pi ^{-1}(\g )$ the surface in $\sl $ obtained by pulling back $\g $ through $\Pi $.
Then:
\begin{enumerate}[(1)]
\item If \, $\G \subset \sl $ is an integral curve of $E_3$, then $\phi (\G )$
is an integral curve of the vertical vector field on $\E (-4,1)$ that
generates the kernel of the differential of $\Pi _0$.
Furthermore, if $R^{\G }\colon \E (-4,1)\to \E (-4,1)$ denotes the rotation of angle $\pi $ around
$\phi (\G )$ (i.e., the  lifting through $\Pi _0$ of the rotation $\psi _{\Pi (\G )}$ by angle $\pi $ around
the point $\Pi (\G )\in \Hip ^2(-4)$), then the map $R_{\G }=\phi^{-1}\circ R^{\G}\circ \phi$
 is an order-two, orientation-preserving isometry of $X$ that fixes $\G$.

\item $\Sigma (\g )$ is ruled by geodesics of $X$ which are integral curves
of $E_3$, and each  integral curve $\G $ of $E_3$ contained in $\Sigma (\g )$ is the fixed
point set of the isometry $R_{\G }\colon X\to X$, which restricts to an order-two, orientation-reversing
isometry on $\Sigma (\g)$.

\item $\Sigma (\g )$ is a stable minimal surface in $X$.

\item Given a curve $\g _r\subset \Hip ^2(-4)$ at constant distance $r>0$ from
$\g $, the surface $\Sigma (\g _r)=\Pi ^{-1}(\g _r)$ has
mean curvature vector of length bounded away from zero (not  necessarily
constant\footnote{The mean curvature of $\Pi_0^{-1}(\g _r)\subset \E(-4,1)$ is constant
equal to half of the geodesic curvature of $\g _r$ in $\Hip^2(-4)$.})
and pointing towards $\Sigma (\g )$.
\item The critical mean curvature of  $X$ (see Definition~\ref{defH(X)}) satisfies $H(X)\geq\de$,
where $\de$ is any positive lower bound for
the mean curvature function of $\Sigma (\g _r)$ for any given $r>0$.
  \end{enumerate}
\end{lemma}
\begin{proof}
Item~1 is well known and follows from the fact that  $E_3$ is a
principal Ricci curvature direction at every point in $X$; see~\cite[Proposition~2.21]{mpe11}
for details.

To prove item 2, consider the integral curve $\G \subset \sl $ of $E_3$
that passes through a given point $p\in \Sigma (\g )$. As $\Pi (\G )=\Pi (p)\in \g $,
  then $\G $ is entirely contained in $\Sigma (\g )$. Since
$R^{\G }\colon \E (-4,1)\to \E (-4,1)$ projects through $\Pi _0$ to
the rotation  $\psi _{\Pi (\G )}\colon \Hip ^2(-4)\to \Hip ^2(-4)$ of angle $\pi $
around $\Pi(\G)\in \g$, and $\psi _{\Pi (\G )}$ leaves
invariant the geodesic $\g $, then $R_\G (\Sigma (\g ))=\Sigma (\g )$.
This proves that item~2 holds.

Minimality in item~3 follows since the mean curvature vector of
$\Sigma (\g )$ maps to its negative under the differential of $R_\G $. $\Sigma (\g )$ is
stable because it admits a positive Jacobi function (namely, $\langle K,N\rangle $ where
$K$ is the right invariant vector field on $\sl $ generated by the hyperbolic
translations along a geodesic orthogonal to $\g $ in $\Hip ^2(-4)$, and $N$ is
a unit normal field to $\Sigma (\g )$).

To prove item~4, take a curve $\g _r$ at
distance $r>0$ from a geodesic $\g \subset
\Hip ^2(-4)$. Given a point $q\in \g _r$, there
exists a geodesic $\hat{\g }\subset \Hip ^2(-4)$ passing
through $q$ which is tangent to $\g _r$ at $q$, and
$\hat{\g }$ lies entirely on the side of
$\g _r$ that does not contain $\g$. By item~3,
the surface $\Sigma (\hat{\g })=\Pi ^{-1}(\hat{\g })$ is minimal.
Since $\Sigma (\hat{\g })$ is tangent to $\Sigma (\g _r)$ and lies at one side of
$\Sigma (\g _r)$, then the usual mean curvature comparison principle implies
that the mean curvature of $\Sigma (\g _r)$ with respect to the
unit normal vector field that points towards $\Sigma (\g )$, is nonnegative
along $\G =\Pi ^{-1}(\{ q\} )$. In fact, the mean curvature
of $\Sigma (\g _r)$ is strictly positive along $\G$
as we explain next: otherwise the second
fundamental forms $II$ of $\Sigma (\g _r)$ and $\wh{II}$ of $\Sigma (\hat{\g})$ would have
the same trace and the same value in the direction given by the tangent vector $\G'$
to $\G $; choosing $\G'$ as the first vector of an orthonormal basis $\mathcal{B}$ of the
common tangent plane of both surfaces at any point of $\G$, the matrix of the difference $II-\wh{II}$ with respect
to $\mathcal{B}$ is symmetric with zero diagonal entries; therefore, $II-\wh{II}$ has opposite eigenvalues
$a,-a$. If $a\neq 0$, then we contradict that $\Sigma (\wh{\g })$ lies at one side of $\Sigma (\g _r)$;
hence $II=\wh{II}$ along $\G $, which in turn implies that
$\g _r$ would have a point of second order contact with $\hat{\g }$, which is false.
Therefore, the mean curvature function of $\Sigma (\g _r)$ is strictly positive.

To see that the mean curvature function of $\Sigma (\g _r)$ is bounded away from zero,
consider the 1-parameter subgroup $G$ of hyperbolic translations of $\sl $ associated
to the translations of $\Hip^2(-4)$ along $\g $. Then, $\Sigma (\g _r)$ is invariant
under every element in $G$. Also, $\Sigma (\g _r)$ is invariant under
left translation by every element
in the center of $\sl $ (recall that elements in this center correspond to liftings to $\sl $
of rotations by multiples of $2\pi $ around a point in $\Hip ^2(-4)$). Hence
$\Sigma (\g _r)$ can be considered to be a doubly periodic surface. This property, together
with the positivity of the mean curvature function of $\Sigma (\g _r)$, imply that item~4
holds.

Finally, item~5 is a direct consequence of item~4 and of the mean curvature comparison
principle by the following argument. Let $\Delta \subset X$ be any closed
surface and take $r>0$. Since the mean convex side of $\Sigma(\g_r)$
contains arbitrarily large balls of $X$ in its interior, then after a left translation of $\Delta $,
we can assume that $\Delta $ lies on the mean convex side of $\Sigma(\g_r)$.
After further continuous left translations inside the mean convex side of $\Sigma(\g_r)$
applied to $\Delta $, we can assume that $\Delta $ is also tangent to $\Sigma(\g_r)$ at some common point $p$.
By the mean curvature comparison principle, the absolute mean curvature
function of $\Delta $ at $p$ is at least as large as any lower positive bound $\de$ of the
positive mean curvature function of $\Sigma(\g_r)$. Hence, $H(X)\geq \de$, which completes the proof.
\end{proof}

\begin{lemma}
\label{lemma3.6}
The push-forward by any left translation in $\sl$ of a right invariant elliptic (resp. hyperbolic, parabolic) vector field
is again right invariant elliptic (resp. hyperbolic, parabolic) vector field. Furthermore,
given a nonzero elliptic or parabolic vector field $F$ on $\sl $, the inner product of $F$
and the left invariant vector field $E_3$ given by (\ref{eq:lframesl}), has constant
nonzero sign on $\sl $, independently of the left invariant metric on $\sl $.
\end{lemma}
\begin{proof}
The fact that the push-forward by a left translation of a right invariant vector
field is again right invariant holds in any
Lie group, and is well known.

Consider now $F_a:= (l_a)_* (F)$, where $F$ is a right invariant vector field and $a\in \sl$.
If $\Gamma$ denotes the $1$-parameter subgroup with $\Gamma'(0)=F(e)$,
then the corresponding $1$-parameter subgroup $\Gamma_a$ that generates the right invariant vector field $F_a$
is given by $\Gamma_a=a \Gamma a^{-1}=l_a(r_{a^{-1}}(\Gamma))$.
In particular, $\Gamma_a$ and $\Gamma$ are conjugate subgroups,
and so they have the same character. Thus, $F_a$ and $F$ also have the same character, as claimed.

Finally, observe that from the paragraph just before the statement of Lemma~\ref{lemma4.1}
we deduce that the sign of the inner product of a pair of tangent vectors $v,w\in T_x\sl $ at any $x\in X$ does not
depend on the left invariant metric on $\sl $. Therefore, in order to prove the last sentence of the lemma,
it suffices to choose a nonzero elliptic or parabolic vector field $F$ and
prove that the function $\langle F,E_3\rangle \colon \sl \to \R$ has no zeros, where
$\langle,\rangle $  is the metric defined just before \eqref{eq:Ekt}.

Given $x\in \sl $, let $F^x$ be the vector field on $\sl $ defined by $(l_x)_*(F^x)=F$. By the first sentence of
Lemma~\ref{lemma3.6}, $F^x$ is a nonzero right invariant vector field of the same character as $F$ (i.e.,
elliptic or parabolic), and
\[
\langle F,E_3\rangle (x)=\langle (dl_x)_e(F_e^x),(E_3)_x\rangle =\langle F^x_e,(E_3)_e\rangle
=\langle F^x,E_3\rangle (e),
\]
where we have used that $E_3$ is left invariant and that $l_x$ is an isometry of $\langle ,\rangle $.
As $F^x$ is elliptic or parabolic, by Lemma~\ref{caracsl}, the coordinates $(a,b,c)$ of $F^x_e$ satisfy $a^2+b^2\leq c^2$,
which implies that $\langle F,E_3\rangle (x)$ cannot vanish, as desired. Now the proof of
Lemma~\ref{lemma3.6} is complete.
\end{proof}

We will finish this section with some considerations about two-dimensional subgroups of
$\sl $ and their constant left invariant Gauss map images.

Given $\t \in \partial _{\infty }\Hip ^2$, let $\Hip ^2_{\theta }\subset \sl $ be the set of
liftings to $\sl $ of the orientation-preserving isometries of $\Hip ^2$ that fix~$\theta $.
Elements in $\Hip^2_{\t }$ are of two types: liftings to $\sl $ of hyperbolic translations along geodesics
one of whose end points is $\t $, and liftings to $\sl $ of parabolic rotations around $\t $.
$\Hip ^2_{\theta }$ is a noncommutative, two-dimensional simply connected  subgroup of $\sl $, and $\Hip ^2_{\theta }$
contains a unique $1$-parameter parabolic subgroup, namely the liftings to $\sl $ of
parabolic rotations around $\t $. Every two-dimensional subgroup of $\sl $ is of the form $\Hip ^2_{\t}$
for some $\t \in \partial _{\infty }\Hip ^2$, see the paragraph around equation~(2.30)
in~\cite{mpe11} for more details.

Given a hyperbolic $1$-parameter subgroup $\G $ of $\sl $, by Lemma~\ref{caracsl}, the image of $\G $ through
the map $\Pi $ defined in (\ref{eq:PI}) is an arc in $\Hip ^2(-4)$ of constant geodesic curvature that
has two extrema $\t _1,\t _2\in \partial _{\infty }\Hip ^2$.
These two end-point values $\t _1,\t _2$ determine the two two-dimensional subgroups
$\Hip ^2_{\t _1}$, $\Hip ^2_{\t _2}$ of $\sl $ that contain $\G $.

The following result describes the image under the left invariant Gauss map of the circle
family $\{ \Hip ^2_{\theta }\ | \ \theta \in \partial _{\infty }\Hip ^2\} $ (recall that
the left invariant Gauss map depends on the left invariant metric on $\sl$, and thus, it depends on
the numbers $\l _1,\l _2,\l _3>0$ chosen just before Lemma~\ref{lemma4.1}).

\begin{lemma}
\label{lemma5.3}
The constant value $G_{\theta }\in \esf^2$ of the left invariant Gauss map of
$\Hip ^2_{\t }$  is (up to a sign coming from a change of orientation):
\begin{equation}
\label{Gsl}
G_{\t }=\frac{\pm 1}{\sqrt{\l _2\l _3\sin ^2\t +\l _1\l _3\cos ^2\t +\l _1\l _2}}
\left( -\sqrt{\l _2\l _3}\sin \t , \sqrt{\l _1\l _3}\cos \t ,-\sqrt{\l _1\l _2}\right) ,
\end{equation}
in coordinates with respect to the orthonormal frame
$\{ \l _i^{-1/2}(E_i)_e\ | \ i=1,2,3\} $ (see equation~(\ref{eq:lframesl})). The set
$\{ \pm G_{\t }\ | \ \t\in  \partial _{\infty }\Hip ^2\} $ is a
pair of antipodal, simple closed curves $\Upsilon \cup (-\Upsilon )\subset \esf^2$,
and it is invariant under the $\pi $-rotations in $\esf^2$ in the directions of
$(E_1)_e,(E_2)_e,(E_3)_e$, see Figure~\ref{slGauss}.
\end{lemma}
\begin{figure}
\begin{center}
\includegraphics[height=4.5cm]{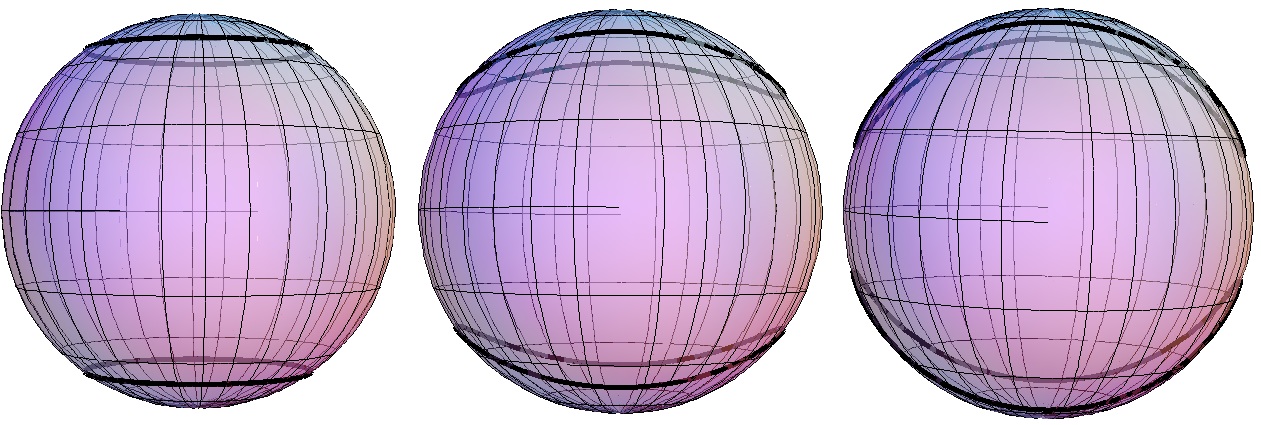}
\caption{Gauss map image of the family of two-dimensional subgroups $\Hip ^2_{\t }$
of $\sl $, with three different left invariant metrics. Left: $\l _1=\l _2=2,\l _3=1$ ($\E (\kappa ,
\tau )$-metric). Center: $\l _1=0.5$, $\l _2=2$, $\l _3=1$. Right: $\l _1=0.1$, $\l _2=4$, $\l _3=1$.}
 \label{slGauss}
\end{center}
\end{figure}
\begin{proof}
Let $v\in T_e\sl $ be a tangent vector with coordinates $(a,b,c)$ with respect to
$\{ (E_1)_e$, $(E_2)_e,(E_3)_e\} $. By Lemma~\ref{caracsl},
the tangent vector $v^P\in T_e\sl $ with coordinates $(a,b,c)=(-\sin \t ,\cos \t ,1)$
(resp. $v^H\in T_e\sl $ with $(a,b,c)=(\cos \t ,\sin \t ,0)$) produces
a parabolic (resp. hyperbolic) 1-parameter subgroup $\G^{P}$
(resp. $\G^H$) of $\sl $. Furthermore, the projection of $\G^P$
(resp. of $\G^H$)
through the map $\Pi $ defined in (\ref{eq:PI}) is the horocycle that passes
through the origin of the Poincar\'e disk and through the point at infinity
$e^{i\theta }\in \partial _{\infty }\Hip^2$
(resp. the segment with extrema $\pm e^{i\t }$).
It follows that both subgroups $\G^P,\G ^H$ are contained in
the two-dimensional subgroup $\Hip ^2_{\t }$ of $\sl $.

The coordinates of $v^P,v^H$ with
respect to the frame $\{ \l _i^{-1/2}(E_i)_e\ | \ i=1,2,3\} $  are
respectively $(-\sqrt{\l _1}\sin \t ,\sqrt{\l _2}\cos \t ,\sqrt{\l _3})$ and
$(\sqrt{\l _1}\cos \t ,\sqrt{\l _2}\sin \t ,0)$. After normalizing the cross product
of these 3-tuples,  we deduce that
the (constant) value of the left invariant Gauss map of $\Hip ^2_{\t }$ in the
left invariant metric on $\sl $ determined by $\l _1,\l _2,\l_3>0$ is the one given by
equation~(\ref{Gsl}). The remaining conclusions of the lemma are direct consequences of
(\ref{Gsl}).
\end{proof}

\section{The geometry of invariant limit surfaces of index-one spheres}
\label{sec:limits}
In what follows, we will denote by  $\chM _X$ the moduli space of all index-one
constant mean curvature spheres in a metric Lie group $X$
diffeomorphic to $\R^3$. Inside $\chM _X$ we have the
component $\mathcal{C}$ described in item~4 of Proposition~\ref{propsu2}.
We will also use the number $h_0(X)\geq 0$ defined in Proposition~\ref{propsu2}.

\begin{definition}
\label{def6.1}
{\rm
 We say that a complete, noncompact, connected
$H$-surface $f\colon\Sigma\looparrowright X$ is a \emph{limit
surface} of $\mathcal{C}$ with base point $p\in \Sigma $ (also called a
{\em pointed limit immersion} and denoted by $f\colon (\Sigma
,p)\looparrowright (X,e)$) if $e=f(p)$ and there exists a sequence
$\{ \wh{f}_n\colon S_n\la X\} _n\subset \mathcal{C}$, compact domains
$\Omega_n\subset S_n$ and points $p_n\in \Omega _n$ such that the
following two conditions hold:
 \begin{enumerate}[(G1)]
   \item  $f$ is a limit as $n\to \infty $ of the immersions $f_n=l_{\wh{f}_n(p_n)^{-1}}\circ (\wh{f}_n|_{\Omega _n})
\colon \Omega_n  \la X$ obtained by left translating $\wh{f}_n|_{\Omega _n}$ by the inverse of $\wh{f}_n(p_n)$
in $X$ (hence $f_n(p_n)=e$).  Here, the convergence is the uniform convergence in the
 $C^k$-topology for every $k\geq 1$, when we view the surfaces as
 local graphs in the normal bundle of the limit immersion.
  \item
 The area of $f_n$ is greater than $n$, for all $n\in \N$.
 \end{enumerate}}
\end{definition}

We should observe that, by item 6 of Proposition~\ref{propsu2}, any such pointed limit immersion has constant
mean curvature equal to $h_0(X)$. Moreover, by~\cite[Lemma 5.2]{mmpr4}, the space of pointed
limit immersions of $\mathcal{C}$ is nonempty.

\begin{proposition}
\label{DeltaF}
Let $f\colon (\Sigma ,p)\looparrowright (X,e)$ be a pointed limit immersion of $\mathcal{C}$, and let $G\colon
\Sigma \to \esf^2\subset T_eX$ be its left invariant Gauss map. Then, the rank of the differential $dG$ is constant on
$\Sigma $. Furthermore:
\begin{enumerate}[(1)]
\item If rank$(dG)=0$, then $f$ is injective and $f(\Sigma )$ is a two-dimensional subgroup of $X$.
\item If rank$(dG)=1$, then there exists a unique (up to scaling) right invariant vector field on $X$ which is
everywhere tangent to $f( \Sigma )$.
\end{enumerate}
\end{proposition}
\begin{proof}
Suppose first that $dG$ has rank zero at some point $q\in \Sigma$.
By \cite[equation (3.9) and Corollary 3.21]{mpe11},
there exists a two-dimensional subgroup $\Delta $ of $X$ which has the same mean
curvature as $f$ and (constant) left invariant Gauss map of value $G(q)\in \esf^2$.
We are going to prove next that $f(\Sigma )=l_{f(q)}(\Delta )=f(q)\, \Delta $.

Consider the foliation $\cF=\{ l_a(\Delta )=a\Delta \ | \ a\in X\} $
of left cosets of $\Delta $. Then, we may write the leaves of $\cF$ as the level sets of a Morse function without critical points
$h\colon X\to \R$, so that $\{ h=0\} = f(q)\Delta $. Also, we can consider local coordinates $(x_1, x_2, x_3)$
in $X$ around $f(q)$ with $x_3:=h$, so that $f(q)=(0,0,0)$ in these coordinates. By taking a small neighborhood $V$
of $f(q)$ in
$f(\Sigma )$, we may view $V$ as a graph $x_3=\varphi(x_1,x_2)$ over a small neighborhood $U$ of the
origin in the $(x_1, x_2)$-plane.

Arguing by contradiction, assume that $f(\Sigma )$ and $f(q)\, \Delta $ do not coincide in a neighborhood of $f(q)$.
Then, $\varphi$ is not identically zero around the origin. Since the graphs $x_3=\varphi(x_1,x_2)$ and $x_3=0$ have the
same constant mean curvature in $X$ (and thus define solutions to the same quasilinear elliptic PDE),
then $\varphi-0=\varphi$ satisfies a second order linear
homogeneous elliptic equation with smooth coefficients.
So, by Bers' theorem~\cite{Bers1}, we have $\varphi(x_1,x_2)=p_k(x_1,x_2) + o(\sqrt{x_1^2+x_2^2})^k$, where
$p_k(x_1,x_2)$ is a homogeneous polynomial of the form $p_k=h_k\circ \Phi$,
with $h_k$ a harmonic homogeneous polynomial of degree $k\geq 2$ in $\R^2$,
and $\Phi$ a linear transformation of the $(x_1,x_2)$-coordinates. In particular,
the level set structure of $\varphi(x_1,x_2)$ around the origin is $C^1$-diffeomorphic
to the level set structure of the harmonic function ${\rm Re}(z^k)$ around the origin in the complex plane;
see e.g., Kuo~\cite{kuo1} or Cheng~\cite{cheng1}.
Note that the level sets of $\varphi$ describe the intersection of the foliation $\cF$ with $V$ around $f(q)$.

As $f$ is a pointed limit immersion of $\mathcal{C}$, there exist
a sequence $\{ \wh{f}_n\colon S_n\la X\} _n\subset \mathcal{C}$, compact domains
$\Omega_n\subset S_n$ and points $p_n\in \Omega _n$ such that the conditions
(G1) and (G2) hold. To find the desired contradiction, we will show that for $n$ large enough,
the foliation $\cF$ intersects tangentially $\wh{f}_n(S_n)$ at some point
$\wh{f}_n(q_n)$ with $q_n\in \Omega_n$, so that $\wh{f}_n(S_n)$ has points at both sides
of $\wh{f}_n(q_n)\Delta $, which contradicts the Transversality Lemma~\cite[Lemma 3.1]{mmp2}.

To start, note that it follows from the above discussion on the asymptotic behavior of $\varphi$
around the origin that $f(\Sigma)$ intersects  $\cF$ nearby $f(q)$
in a hyperbolic manner, which is to say that the intersection of $\cF$ with $V$ produces a $1$-dimensional foliation
of $V$ with an isolated singularity  of index $-k$ at $f(q)$, and the leaves of $\cF$ are transverse to
$V$ except at $f(q)$.
For $a>0$ sufficiently close to zero,
 $V\cap \{x_3=a\}$ separates $V$ into of $k+1$ components,
of which $k$ lie in $\{ x_3>a\} $ while just one lies in $\{ x_3<a\} $.
Analogously, for $a<0$ sufficiently close to zero, $V\cap \{x_3=a\}$ separates $V$
into of $k+1$ components, of which $k$ lie in $\{ x_3<a\} $ while just one lies in $\{ x_3>a\} $.
For $n$ large, let $V(n) \subset \wh{f}_n(\Omega_n)$ be subdomains that are expressed as small
normal graphs over $V$ and converge to $V$ as $n\to \infty$.
Since the leaves of $\cF$ are transverse to
$V$ except at $f(q)$, for $a\neq 0$ fixed and sufficiently close to zero,
this description of components of $V- \{ x_3=a\} $
persists on $V(n)- \{ x_3=a\} $ for $n$ sufficiently large. Also observe
that $x_3|_{\partial V(n)}$ has no critical points in $\partial V(n)-\{ |x_3|>a\} $
(since this property holds for  $x_3|_{\partial V}$).
By Morse theory,
we then conclude that for $n$ large enough, there exists a tangency point
$\wh{f}_n(q_n)\in V(n)$ of $\wh{f}_n(S_n)$ with some leaf $\{ x_3=a_n\} $ such that
$\wh{f}_n(q_n)\to f(q)$ and $a_n\to 0$ as $n\to \infty$. This last property implies
that for $n$ sufficiently large,
there exist points $r_n,s_n\in \partial V(n)\subset \wh{f}_n(S_n)$ such that $x_3(r_n)<x_3(\wh{f}_n(q_n))$
and $x_3(s_n)>x_3(\wh{f}_n(q_n))$, which is the desired contradiction with the Transversality Lemma~\cite[Lemma 3.1]{mmp2}.

\par
\vspace{.1cm}
Suppose now that the rank of $dG$ is constant one on $\Sigma $.
By item~1 of Proposition~\ref{propos3.1}, $f(\Sigma )$ is invariant under the flow of a
nonzero right invariant vector field $K$ on $X$. If $\wt{K}$ is a right invariant vector field
on $X$ linearly independent from $K$, and $\wt{K}$ is everywhere tangent to $f(\Sigma )$,
then the Lie bracket $[K,\wt{K}]$ is also everywhere tangent to $f(\Sigma )$. This
implies that $K,\wt{K}$ generate an integrable two-dimensional subalgebra
of the algebra of right invariant vector fields on $X$, whose integral submanifold passing through $e$
is a two-dimensional subgroup $\Delta $ of $X$. By uniqueness of integral submanifolds of an
integrable distribution, we have $f(\Sigma )=\Delta $, which contradicts that the rank of $dG$
is constant one on $\Sigma $. Hence item~2 of the proposition is proved.

We next prove that if $dG$ has rank one at a some point $q\in \Sigma $, then the rank of $dG$ is one
everywhere on $\Sigma $, which together with the previous paragraphs
implies the constancy of the rank of $dG$ in any case and finishes the proof of Proposition~\ref{DeltaF}.
Suppose then that rank$(dG_q)=1$ at some $q\in \Sigma $. By the argument in Step~1 of the proof
of~\cite[Theorem~4.1]{mmpr4}, there is a nonzero right invariant vector field
$V$ on $X$ that has a contact at $f(q)$ with $f(\Sigma )$ of order at least two.
Suppose for the moment that $V$ is everywhere tangent to $f(\Sigma)$. In this case, the
differential of $G$ has rank at most one at every point of $\Sigma $ by item~1 of
Proposition~\ref{propos3.1}. But by the previous arguments of this proof, the rank of $dG$ cannot be zero
at any point of $\Sigma $, and so, $dG$ has constant rank one.

To finish the proof, we assume that $V$ is not everywhere tangent to $f(\Sigma)$ and we
will obtain a contradiction. This contradiction will follow from the
existence of two disjoint compact domains in the index-one $H_n$-spheres
$\wh{f}_n\colon S_n\looparrowright X$ that give rise to the limit immersion
$f$ (in the sense of conditions (G1) and (G2) above), such that each of
these compact domains is unstable. To create these compact
subdomains, we proceed as follows. Consider the Jacobi function
$J=\langle V, N \rangle$ on $\Sigma $, where $N$ stands for the unit
normal vector field along $f$. By~\cite[Theorem~2.5]{cheng1},
in a small compact neighborhood $E_q$ of $q$
in $\Sigma $, the set $J^{-1}(0)$ has the appearance of a
set of $k$ embedded arcs $\a_1,\ldots, \a_k$ crossing at
equal angles at $q$ with very small geodesic curvatures (their
geodesic curvatures all vanish at the common point $q$), where
$k\geq2$ is the degree of vanishing of $J$ at $q$; see
Figure~\ref{fig:nodal}. Let $E(n)\subset \Omega _n$ be compact disks
that converge to $E_q$ as $n\to \infty $, where $\Omega _n\subset S_n$ is
defined in condition (G1) above.
As $S_n$ has index one, for $n$ large,
the zero set of the Jacobi function $J_n=\langle V, N_n
\rangle$ (here $N_n$ is the unit normal vector of $\wh{f}_n$) is a
regular analytic Jordan
curve that decomposes
$S_n$ into two nodal domains $D_1(n)$, $D_2(n)$ and the zero sets
$J_n^{-1}(0)\cap \Omega _n$ converge as $n\to \infty$ to the zero set $J^{-1}(0)$ of
$J$. Then one of the two nodal domains,
say $D_1(n)$, intersects $E(n)$ in a connected set, see Figure~\ref{fig:nodal}.
\begin{figure}
\begin{center}
\includegraphics[height=5cm]{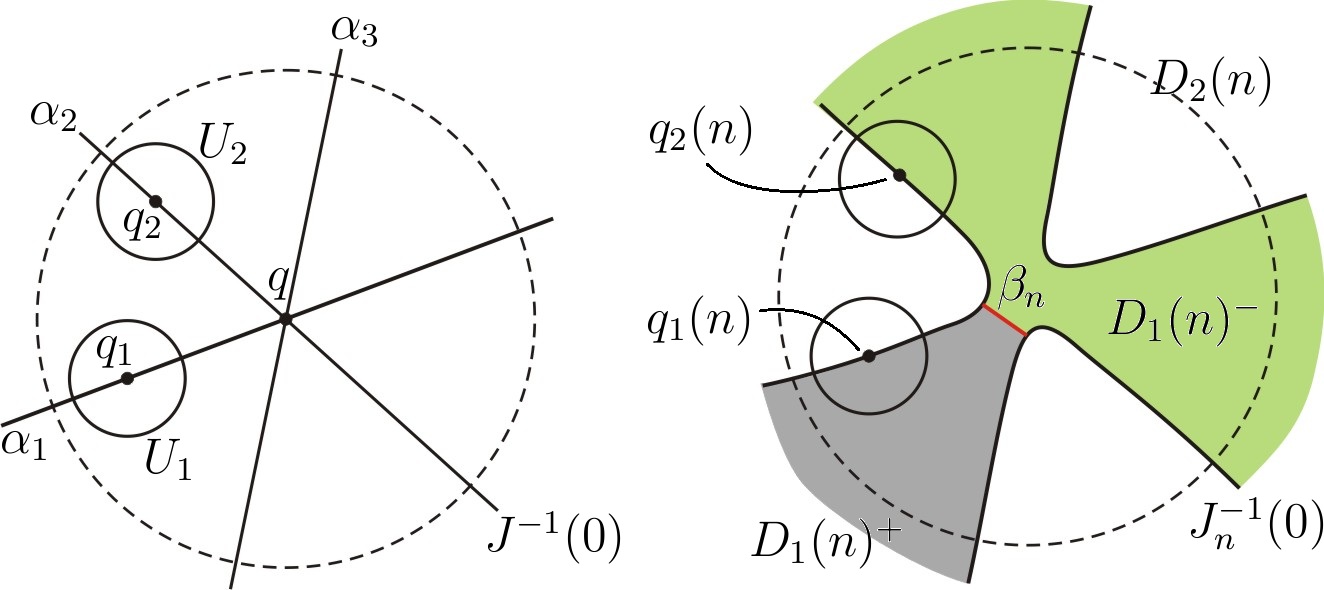}
\caption{Left: In a small neighborhood $E_q$ (represented as a dotted disk)
of the point $q\in \Sigma $,
the nodal lines of the Jacobi function $J$ form an
equiangular system of (almost geodesic) arcs crossing at $q$.
Right: The nodal lines of the related Jacobi
functions $J_n$ in a (dotted) disk $E(n)\subset S_n$ converging to $E_q$. $J_n^{-1}(0)$
divides $S_n$ into two disk regions $D_1(n),D_2(n)$, and the short
arc $\be _n$ divides
$D_1(n)$ in $D_1(n)^+,D_1(n)^-$. The two small disks centered at $q_1(n)$, $q_2(n)$
are $U_1(n)$, $U_2(n)$.}
\label{fig:nodal}
\end{center}
\end{figure}

Let $U_1,U_2\subset \Sigma $ be two small compact disjoint geodesic
disks of radius $\ve_0>0$ centered at points $q_1 \in \a_1$, $q_2\in
\a_2$ such that $U_1\cup U_2\subset E_q$ and $(U_1\cup U_2)\cap\a_j=\mbox{\O }$
for all $j\neq 1,2$. Let $q_1(n),q_2(n) \in J_n^{-1}(0)\cap E(n)$ be
points in the zero set of $J_n$ that converge to $q_1,q_2$,
respectively, and let $U_1(n),U_2(n)$ be $\ve_0$-disks in $E(n)$
centered at the points $q_1(n),q_2(n)$ that converge to $U_1,U_2$ as $n\to \infty $.
Choose a compact embedded short arc $\be _n\subset
D_1(n)-[U_1(n)\cup U_2(n)]$ such that
\begin{enumerate}[(H1)]
\item $\be _n$ joins points in different components of $E(n)\cap \partial
D_1(n)$, and $\be _n$ separates $U_1(n)\cap D_1(n)$ from $U_2(n)\cap
D_1(n)$ in $D_1(n)$.
 \item $\{ x\in S_n\ | \ d_{S_n}(x,\be _n)\leq\frac{1}{n}\} _n$ converges to $\{ q\} $ as $n\to \infty $
 and $\mbox{dist}_{S_n}\left( \be _n,U_1(n)\cup U_2(n)\right) >\de $ for some small $\de >0$ independent of $n$,
  where dist$_{S_n}$ denotes intrinsic distance in $S_n$.
\end{enumerate}

Note that $\be _n$ decomposes $D_1(n)$ into two disk components
$D_1(n)^+,D_1(n)^-$, each one being disjoint from exactly one of the
disks $U_1(n),U_2(n)$. Without loss of generality, we can assume that
$U_1(n)\cap D_1(n)^-=\mbox{\O }$ (and thus, $U_2(n)\cap D_1(n)^+=\mbox{\O }$), see
Figure~\ref{fig:nodal}.
 Let
$\phi _n\colon S_n\to [0,1]$ be a piecewise smooth cut-off function with the
following properties:
\begin{enumerate}[({I}1)]
\item $\phi _n(x)=0$ if dist$_{S_n}(x,\be _n)\leq \frac{1}{3n}$. 
\item $\phi _n(x)=1$ if dist$_{S_n}(x,\be _n)\geq \frac{1}{2n}$. 
\item $|\nabla \phi _n|\leq nC$ in $S_n$, 
where $C>0$ is a universal constant.
\end{enumerate}
Note that $\phi _nJ_n$ is a piecewise
smooth function that vanishes along $\partial D_1(n)^+$.

We claim that
\begin{equation}
\label{eq:31}
 \int _{D_1(n)^+}\left( |\nabla (\phi _nJ_n)|^2-P_n
\phi _n^2J_n^2 \right) \rightarrow 0 \quad \mbox{as }n\to \infty ,
\end{equation}
where $\mathcal{L}_n=\Delta +P_n$ is the stability operator of
$S_n$, i.e., $P_n=|\sigma _n|^2+\mbox{Ric}(N_n)$, $|\sigma _n|$
is the norm of the second fundamental form of $S_n$ and Ric$(N_n)$
denotes the Ricci curvature of $X$ in the direction of the unit
normal to $S_n$. To prove (\ref{eq:31}), observe that since
$\mathcal{L}_nJ_n=0$ on $S_n$, then
\begin{equation}
\label{eq:32}
 |\nabla J_n|^2-P_nJ_n^2 =|\nabla J_n|^2+J_n\Delta J_n=
\Div (J_n\nabla J_n).
\end{equation}
Multiplying by $\phi_n^2$ in (\ref{eq:32}) and integrating on $D_1(n)^+$ we find
\begin{equation}
\label{eq:33}
 \int _{D_1(n)^+}\left( \phi _n^2|\nabla J_n|^2-P_n\phi _n^2J_n^2\right)
= \int _{D_1(n)^+}\phi _n^2\Div (J_n\nabla J_n).
\end{equation}
 On the other hand, $\Div (\phi _n^2J_n\nabla J_n)= \frac{1}{2}\langle
\nabla (\phi _n^2),\nabla (J_n^2)\rangle +\phi _n^2\Div (J_n\nabla
J_n)$, hence the Divergence Theorem and the fact that
$\phi _nJ_n=0$ along $\partial D_1(n)^+$ give that
\begin{equation}
\label{eq:34}
\int _{D_1(n)^+}\phi _n^2\Div (J_n\nabla J_n)=
-\frac{1}{2}\int _{D_1(n)^+}\langle \nabla
(\phi _n^2),\nabla (J_n^2)\rangle .
\end{equation}

Equations
(\ref{eq:33}), (\ref{eq:34}) together with (I1), (I2) and (I3) imply that
\[
\int _{D_1(n)^+}\left( |\nabla (\phi
_nJ_n)|^2-P_n\phi _n^2J_n^2 \right)
 = \int_{D_1(n)^+}J_n^2|\nabla \phi _n|^2 =
\int_{\{
\frac{1}{3n}<d_{S_n}(x,\be _n)<\frac{1}{2n}\} \cap D_1(n)^+}J_n^2(x)|\nabla \phi _n|^2 (x)
\]
\[
\leq C^2n^2\cdot \mbox{Area}({\textstyle \{ \frac{1}
{3n}<d_{S_n}(x,\be _n)<\frac{1}{2n}\} })\cdot \max _{\{
\frac{1}{3n}<d_{S_n}(x,\be _n)<\frac{1}{2n}\}}(J_n^2(x)).
\]
As $\{ \wh{f}_n|_{\Omega_n}\} _n$ is convergent,  we conclude that
$\mbox{Area}(\{ \frac{1}{3n}<d_{S_n}(x,\be _n)<\frac{1}{2n}\})$ can be made smaller than
$C_1/n^2$ for $n$ large, for some $C_1>0$ independent of $n$.
Therefore, the last right-hand-side is bounded
from above by
\[
C^2C_1{\displaystyle \max _{\{
\frac{1}{3n}<d_{S_n}(x,\be _n)<\frac{1}{2n}\}}(J_n^2(x)).}
\]
which tends to zero as $n\to \infty $ by the first property in (H2) above, because
$J(q)=0$.
This proves our claim~(\ref{eq:31}).

Next we will define a test function for the stability operator on the domain
$D_1(n)^+\cup U_1(n)$.

To start, let us consider for $\ve_0>0$ sufficiently small, local coordinates $(x,y)$ around $q_1$ in $\Sigma$, so that:

\begin{itemize}
\item
$(x,y)$ are defined on $\D(\ve)=\{(x,y): x^2+y^2\leq \ve\}$ for some $\ve\in (0,\ve_0]$, and $\D(\ve)$
corresponds in these coordinates to a compact neighborhood $V_1\subset U_1$ of $q_1$.
\item
$\alfa_1\cap V_1 $ corresponds to the arc $\D(\ve)\cap \{x=0\}$.
 \item
The curves $\{y=\text{constant}\}$ are geodesic arcs orthogonal to $\{x=0\}\subset \alfa_1$.
\end{itemize}
By noting the convergence of the disks $U_1(n)$ to $U_1$ as $n\to \8$,
it is clear that we can choose coordinates $(x_{(n)},y_{(n)})$ in $S_n$
around $q_1(n)$ for each $n$ large enough on domains $V_1(n)$ converging
to $V_1$, given by a slight deformation of the coordinates $(x,y)$, and
with similar properties. More specifically, for $n\geq n_0(\ve)$ large enough
there exist local coordinates around $q_1(n)$ in $S_n$, that will be
denoted for simplicity also by $(x,y)$, so that:

\begin{itemize}
\item
$(x,y)$ are defined on $\D(\ve)=\{(x,y): x^2+y^2\leq \ve\}$ for some $\ve\in (0,\ve_0]$,
and $\D(\ve)$ corresponds in these coordinates to a compact neighborhood $V_1(n)\subset U_1(n)$ of $q_1(n)$.
\item
$D_1(n)^+\cap V_1(n) $ corresponds to $\D(\ve)\cap \{x\geq 0\}$.
 \item
The curves $\{y=\text{constant}\}$ are geodesic arcs orthogonal to $\{x=0\}\subset \partial D_1(n)^+$.
\end{itemize}

Given $\l\in (0,1)$, consider the {\it rhombus}
$R(n)=R(\ve,\l,n)\subset \D(\ve)$
that corresponds to the convex hull of the points $(\pm\l \ve,0),(0,\pm \ve)$ in these coordinates.
We define a function $v_1(n)\colon \D(\ve) \to \R$
in these local coordinates as a function of $x,y$ as follows:
\begin{enumerate}[(J1)]
\item $v_1(n)=0$ in $[D_2(n)\cap V_1(n)]-R(n)$.
\item $v_1(n)=J_n$ in $[D_1(n)^+\cap V_1(n)]-R(n)$.
\item $v_1(n)$ linearly interpolates the values of
$J_n|_{D_1(n)^+\cap \partial R(n)}$ and the value zero on $\partial R(n)\cap D_2(n)$ along the
geodesics $\{ y=\mbox{constant}\} $.
\end{enumerate}

For $n$ large, $v_1(n)$ can be extended to a piecewise smooth map on $S_n$ (also denoted by $v_1(n)$),
with the following properties for $\l ,\ve $ sufficiently small:
\begin{enumerate}[(J1)']
\item $v_1(n)=0$ in $S_n-[D_1(n)^+\cup R(n)]$,
\item $v_1(n)= \phi _nJ_n$  in $D_1(n)^+-R(n)$,
\item ${\displaystyle
\max _{R(n)}v_1(n)^2\leq   \max _{D_1(n)^+\cap \partial R(n)}J_n^2
=  \max _{D_1(n)^+\cap R(n)}J_n^2}$.
\item ${\displaystyle
\max _{R(n)}|\nabla v_1(n)|\leq \frac{3}{5}\min _{D_1(n)^+\cap R(n)}|\nabla J_n| }$.
\end{enumerate}
Observe that (J1)', (J2)' respectively follow from (J1), (J2)
and the fact that $\phi_n=1$ in $U_1(n)$ for $n$ large.
Let us explain next why properties (J3)' and (J4)' hold for $v_1(n)$.

To start, note that, in terms of the $(x,y)$ coordinates in $V_1(n)$,
the Jacobi function $J_n$ satisfies $J_n(0,y)=0$ and $|\nabla J_n|(0,y)\neq 0$
for every $y\in [-\ep,\ep]$. Thus, for $\ep$ small enough, $J_n$ can be
arbitrarily well approximated in $\D(\ep)$ by a linear function of the type $ax$
for some $a\neq 0$. This shows that the equality in (J3)' holds for $\ve,\landa$
small enough. The inequality in (J3)' is immediate from the definition of $v_1(n)$.

Similarly, by the previous comments and the definition by interpolation of $v_1(n)$
on $R(n)$, we see that $v_1(n)$ is arbitrarily well-approximated in $R(n)\subset
\D(\ep)$ for $\ve$ small enough by the piecewise linear function
$\frac{a}{2}(x-\landa |y| + \landa \ve)$, where $a\neq 0$ is the previous
constant associated to $J_n$. Thus, the (almost constant) norm of the gradient of
$v_1(n)$ in $R(n)\subset \D(\ep)$ is approximately $\sqrt{1+\landa^2}/2$ times
the (almost constant) norm of the gradient of $J_n$ in $\D(\ep)$.
By taking $\landa$ small enough, we obtain condition (J4)'.

Observe that $v_1(n)$ vanishes along the boundary of $D_1(n)^+\cup
R(n)$. We will prove that, for $\ve,\landa$ sufficiently small,
there is $n_0=n_0(\ve,\landa)$ such that, whenever $n\geq n_0$,
the following inequality holds:
\begin{equation}
\label{eq:35}
 \int _{D_1(n)^+\cup R(n)}\left( |\nabla v_1(n)|^2-P_nv_1(n)^2\right) <0.
\end{equation}
Assuming (\ref{eq:35}) holds, we conclude that $D_1(n)^+\cup R(n)$
is a strictly unstable domain in $S_n$, and thus, $D_1(n)^+\cup U_1(n)$ is strictly unstable as well.
Arguing in a similar way we
also conclude that $D_1(n)^-\cup U_2(n)$ is strictly unstable.
As $D_1(n)^+\cup U_1(n)$, $D_1(n)^-\cup U_2(n)$ have disjoint
interiors, we contradict that the index of $S_n$ is one. Hence it
only remains to prove (\ref{eq:35}) in order to find the
desired contradiction in the case that the differential $dG$
has rank one at some point but the right invariant Killing vector field
$V$ is not everywhere tangent
to $\Sigma $. We next show (\ref{eq:35}).

Since $v_1(n)=\phi _nJ_n$ in $D_1(n)^+-R(n)$, then
\begin{equation}
\label{eq:36}
 \begin{array}{rcl}
{\displaystyle
 \int _{D_1(n)^+\cup R(n)}\left( |\nabla v_1(n)|^2-P_nv_1(n)^2\right) }
&=&
 {\displaystyle
\int _{D_1(n)^+-R(n)}\left( |\nabla (\phi _nJ_n)|^2-P_n\phi
_n^2J_n^2\right) }
 \\
&+&
 {\displaystyle
\int _{R(n)}\left( |\nabla v_1(n)|^2-P_nv_1(n)^2\right) .}
 \end{array}
 \end{equation}
Using property (I2) above, we split the first integral in the right-hand-side of (\ref{eq:36}) for $n$ large
as
\begin{equation}
\label{eq:38} \int _{D_1(n)^+}\left( |\nabla (\phi _nJ_n)|^2-P_n\phi
_n^2J_n^2\right) -
 \int _{D_1(n)^+\cap R(n)}\left( |\nabla J_n|^2-P_nJ_n^2\right) .
\end{equation}
By the previous claim in equation~(\ref{eq:31}), the first integral
in (\ref{eq:38}) tends to zero as $n\to \infty $. 
Hence, to prove (\ref{eq:35})
we just need to show that if $\ve,\landa$ are sufficiently small and $n$
is large enough, then
\begin{equation}
  \label{eq:38bis}
  \int _{R(n)}\left( |\nabla v_1(n)|^2-P_nv_1(n)^2\right )-
\int _{D_1(n)^+\cap R(n)}\left( |\nabla J_n|^2-P_nJ_n^2\right )\leq C(\ve)
\end{equation}
for some constant $C(\ve,\landa )<0$.

First note that for $\ve,\landa$ sufficiently small and $n$ large enough
(once $\ep,\landa$ are fixed), we have by (J3)'
$$
\left| -\int _{R(n)}P_nv_1(n)^2+ \int _{D_1(n)^+\cap R(n)}P_nJ_n^2
\right| \leq
2 \max _{R(n)}|P_n|\cdot \max _{D_1(n)^+\cap R(n)}(J_n^2)\cdot \mbox{Area}[R(n)]
$$
\begin{equation} \label{eq:New5.9}
\leq 2\max _{R(n)}|P_n|\cdot \ve ^2\max _{D_1(n)^+\cap R(n)}|\nabla J_n|^2\cdot \mbox{Area}[R(n)].
\end{equation}
Observe that the potential $P_n$ converges smoothly to the corresponding potential
$P=|\sigma |^2+\mbox{Ric}(N)$ for the limit surface $\Sigma $. Therefore, $\max _{R(n)}|P_n|$
can be supposed to be less than some $\mu >0$ independent of $\ve,\landa$. This implies that
\begin{equation}
\label{eq:39}
2\max _{R(n)}|P_n|\cdot \ve ^2\max _{D_1(n)^+\cap R(n)}|\nabla J_n|^2\cdot \mbox{Area}[R(n)]
\stackrel{(A)}{\leq }4\mu \cdot \ve ^2\min _{D_1(n)^+\cap R(n)}|\nabla J_n|^2\cdot \mbox{Area}[R(n)],
\end{equation}
where in (A) we have used that the sequences of numbers
\[
\{ \max _{D_1(n)^+\cap V_1(n)}|\nabla J_n|\} _n,\quad \{ \min _{D_1(n)^+\cap V_1(n)}|\nabla J_n|\} _n
\]
 converge to the same
positive limit $|(\nabla J)(q_1)|$ when we make $\ve>0$ decrease to zero,
and take $n$ large enough with respect to each such choice of $\ve$.

On the other hand,
\[
\int _{R(n)}|\nabla v_1(n)|^2-\int _{D_1(n)^+\cap R(n)}|\nabla J_n|^2
 \leq
\]
\[
\max _{R(n)}|\nabla v_1(n)|^2 \cdot \mbox{Area}[R(n)]-
 \min _{D_1(n)^+\cap R(n)}|\nabla J_n|^2 \cdot \mbox{Area}[D_1(n)^+\cap R(n)]
\]
\[
\stackrel{(B)}{\leq }
\max _{R(n)}|\nabla v_1(n)|^2 \cdot \mbox{Area}[R(n)]-
 \min _{D_1(n)^+\cap R(n)}|\nabla J_n|^2 \cdot \frac{2}{5}\mbox{Area}[R(n)],
\]
where (B) holds for $\ve $ small and $n$ large enough;
here, we are using that by taking $\ve $ small enough, the metric space
structure on $S_n$ induced by $\wh{f}_n$ can be assumed to be arbitrarily close
to the flat one of the $(x,y)$-coordinates.
By property (J4)',
we obtain from the above inequality for $\ve,\landa$ small and $n$ large that
\begin{equation}
  \label{eq:41}
\int _{R(n)}|\nabla v_1(n)|^2-\int _{D_1(n)^+\cap R(n)}|\nabla J_n|^2 \leq
\left( \frac{9}{25}-\frac{2}{5}\right)
\min _{D_1(n)^+\cap R(n)}|\nabla J_n|^2 \cdot \mbox{Area}[R(n)].
\end{equation}
Hence, (\ref{eq:New5.9}), (\ref{eq:39}) and (\ref{eq:41}) give that
\[
\int _{R(n)}\left( |\nabla v_1(n)|^2-P_nv_1(n)^2\right)
 - \int _{D_1(n)^+\cap R(n)}\left( |\nabla J_n|^2-P_nJ_n^2\right)
\leq
\]
\[
\left( 4 \mu\, \ve ^2- \frac{1}{25} \right)
\min _{D_1(n)^+\cap R(n)}|\nabla J_n|^2\cdot \mbox{Area}[R(n)],
\]
which implies directly that inequality (\ref{eq:38bis}) holds for
$\ve,\landa $ small enough and $n$ large.
 It then follows that for $n$
large, there exist two disjoint
unstable regions on $S_n$ which  contradicts that $S_n$ has index one.
This completes the proof of the proposition.
\end{proof}

\begin{definition}
\label{def6.3}
{\rm
Suppose $f'\colon (\Sigma',p') \looparrowright
(X,e)$ is a pointed limit immersion of $\mathcal{C}$.
We define $\Delta(f')$ as the set of pointed
immersions $f\colon (\Sigma,p)\looparrowright (X,e)$ where $\Sigma$ is a complete,
noncompact connected surface, $p\in \Sigma $, $f(p)=e$ and
$f$ is obtained as a limit of $f'$ under an (intrinsically)
divergent sequence of left translations. In other words, there exist
compact domains $\Omega'_n\subset \Sigma'$ and points $q_n\in \Omega'_n$
diverging to infinity in $\Sigma'$ such that the sequence of
left translated immersions $\{ (l_{f'(q_n)^{-1}}\circ f')|_{\Omega'_n}\} _n$
converges on compact sets of $\Sigma$ to $f$ as $n\to \infty $.
}
\end{definition}

\begin{proposition}
\label{rank:limits}
Let $f'\colon (\Sigma',p') \looparrowright (X,e)$ be a pointed limit
immersion of $\mathcal{C}$. Then, the
space $\Delta(f')$ is nonempty and every $[f\colon (\Sigma ,p)
\looparrowright (X,e)]\in \Delta(f')$ satisfies:
\begin{enumerate}[(1)]
\item
$f$ is a limit surface of $\mathcal{C}$, thus of constant mean curvature $h_0(X)$.
 \item
$f$ is stable.
\item There exists a nonzero
right invariant vector field 
on $X$ which is everywhere tangent to $f(\Sigma)$.
\item $f(\Sigma)$ is topologically an immersed plane or annulus in $X$.
\item $\Sigma$ is diffeomorphic to a plane or an annulus.
\end{enumerate}
Moreover, if $f(\Sigma)$ is not a two-dimensional subgroup of $X$, then the
left invariant Gauss map image $G(\Sigma)$ of $f$ is a regular curve in $\S^2$.
\end{proposition}
\begin{proof}
The fact that $\Delta(f')$ is not empty and items 1,2 of the proposition follow directly from
the main statement of~\cite[Corollary~5.4]{mmpr4}. Item 3 of the proposition was
proved in item~1 of~\cite[Corollary~5.4]{mmpr4}, under the extra assumption (K1)
below (as in our explanation previous to Proposition~\ref{propos3.1}, assumption (K1)
was stated in~\cite[Corollary~5.4]{mmpr4} in terms of the $H$-potential, so one must
use~\cite[Corollary~3.21]{mpe11}
to reformulate it in the following manner):
\begin{enumerate}[(K1)]
\item No left translation of a two-dimensional subgroup in $X$ with constant mean curvature $h_0(X)$
is tangent to $f(\Sigma )$ at some point.
\end{enumerate}
Therefore, item~3 of the proposition will be proved if we demonstrate that
\begin{enumerate}[(K2)]
\item If $f(\Sigma )$ is not a two-dimensional subgroup, then property (K1) holds.
\end{enumerate}
We next prove (K2). Arguing by contradiction, suppose that $f(\Sigma )$ is not
a two-dimensional subgroup and there exist a two-dimensional subgroup $\Delta $ of $X$
with mean curvature $h_0(X)$, and points $x\in X$, $y\in f(\Sigma )\cap (x\Delta )$ such
that $f(\Sigma )$ and $x\Delta =l_x(\Delta)$ are tangent at $y$.
Therefore, $f(\Sigma )$, $x\Delta $ are different
surfaces tangent at $y$ with the same unit normal at this point and the same constant mean curvature.
Since $f$ is a limit surface of $\mathcal{C}$ by the already proven item~1 of this proposition,
we can easily adapt the arguments in the first part of the proof of
Proposition~\ref{DeltaF} to deduce that $f(\Sigma )=x\Delta $, which is
impossible since $e\in f(\Sigma)$, and so $f(\Sigma)=\Delta$, a subgroup.
Therefore, property (K2) holds and the proof
of item~3 of Proposition~\ref{rank:limits} is complete.

Once item~3 is proved, the proofs of items 4,5 of the proposition are the same as the proofs of the
related items~2,3 of~\cite[Corollary~5.4]{mmpr4}, respectively.

It remains to prove the `Moreover part' of the proposition, so assume that $f(\Sigma )$
is not a two-dimensional subgroup. By properties (K1), (K2) above, we conclude that
there are no two-dimensional subgroups in $X$ with mean curvature $h_0(X)$, whose
(constant) left invariant Gauss map lies in $G(\Sigma)$. In this setting, the already proven item~3
of this proposition implies that we can apply Proposition~\ref{propos3.1} to deduce that $f(\Sigma )$
is a regular curve. This completes the proof.
\end{proof}

The next technical lemma will be used in the proof of Corollary~\ref{cor:unique}.

\begin{lemma}[Unique Limit Surface Lemma]
\label{unique-limit}
\mbox{}\newline
Let $f_1\colon (\Sigma _1,p) \looparrowright (X,e)$ be a
pointed limit immersion of $\mathcal{C}$, with associated pointed immersions
$\wh{f}_n\colon (S_n,p_n) \looparrowright (X,e)$, $S_n\in \mathcal{C}$, so that
each $H_n$-sphere $S_n$ contains a compact subdomain $\Omega _n^1$ with $p_n\in \Omega _n^1$
and $\wh{f}_n|_{\Omega_n^1}$ converges to $f_1$  as $n\to \infty $,
and {\rm Area}$(\wh{f}_n|_{\Omega _n^1})>n$.
 Suppose that the rank of the differential $dG^1$ of the left invariant Gauss map $G^1$
of $f_1$ is one. Let $G_n$ be the left invariant Gauss map  of $\wh{f}_n$ and let
$q_n\in S_n$ be
the unique point with $G_n(q_n)=v=G^1(p)$.
Then, after choosing a subsequence, there exist  compact domains $\Omega_n^2
\subset S_n$ with $q_n \in \Omega_n^2$ such that:
\begin{enumerate}[(1)]
\item  The sequence of pointed $H_n$-immersions
$(l_{\wh{f}_n(q_n)^{-1}}\circ \wh{f}_n)|_{\Omega_n^2} \colon (\Omega_n^2,q_n)
\looparrowright (X,e)$ obtained by left translating $\wh{f}_n|_{\Omega_n^2}$ by
the inverse of $\wh{f}_n(q_n)$ in $X$, converges to a pointed
immersion $f_2\colon (\Sigma_2,q) \looparrowright (X,e)$ that is a limit of
$\mathcal{C}$,
and which has unit normal vector $v$ at $q\in \Sigma _2$.
\item The immersion $f_2$ has the same image as $f_1$.
\end{enumerate}
\end{lemma}
\begin{remark}
{\em
Lemma~\ref{unique-limit} remains true if we drop the hypothesis that
the rank of $dG^1$ is one, but we will not need this more general version in this paper.
}
 \end{remark}
\begin{proof}
The existence of the compact subdomains $\Omega _n^2\subset S_n$ (after passing to a
subsequence) and item~1 follow from standard arguments in elliptic theory, as we have uniform curvature
estimates for the $H_n$-immersions $l_{\wh{f}_n(q_n)^{-1}}\circ \wh{f}_n\colon
(S_n,q_n)\looparrowright (X,e)$ by item~5 of Proposition~\ref{propsu2}.
By construction,  the $h_0(X)$-immersions $f_1,f_2$ have an oriented
contact of order at least one at $e$.
By Proposition~\ref{DeltaF}, the differentials of the respective Gauss
maps $G^1$ and $G^2$ of $f_1$ and $f_2$ have constant ranks, which are
possibly different. We now distinguish cases depending on the value of rank$(dG^2)$.
\begin{enumerate}[(L1)]
\item \underline{Suppose rank$(dG^2)=0$.}
In this case, item~1 of Proposition~\ref{DeltaF} gives that
$f_2$ is injective and $f_2(\Sigma _2)$ is a two-dimensional
subgroup of $X$. Since $f_1(\Sigma )$ is not a two-dimensional
subgroup (because rank$(dG^1)=1$), and both $f_1(\Sigma _1),f_2(\Sigma _2)$ are
tangent at $e$ with the same constant mean curvature, we can apply the arguments in
the first part of the proof of Proposition~\ref{DeltaF} (see also the proof of
property (K2) above) to find a contradiction. Therefore, this case cannot occur.
\item
\underline{Suppose rank$(dG^2)=1$.}
In this case, $G^1(\Sigma _1)$ and $G^2(\Sigma _2)$ are analytic
immersed curves $\a_1,\a_2$ in $\S^2$ passing through $v\in \esf^2$.
If $\a_1,\a_2$ intersect tangentially at $v$, then by Lemma~\ref{tangin}
we have $f_1(\Sigma_1)=f_2(\Sigma_2)$ and so the lemma holds. So, assume next that $\a_1,\a_2$
are transverse at $v$. Thus, given $\ep>0$ small enough there is
some $n_0=n_0(\ep)$ such that, if $n\geq n_0$, the images through $G_n$
of the intrinsic balls $B_{S_n}(p_n,\ep)$, $B_{S_n}(q_n,\ep)$ centered at $p_n$ and $q_n$
of radius $\ep$, must intersect in an open nonempty set of $\S^2$ near $v$. As
$G_n$ is a diffeomorphism, $B_{S_n}(p_n,\ep)\cap B_{S_n}(q_n,\ep)\neq \mbox{\O}$,
which implies that $d_{S_n}(p_n,q_n)<2\ep$. Therefore, $d_{S_n}(p_n,q_n)\to 0$ as $n\to \8$,
and this implies $f_1(\Sigma_1)=f_2(\Sigma_2)$.
This proves Lemma~\ref{unique-limit} in this case.
\item
\underline{Suppose rank$(dG^2)=2$.} By the Inverse Function Theorem applied to
$G^2$ around $q$, we can find $\ve >0$ small such that if
$B_{\ep}$ denotes the intrinsic ball in $\Sigma_2$ of radius $\ep$ centered at $q$,
then $G^2|_{B_{\ep}}$ is a diffeomorphism onto its image. Choosing $\ep>0$ small enough,
we can assume that for $n$ large, $G^2(B_{\ep})\subset G_n(B_{S_n}(q_n,2\ep))$,
where $B_{S_n}(q_n,2\ep)$ is the intrinsic ball of radius $2\ep$ around $q_n$ in $S_n$.
As $G_n(p_n)\to v$, we have $G_n(p_n)\in G_n(B_{S_n}(q_n,2\ep))$, and as $G_n$
is a diffeomorphism, this implies that $p_n\in B_{S_n}(q_n,2\ep)$ for $n$ large enough.
As $\ep$ is arbitrarily small, we conclude that $d_{S_n}(p_n,q_n)\to 0$ as $n\to \8$.
This implies that the limit surfaces $f_1(\Sigma_1),f_2(\Sigma _2)$ are the same,
so in particular ${\rm rank}(dG^2)=1$, a contradiction.
Therefore, this case does not occur, and this completes the proof of Lemma~\ref{unique-limit}.
\end{enumerate}
\par
\vspace{-.5cm}
\end{proof}
As a consequence of Lemma~\ref{unique-limit}
and the Transversality Lemma~\cite[Lemma~3.1]{mmp2}, we have the
following corollary.

\begin{corollary}
\label{cor:unique}
Suppose $f\colon(\Sigma,p) \looparrowright (X,e) \in \Delta(f')$, where $f'\colon (\Sigma',p')
\looparrowright (X,e)$ is a limit immersion of $\mathcal{C}$.
If $f(\Sigma)$ is tangent at some point to a left or right coset of a
two-dimensional subgroup of $X$, then $f(\Sigma)$ is contained in one of the two closed
complements of this coset in $X$.
\end{corollary}
\begin{proof}
Suppose the image immersed surface $f(\Sigma)$ is
tangent to a left or right coset $E$ of some two-dimensional
subgroup of $X$ at a point $f(y)\in f(\Sigma )\cap E$, where
$y\in \Sigma $. By the discussion just before property (C)
in Section~\ref{backgroundCMC}, we can assume that $E=f(y)\Delta $ for
some two-dimensional subgroup $\Delta $ of $X$.
After a left translation of $f$ by $f(y)^{-1}$, we may assume that $p=y$,
and hence $f(\Sigma )$
is tangent to $\Delta $ at $e$.

If $f(\Sigma )$ is a two-dimensional subgroup of $X$, then
$f(\Sigma )=\Delta $ and there is nothing to prove. Hence in the sequel we will
assume that $f(\Sigma )$ is not a two-dimensional subgroup of $X$. In this case,
as $f\in \Delta(f')$ and $f'$
is a limit
immersion of $\mathcal{C}$, then Proposition~\ref{rank:limits} gives that
$G(\Sigma )$ is a regular curve in $\esf^2$, where $G\colon \Sigma \to \esf^2
\subset T_eX$ is the left invariant Gauss map of $f$. In particular, rank$(dG)=1$.

Applying Lemma~\ref{unique-limit} to $f_1=f$, we conclude that
 $f(\Sigma)$ coincides with the image
set $f_2(\Sigma_2)$ of the limit $f_2\colon (\Sigma_2,p_2)\la (X,e)$
of a sequence of pointed immersions $f_n^2\colon (\Omega_n ^2,q_n)
\looparrowright (X,e)$, where each $\Omega _n^2$ is a compact
subdomain of an $H_n$-sphere $S_n$ in $\mathcal{C}$ and such that the
left invariant Gauss map of $f_n^2$ at $q_n\in \Omega _n^2$ is $v=G(p)$.
By the Transversality Lemma,  the spheres
 $f^2_n(S_n)$ all lie on one side of
$\Delta$. Thus, $f(\Sigma )=f_2(\Sigma_2)=\lim _{n\to \infty}f^2_n(\Omega
_n^2)$ must lie on one side of $\Delta$. This completes the proof of the corollary.
\end{proof}

\begin{lemma}
\label{lam:gauss}
Suppose $f\colon(\Sigma,p) \looparrowright (X,e) \in \Delta(f')$,
where $f'\colon (\Sigma',p') \looparrowright (X,e)$ is a limit immersion
for $\mathcal{C}$. For each $q\in \Sigma$,
let $f_{q}=l_{f(q)}^{-1}\circ f\colon (\Sigma ,q)\looparrowright (X,e)$
denote the related pointed immersion from
$(\Sigma,q)$ to $(X,e)$ obtained from $f$ after a change of base
point and left translating by $f(q)^{-1}$.
Consider the set $\Delta(f)$ of limits of $f$ under a divergent sequence left translations,
in the sense of Definition~\ref{def6.3}.
Then:
\begin{enumerate}[(1)]
\item 
$\Delta(f)$ is a subset of $\Delta(f')$.
\end{enumerate}
Assume that 
$\Delta(f)$ contains no elements with
constant left invariant Gauss map. Then:
\begin{enumerate}[(1)]
\setcounter{enumi}{1}
\item The image $\gamma_f=G(\Sigma )$ of the left invariant Gauss map $G \colon \Sigma \to \esf^2 \subset
T_eX$ of $f$
is a complete embedded regular curve in $\S^2$. Moreover, its closure
$\overline{\gamma_f}\subset \S^2$ admits the structure of a
lamination of $\S^2$, whose leaves correspond to the Gauss
map images $\gamma_{\hat{f}}$ of elements $\hat{f}\in
\Delta(f)$. In particular, if $\hat{f}\in
\Delta(f)$, then
$\Delta(\hat{f})\subset \Delta(f)$ and so, $\overline{\gamma_{\hat{f}}}$ is a sublamination of
$\overline{\gamma_f}$.

\item There is a uniform upper bound on the absolute geodesic
curvature of all the leaves of $\overline{\gamma_f}$.

\item There exists $\hat{f}\in 
\Delta(f)$ such that $\overline{\gamma_{\hat{f}}}$ contains no proper sublaminations.
Furthermore, one of the following two possibilities holds:
\begin{enumerate}[(a)]
\item If $\g _{\hat{f}}$ is a closed curve, then the lamination
$\overline{\g _{\hat{f}}}$
contains a single leaf.
\item If $\g _{\hat{f}}$ is not a simple closed curve,
then the lamination $\overline{\g _{\hat{f}}}$ has uncountably many leaves and
$\gamma_{\hat{f}}$ has the following recurrency property: given
any compact arc $I$ of $\gamma_{\hat{f}}$, there exists a sequence
of intrinsically divergent, pairwise disjoint arcs in
$\gamma_{\hat{f}}$ that converge to $I$ in the $C^{1}$-topology.
 \end{enumerate}
 \end{enumerate}
\end{lemma}
\begin{proof}
Item~1 follows from a standard diagonal argument.

In the sequel, we will assume
that $\Delta(f)$ 
contains no elements with constant left invariant Gauss map. In particular, $f(\Sigma )$
is not a two-dimensional subgroup of $X$.

We next prove item~2. By Proposition~\ref{rank:limits}, there exists a nonzero, right invariant
vector field on $X$ which is everywhere tangent to $f(\Sigma )$ and
the Gauss map image $\g _f$ is a regular curve in $\esf^2$. Next we show
that the curve $\g _f$ is complete, or equivalently, it has no end points in
$\esf^2$. If such an end point $x\in \esf^2$ of $\g _f$ exists, then we can
consider a divergent sequence $q_n\in \Sigma $ such that $G(q_n)\to x$.
After passing to a subsequence,
there exist compact subdomains $\Omega _n\subset \Sigma $ with $q_n\in \Omega_n$ such that
the restrictions $f_{q_n}|_{\Omega _n}$ converge to an element $\widetilde{f}\in \Delta(f)$
(see the first paragraph in the proof of Lemma~\ref{unique-limit} for a similar argument).
Clearly $\wt{f}$ has constant left invariant Gauss map, which contradicts our hypothesis.
Therefore, $\g _f$ is complete.

By the arguments in the proof of Lemma~\ref{tangin}, $\g _f$ has no tangential self-intersections.
Transversal self-intersections of $\g _f$ can also be ruled out
by using a straightforward modification of the arguments in case~(L2) of the proof of
Lemma~\ref{unique-limit}.
Therefore, the curve $\g _f$ is embedded.
Also, note that the surface $f(\Sigma)$ can be viewed locally as the graph of a solution
to a quasilinear elliptic PDE, and recall that $f$ has uniformly bounded second fundamental form.
By standard elliptic estimates, this provides a priori $C^3$ estimates for $f(\Sigma)$,
and this shows in particular that $\g _f$ has bounded geometry.
These properties for $\g _f$ imply that its closure $\overline{\g _f}$
has the structure of a lamination of $\esf^2$, all whose leaves have bounded
geometry as well. Hence item~3 of the lemma is proved.

If $\widehat{f}\in 
\Delta(f)$,
then the arguments above apply to $\widehat{f}$ to give that the Gauss map
image $\g _{\widehat{f}}$ of $\widehat{f}$ is a complete embedded regular
curve in $\esf^2$. By definition, there exists a sequence $\{ q'_n\} _n
\subset \Sigma $ such that
the immersions $f_{q'_n}$ converge smoothly to $\widehat{f}$,
and thus the corresponding
Gauss map images $\g _{f_{q'_n}}$ converge to $\g _{\widehat{f}}$ in $\esf^2$.
But clearly $\g _{f_{q'_n}}=\g _{f}$ as sets, from where we deduce that the leaves
of the lamination $\overline{\g _f}$ correspond to the Gauss map images of
 elements of
$\Delta(f)$. This proves item 2 of the lemma.

We finish by proving item~4.
Consider the set $\mathcal{S}$ of sublaminations of the lamination $\overline{\g _f}$,
which is partially ordered by the inclusion. We want to apply Zorn's lemma to
$\mathcal{S}$ in order to find a minimal element in $\mathcal{S}$, i.e., a sublamination of
 $\overline{\g _f}$ with no proper sublaminations. To do this, we must check that
every totally ordered subset $\mathcal{S}_1$ of $\mathcal{S}$ has a lower bound.
This is clear provided that the intersection of all sublaminations in $\mathcal{S}_1$
is nonempty. Since $\mathcal{S}_1$ is a collection of closed sets of the compact
topological space $\esf^2$ and $\mathcal{S}_1$ clearly satisfies the finite
intersection property, then the intersection of all sublaminations in $\mathcal{S}_1$
is nonempty. By Zorn's lemma, there exists a sublamination $\mathcal{L}$ of
 $\overline{\g _f}$ with no proper sublaminations. Take a leaf $\G $ of $\mathcal{L}$.
 As $\G $ is a leaf of $\overline{\g _f}$, then item~2 implies that there exists
an immersion $\widehat{f}\in
\Delta(f)$ whose Gauss map image
$\g _{\widehat{f}}$
is equal to $\G $. Therefore, $\overline{\g _{\widehat{f}}}$ is a sublamination of
$\mathcal{L}$, and by minimality of $\mathcal{L}$ in $\mathcal{S}$ implies that
 $\overline{\g _{\widehat{f}}}=\mathcal{L}$. This proves the
first sentence of item~4 of the lemma.

In order to prove items 4a and 4b, it is worth adapting some
known facts about laminations to our setting. Take a sublamination
$\overline{\g _{\widehat{f}}}$ of $\overline{\g _f}$ with no proper sublaminations,
corresponding to an element $\widehat{f}\in 
\Delta(f)$. A point $x\in \overline{\g _{\widehat{f}}}$ is called a
{\it limit point of} $\overline{\g _{\widehat{f}}}$ if $x$ is the limit in $\esf^2$ of
an intrinsically divergent sequence $\{ x_n\} _n\subset \g _{\widehat{f}}$.
If $x\in \overline{\g _{\widehat{f}}}$ is a limit point of $\overline{\g _{\widehat{f}}}$,
then the leaf component $L$ of the lamination $\overline{\g _{\widehat{f}}}$ that contains $x$
consists entirely of limit points of $\overline{\g _{\widehat{f}}}$ (and $L$ is called
a {\it limit leaf}).
The set ${\rm Lim} (\g _{\widehat{f}})$ of limit points of $\overline{\g _{\widehat{f}}}$
is a (closed) sublamination of $\overline{\g _{\widehat{f}}}$, possibly empty.

If $\g _{\widehat{f}}$ is a closed curve, then clearly ${\rm Lim} (\g _{\widehat{f}})$ is empty and so,
$\overline{\g _{\widehat{f}}}$ consists of the single leaf $\g _{\widehat{f}}$, which is item~4a of the lemma.
Next suppose that $\g _{\widehat{f}}$ is not a simple closed curve.
Thus, ${\rm Lim} (\g _{\widehat{f}})\neq \mbox{\O }$.
Since $\overline{\g _{\widehat{f}}}$ contains no proper sublaminations, then
${\rm Lim} (\g _{\widehat{f}})=\overline{\g _{\widehat{f}}}$.
We next show that $\overline{\g  _{\widehat{f}}}$ contains an uncountable number of leaves.
Consider a small compact arc $\a\subset \esf^2$ cutting $\ov{\g_{\widehat{f}}}$ transversally.
Since complete one-manifolds are second countable in their intrinsic topology
(each leaf of the lamination $\ov{\g _{\widehat{f}}}$ has this property),
to prove that $\ov{\g _{\widehat{f}}}$ has an uncountable number of leaves,
as a lamination has a local product structure it suffices to
prove that  $W=\a\cap \ov{\g _{\widehat{f}}}$ is uncountable. We can consider $W$ to be a
complete metric space (note that $\a $ is a compact arc in $\esf^2$ with its
usual topology, hence $\a $ has a natural structure of a complete metric
space, and $W$ is a closed subset of $\a $). $W$ has no isolated points, since
${\rm Lim} (\g _{\widehat{f}})=\overline{\g _{\widehat{f}}}$. In this setting, the uncountability
of $W$ is a consequence of the following well known elementary application of the Baire category theorem:
\emph{any complete metric space without isolated points is uncountable.}

Finally, the recurrency property in item~4b of the lemma follows easily from the fact that
$\g _{\widehat{f}}\subset \overline{\g _{\widehat{f}}}={\rm Lim}(\g _{\widehat{f}})$.
Now Lemma~\ref{lam:gauss} is proved.
\end{proof}

For the remainder of this section
we will use the notation introduced in Lemma~\ref{lam:gauss},
and denote by $f\colon (\Sigma ,p)\la (X,e)$ an element of
$\Delta(f')$ such that
{\bf $\Delta(f)$  contains no elements with constant left invariant Gauss map}.
In particular, $f(\Sigma )$ is not a two-dimensional subgroup of $X$.

By Proposition~\ref{rank:limits}, $\g_f$ is a regular curve in $\esf^2$ and thus,
the right invariant vector field
$K_{\Sigma }$ given by item~2 of Proposition~\ref{DeltaF} is unique up to scaling.
Once we pick $K_{\Sigma}$ and an orientation on $X$,
 the curve $\gamma_f\subset \esf^2$
has a natural orientation as follows.

\begin{definition}
\label{def4.7}
{\rm
Suppose that $q\in \Sigma$ and
$N(q)\in T_{f(q)} X$ is the unit normal vector to $\Sigma$ at $q$. Consider
a short arc $\be$ in $\Sigma$ transverse to the integral curves of
the vector field $\wt{K}_{\Sigma }$ induced by $K_{\Sigma}$ on $\Sigma $,
parameterized so that
$\be(0)=q$ and $\{ df_q(\be '(0)),
K_{\Sigma}(f(q)),N(q)\} $ is a positively oriented basis for $T_{f(q)}
X$. Now, the orientation on $\g _f$ is defined to be the one given
by $G\circ \be $ where $G$ is the Gauss map of $f$.
}
\end{definition}

\begin{lemma}
\label{ass:limit}
Assume that $\g _f$ is not a simple closed curve
in $\esf^2$. Then, there exists a compact interval $\sigma \subset
\g _f$ with $G(p)\in \sigma $ (here $G$ is the left invariant Gauss map of $f$
and $p$ is the base point of $f$)
and there exists a sequence of pairwise disjoint arcs $\sigma_n\subset \g
_f$ that are small normal graphs over $\sigma$, which converge to $\sigma$
in the $C^1$-topology and, as graphs oriented by $\g _f$, satisfy
that their induced orientations are opposite of the orientation of $\sigma$, see Figure~\ref{fig15}.
\begin{figure}
\begin{center}
\includegraphics[height=3.5cm]{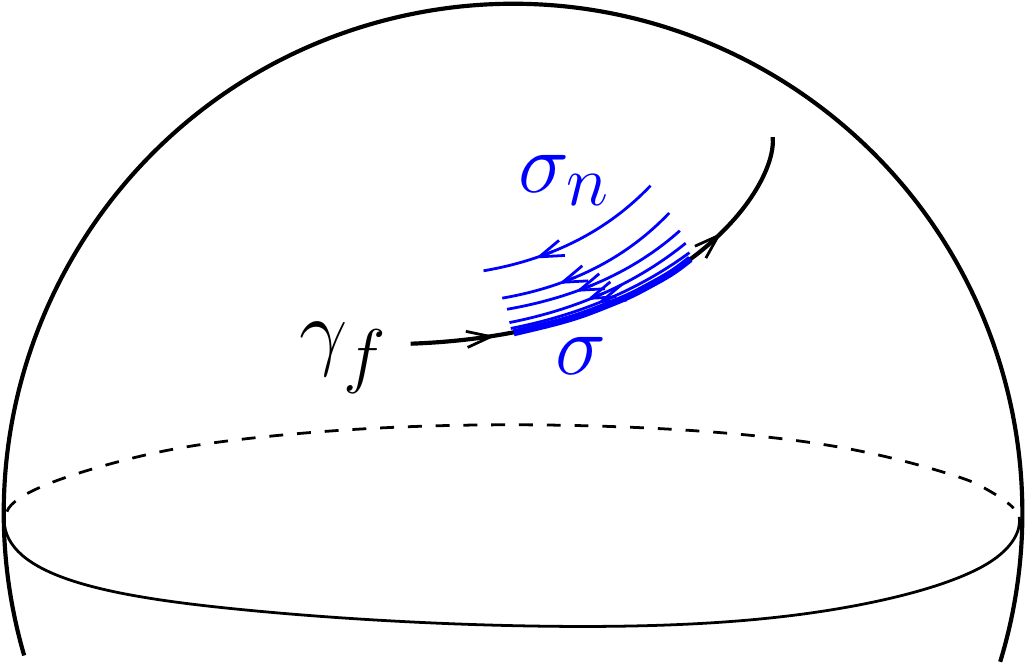}
 \caption{If $\g_f$ is not a simple closed curve, then it contains
compact arcs $\sigma_n$ converging to a compact arc $\sigma$, with the orientations of the $\sigma_n$
being opposite to the one of $\sigma $.}
\label{fig15}
\end{center}
\end{figure}
\end{lemma}
\begin{proof}
By item~4b of Lemma~\ref{lam:gauss},
given a small  arc $\sigma \subset \g _f$ containing $G(p)$,
there exists a
sequence of pairwise disjoint intervals $\sigma _n\subset \g _f$
which converge $C^1$ in $\esf^2$ to $\sigma $. After extracting a
subsequence, we can assume that all the arcs $\sigma _n$ lie on the
same side of $\sigma $; more specifically, we may assume that there exists a one-sided
neighborhood $U$ of $\sigma $ diffeomorphic to a square $[0,1]\times
[0,1]$, so that $\sigma $ corresponds to $\{ 0\} \times [0,1]$ and
for each $n\in \N$,
$\sigma _n$ corresponds to $\{ 1/n\} \times [0,1]$ under this
diffeomorphism. Also note that the lamination structure on the closure of
$\g _f$ ensures that we can take $U$ so that every component of $\g
_f\cap U$ is a small normal graph over $\sigma $. Therefore, in the
model of $U$ as $[0,1]\times [0,1]$, $\g _f\cap U$ can be
represented as $A\times [0,1]$ where $A\subset [0,1]$ is a compact
infinite subset, $1\in A$ and $0$ is an accumulation point of $A$.

Without loss of generality, we can also assume that the
orientation on $\sigma $  induced by the one of $\g _f$ is pointing upwards
in this model of $\sigma $ as $\{ 0\} \times [0,1]$. If Lemma~\ref{ass:limit}
fails, then we have that after choosing $U$ small
enough, every component of $\g _f\cap U\equiv A\times [0,1]$ is
oriented by pointing upwards. We can assume that $\g _f$ is
parameterized by the real line $\R $, and its orientation is the
induced one by this parameterization. Then, each of the segments
$\sigma , \sigma _n$ corresponds respectively to a closed interval $I,I_n\subset
\R$ by this parameterization, and these intervals form a pairwise
disjoint collection. Each point or closed interval $J$ in $\R $
defines a ``future'' (the component of $\R -J$ that limits to
$+\infty $) and a ``past'' (the component of $\R -J$ that limits to
$-\infty $), and these definitions can be translated to $\g _f$ via
the parameterization. After passing to a subsequence, we can assume
that one of the two following possibilities holds:
\begin{enumerate}[(M1)]
\item $I_n$ is contained in the future of $I$, for all $n\in \N$.
\item $I_n$ is contained in the past of $I$, for all $n\in \N$.
\end{enumerate}

If case (M1) holds, then consider the  end point
$t_1=\max (I)$ and the first $t_2>t_1$ whose image by the
parameterization lies in the closed topological square $U\subset
\S^2$; the image of the interval $(t_1,t_2)\subset \R $ is an open
Jordan arc $J$ in $\esf^2-U$ with endpoints in $\partial U$
corresponding to the points $(0,1)$ and $(r,0)$ for a certain $r\in
(0,1]$, in the model of $U$ as $[0,1]\times[0,1]$, see Figure~\ref{figsquare}.
\begin{figure}
\begin{center}
\includegraphics[height=5cm]{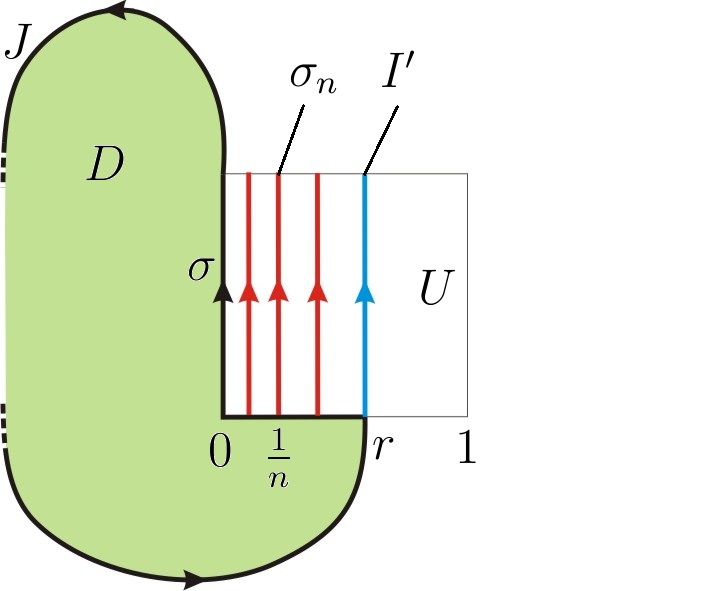}
 \caption{The future of $I'$ in $\g _f$ must enter $D$ by crossing $U$
 pointing down, a contradiction.}
\label{figsquare}
\end{center}
\end{figure}
By definition, $J\cap U=\mbox{\O }$ and $(r,0)$ lies in the boundary of an
interval component $I'$ of $\g _f\cap U$, which in our square model
is represented by $\{ r\} \times [0,1]$. Then, $\a :=\sigma \cup J\cup
([0,r]\times \{ 0\} )$ defines a Jordan curve in $\esf^2$. Let
$D\subset \esf^2$ be the topological disk bounded by $\a $ whose
interior is disjoint from the square $U$.
Consider the embedded  open arc $(I')^+$ of $\g _f$ given by the future of
$I'$. Note that $(I')^+$ must eventually intersect the interior of
$D$ (to see this, observe that for $n$ sufficiently large,
the segment $\sigma _n \equiv \{ 1/n\} \times [0,1]$ lies in $(I')^+$ and
intersects $\partial D$ transversely at its initial point). Since  $\g _f$ has no
self-intersections, we deduce that $(I')^+$ must cross $\partial D$ along
the arc $(0,r)\times \{ 0\} $. Moreover, at the points of $(I')^+\cap \partial D $
along the arc $(0,r)\times \{ 0\} $ where $(I')^+$ enters $D$,
the induced orientation of $(I')^+$ points down, which contradicts our assumption that every
component of $\g _f\cap U\equiv A\times [0,1]$ is oriented by
pointing up. If possibility~(M2) above occurs, then a similar
argument leads to a contradiction. This finishes the proof of Lemma~\ref{ass:limit}.
\end{proof}

\begin{lemma}
\label{oldstep1}
Suppose that $\g_f$ is not a closed curve. Let $\sigma, \sigma _n\subset \gamma_f$,
$n\in \N$, be the arcs given by Lemma~\ref{ass:limit}, and assume that $G(p)\in \sigma$
(here $p$ is the base point of $f$ and $G$ is its left invariant Gauss map). Let $D,D_n\subset \Sigma$ be given by
$D:=G^{-1}(\sigma)$, $D_n:=G^{-1}(\sigma_n)$.
Take points $p_n\in D_n$ so that $\{G(p_n)\}_n$ converges to $G(p)$. Then, there exist
diffeomorphisms $\Phi_n\colon D\flecha D_n$ such that the immersions
$l_{f(p_n)^{-1}} \circ f \circ \Phi_n$ converge to $f$ uniformly on compact sets of $D$.

As a consequence, the pointed immersions $l_{f(p_n)^{-1}}\circ f
\colon (\Sigma,p_n)\la (X,e)$ converge uniformly on compact sets to
$f\colon (\Sigma,p)\la (X,e)$.
\end{lemma}

\begin{proof}
As $\gamma_f$ is not a closed curve, $\Sigma$ cannot be diffeomorphic to an annulus, and so by
Proposition~\ref{rank:limits} it is simply connected. Thus, by the arguments described
after Proposition~\ref{propos3.1} and the uniformization theorem,
we can parameterize $f$ as a conformal immersion
$f\colon \cU\subset \C\la X$ where $\cU$ is either $\C$ or a vertical strip
$\{s+it\ | \ a_1<s<a_2\}$, and the
left invariant Gauss map $G\colon \cU\to \S^2$
of $f$ depends solely on $s$. From now on, we will identify $\Sigma $ with $\cU $
and assume that the imaginary part of the base point $p\in \cU$ of $f$ is zero.
We will also consider the stereographically projected Gauss map $g:=\pi \circ G\colon  \cU\flecha \overline{\C}$, where
$\pi\colon \S^2\flecha \overline{\C}$ is the stereographic projection
from the South pole of $\esf^2$ defined in terms of
a previously chosen left invariant orthonormal frame $\{E_1,E_2,E_3\}$ of $X$
(see the paragraph before~\cite[Definition~3.2]{mmpr4}).
By item~2 of Lemma~\ref{lam:gauss}, $s\mapsto g(s)$ is a
complete embedded regular curve in $\overline{\C}$. Moreover, after choosing an
adequate left invariant frame $\{E_1,E_2,E_3\}$, \cite[Theorem~3.7]{mmpr4} implies that $g(s)$ is a solution of a
second order analytic autonomous ODE, see its explicit form in~\cite[equation~(4.3)]{mmpr4}.
In particular, $g(s)$ satisfies property~(D), introduced just before Lemma~\ref{tangin}.

Write $\sigma=g([b_1,b_2])$, and $\sigma_n=g(I_n)$, where $[b_1,b_2]$ and
each $I_n$ are compact arcs in $(a_1,a_2)$; note that $p$ can be then identified
with some point in $[b_1,b_2]$.
Let $\Phi_n(z)= \delta_n z +\mu_n$, where $\delta_n,\mu_n\in \R$ satisfy that
$\Phi_n([b_1,b_2])=I_n$. Then, we can consider the sequence of maps
$g_n:=g\circ \Phi_n\colon [b_1,b_2]\flecha \overline{\C}$. It follows from
property (D) above
and from the $C^1$ convergence of the arcs $\sigma_n$ to
$\sigma$ given by Lemma~\ref{ass:limit} that the following properties hold:
\begin{enumerate}[(N1)]
\item
Each $g_n$ is a solution of the ODE~\cite[equation~(4.3)]{mmpr4}. In fact,
$g_n$ is the Gauss map of $f\circ \Phi_n$ restricted to
$[b_1,b_2]\times \R \subset \C$.
\item $\{ g_n\} _n\to g$ uniformly on $[b_1,b_2]$ in the $C^1$-topology.
\end{enumerate}
Let $t_n\in [b_1,b_2]$ such that $g_n(t_n)$
is the value of the left invariant Gauss map of $f$ at $p_n$.
By hypothesis on the points $p_n$, we have that $\{ t_n\} _n\to
p\in [b_1,b_2]$.

We next translate property (N2) above into the desired
convergence of $l_{f(p_n)^{-1}}\circ f\circ \Phi_n$ to $f$.
To do this, we observe that given $z\in \mathcal{U}$, we have
\begin{equation}
\label{eq:4.12}
f_z(z)=\sum_{i=1}^3 A_i(z) (E_i)_{f(z)},
\end{equation}
where $A_1,A_2,A_3\colon \mathcal{U}\to \C $ are given in terms of $g$
by~\cite[equation~(3.4)]{mmpr4}, and $\{ E_1,E_2,E_3\} $ is the orthonormal basis
of left invariant vector fields on $X$ that we prescribed previously.
Note that the immersion $f\circ \Phi_n$ satisfies
the same complex linear differential equation (\ref{eq:4.12}) with respect to $g_n$.
Using properties~(N1), (N2) above and the smooth dependence of the
solution of an ODE with respect to initial conditions, we deduce
that $l_{f(p_n)^{-1}}\circ f\circ \Phi_n$ converges uniformly on
compact sets of $[b_1,b_2]\times \R $ to $f$.

Finally we prove the last sentence in the statement of Lemma~\ref{oldstep1}.
Let us consider the pointed immersions $l_{f(p_n)^{-1}}\circ f \colon (\Sigma,p_n)\la (X,e)$,
which have uniform curvature estimates by item~5 of Proposition~\ref{propsu2}.
By elliptic theory, these immersions converge
(up to subsequence) to some other complete pointed immersion
of constant mean curvature. Now, the uniform convergence property
that we obtained in the previous paragraph, and
the unique continuation principle for
surfaces of constant mean curvature (see e.g., Aronszajn~\cite{arons1})
concludes the proof of Lemma~\ref{oldstep1}.
\end{proof}

\section{The existence of periodic invariant limit surfaces}
\label{sec:periodic}

In this section we will apply the results of Section~\ref{sec:limits}
to prove Theorem~\ref{g_f closed} below, as well as
its consequence, Corollary~\ref{thm6.11}.
In later sections we will also quote some of the steps in the proof of
Theorem~\ref{g_f closed}.

We will use the following notation along this section.
$X$ is a metric Lie group diffeomorphic to $\R^3$; $h_0(X)\geq 0$,
$\mathcal{C}$ are respectively the number and the component
of the space $\chM_X$ of index-one spheres with constant mean curvature in $X$
that appear in Proposition~\ref{propsu2}.
$f\colon (\Sigma ,p)\la (X,e)$ is a pointed limit immersion of $\mathcal{C}$ produced by
Proposition~\ref{rank:limits}; in particular, $f$ is stable with constant mean curvature $h_0(X)$,
and there exists
a nonzero, right invariant vector field $K_{\Sigma }$ on $X$ which is everywhere tangent to $f(\Sigma)$.
$\G_{\Sigma }\subset X$ will denote the 1-parameter subgroup of $X$ generated by $K_{\Sigma }$.

Observe that if
the space $\Delta(f)$ defined in Lemma~\ref{lam:gauss}
contains an element $f_1\colon (\Sigma _1,p_1)\looparrowright (X,e)$ with constant
left invariant Gauss map (equivalently, $f_1(\Sigma _1)$ is a two-dimensional subgroup of $X$),
then $f_1$ is injective and $f_1(\Sigma _1)$ is an entire Killing graph in $X$ with respect to any
nonzero right invariant vector field in $X$ transverse to $f_1(\Sigma_1)$.
In particular, Theorem~\ref{topR3case} holds trivially in this case
(note that by item 1 of Proposition~\ref{rank:limits}, $f_1$ is a limit surface
of $\mathcal{C}$, and so it has the convergence properties in the statement of
Theorem~\ref{topR3case}).
Therefore,
{\bf in the sequel we will assume that
no element in $\Delta(f)$ has constant left invariant Gauss map;}
in particular, $f(\Sigma )$ is not a two-dimensional subgroup of $X$.

By Proposition~\ref{rank:limits},
 $\Sigma$ is either simply connected (with $f(\Sigma )$ being a plane or an annulus)
or $\Sigma $ is diffeomorphic to an annulus (so $f(\Sigma )$ is an annulus as well).
By Lemma~\ref{lam:gauss}, the left invariant Gauss map image $\g _f=G(\Sigma )$ of $f$
is an embedded regular curve in $\esf^2\subset T_eX$, and we can choose $f$ so that the closure
$\overline{\g _f}$ is a lamination of $\S^2$ with no proper sublaminations.

\begin{theorem}
\label{g_f closed}
In the above situation, $\g_f$ is a closed curve in $\S^2$.
\end{theorem}
\begin{proof}
As $\g_f$ is not a point, then we can consider on $\g _f$ the orientation given by Definition~\ref{def4.7}.
We will divide the proof of the theorem into three assertions for the sake of clarity.
\par
\vspace{0.2cm}

\begin{assertion}
\label{ass6.2}
$\g _f$ is a closed curve if one of the following two conditions hold:
\begin{enumerate}[(1)]
\item $X$ is isomorphic to $\R^2\rtimes _A\R $ for some matrix $A\in
\mathcal{M}_2(\R )$ and $K_{\Sigma }\notin \mbox{\rm Span}\{ F_1,F_2\} $ where
$F_1,F_2$ are given by equation~(\ref{eq:6}).
\item $X$ is isomorphic to $\sl $ and $K_{\Sigma }$ is either elliptic or
parabolic.
\end{enumerate}
\end{assertion}
\begin{proof}[Proof of Assertion~\ref{ass6.2}]
Arguing by contradiction, suppose that $\g _f$ is not a simple closed
curve. Let $\sigma$ and $\{ \sigma_n\} _n$ be the arcs given in
Lemma~\ref{ass:limit}, where $q:=G(p)$ lies in $\sigma $. Let
$p_n$ be points of $\Sigma$ such that the sequence $G(p_n)\in \sigma_n$
 converges to $q$ as $n\to \infty $. By Lemma~\ref{oldstep1},
the pointed immersions $l_{f(p_n)^{-1}} \circ f \colon (\Sigma ,p_n)\looparrowright (X,e)$ converge
uniformly on compact sets to $f\colon (\Sigma,p)\looparrowright (X,e)$. As $K_{\Sigma }$ is
everywhere tangent to $f(\Sigma )$, then the nonzero
vector field $(l_{f(p_n)^{-1}} )_*(K_{\Sigma })$ is everywhere
tangent to $(l_{f(p_n)^{-1}} \circ f)(\Sigma )$.
Observe that $(l_{f(p_n)^{-1}} )_*(K_{\Sigma })$ is right
invariant by the first paragraph of the proof of Lemma~\ref{lemma3.6}.
After normalizing, we deduce
that the sequence of right invariant vector fields
\begin{equation}\label{kan}
K(n)= \frac{1}{|
[(l_{f(p_n)^{-1}} )_*(K_{\Sigma})](e)|}(l_{f(p_n)^{-1}} )_*(K_{\Sigma})
\end{equation}
converges to a nonzero right invariant vector field $K_1$ on $X$
which is everywhere tangent to $f(\Sigma )$. Since $f(\Sigma )$ is not
a two-dimensional subgroup of $X$, then $K_1$ must be a (constant)
nonzero multiple of $K_{\Sigma }$. As the orientations of the arcs
$\sigma _n$ are all opposite to the orientation of the arc $\sigma
$, by definition of the orientation of $\g _f$ we deduce that $K_1$
is a {\it negative} multiple of $K_{\Sigma}$.

We now check that the last property cannot happen under the hypotheses of
Assertion~\ref{ass6.2}. First suppose that we are in case~1 of this assertion. Since
$K_{\Sigma}\notin \mbox{\rm Span}\{ F_1,F_2\} $, then we may assume that
\[
K_{\Sigma}=F_3+c_1\partial_x +c_2\partial_y
\]
 for some constants $c_1,c_2\in
\R$; here $F_3$ is given by (\ref{eq:6}).
Plugging (\ref{eq:6}) into this last equation we get
$K_{\Sigma}=\partial _z+h_1(x,y)\partial_x +h_2(x,y)\partial_y$ for some
functions $h_1,h_2$.
Since $\partial_z$ is left
invariant and $(l_{a})_*$ preserves $\mbox{\rm Span}\{ \partial _x,
\partial _y\} $ for every $a\in X$, then we conclude that $K(n)$ cannot converge to a
negative multiple of $K_{\Sigma}$, which is the desired
contradiction if case~1 holds.

In case 2, we can use similar arguments to find a contradiction, using
the last sentence of Lemma~\ref{lemma3.6}.
This finishes the proof of Assertion~\ref{ass6.2}.
\end{proof}

\begin{assertion}
\label{ass6.3}
If $X=\R^2\rtimes _A\R $ and $K_{\Sigma }\in \mbox{\rm Span}\{ F_1,F_2\} $,
then $\g_f$ is a closed curve. Moreover, if
$\g_f$ does not pass through the North or South pole, then $f(\Sigma
)$ is an entire Killing graph with respect to any horizontal right
invariant vector field linearly independent from $K_{\Sigma }$.
\end{assertion}
\begin{proof}[Proof of Assertion~\ref{ass6.3}]
As $K_{\Sigma }\in \mbox{\rm Span}\{ F_1,F_2\} $, equation~(\ref{eq:6}) implies
that $K_{\Sigma }$ is a linear combination $\l \partial _x+\mu \partial _y$
for some $\l ,\mu \in \R $. 
Since $f(\Sigma )$ is foliated by integral curves of $K_{\Sigma }$, then $f(\Sigma )$
is ruled by straight lines in the direction of
$\l \partial _x+\mu \partial _y$.
There are two cases to consider:
\begin{enumerate}[(O1)]
  \item The angles between the immersed surface $f(\Sigma )$
  and the planes $\R^2\rtimes _A\{ z\} $, $z\in \R $,
are bounded away from zero.
Equivalently, $\overline{\gamma_f}$ does not contain the North or South poles of $\S^2$.

\item There exists a sequence $\{ q_n\} _n\subset \Sigma $
such that the angle between $f(\Sigma )$ and $\R^2\rtimes _A\{
z(f(q_n))\} $ at $f(q_n)$ tends to zero as $n\to \infty $.
Equivalently, the North pole or the South pole in $\S^2$ is in $\overline{\gamma_f}$.
\end{enumerate}

Suppose case (O1) holds. Note that the length of the pullback $\nabla
(z\circ f)$ by $f$ of the tangential component of $\partial _z$ to $\Sigma $
is bounded away from zero. Without loss of generality, we can assume that
$K_{\Sigma }=\partial_x$ (this is just a linear change of the coordinates
$x,y$, which does not affect to our hypotheses). Hence, the image through
$f$ of the integral curve of $\nabla (z\circ f)$ passing through the base point $p$
is a curve in $\R^2\rtimes _A\R $ that passes through the origin and
intersects every horizontal plane $\R^2\rtimes _A\{ t\} $ at a single
point $Q(t)=(x(t),y(t), t)$. Furthermore, the intersection of $f(\Sigma )$
with $\R^2\rtimes _A\{ t\} $ is the straight line passing through
$Q(t)$ in the
direction of $\partial _x$. Hence, $f(\Sigma )$ is an entire Killing
graph with respect to any horizontal right invariant vector field
$V=a \parc_x + b\parc_y$ linearly independent from $\partial _x$.

We now prove that $\g _f$ is a simple closed curve by using
arguments that are similar to the ones in the proof of the first item of Assertion~\ref{ass6.2}.
Consider the cross product  $K_{\Sigma} \times N$, which defines a nowhere zero
tangent vector field to $f(\Sigma )$ (here $N$ is the unit normal of $f$).
Thus, the tangent vector field $W$ to $\Sigma$ given by
pulling back $K_{\Sigma} \times N$ through $f$
is never tangent to the foliation
$\{ \R^2\rtimes _A\{ z\} \ | \ z\in \R \} $. By connectedness
of $\Sigma$, the function $\esiz W, \nabla (z\circ f)\esde$ has constant
(nonzero) sign on $\Sigma$.
Observe that this sign does not change when we compose
$f$ with a left translation in $X$ and do the same process with the translated surface.
Assume $\g _f$ is not a simple closed curve and consider for each $n\in \N$
the right invariant vector field $K(n)$ given by \eqref{kan},
which in this case lies
in $\mbox{\rm Span}\{ \partial _x,\partial _y\} $ (because
$(l_{a})_*$ preserves $\mbox{\rm Span}\{ \partial _x,\partial _y\} $ for every $a\in X$).
As before, the $K(n)$ converge to a negative
multiple of $K_{\Sigma}$, and so the vector fields
$W_n$ obtained by pulling back $K(n)\times N_n$ through $l_{f(p_n)^{-1}}\circ f$
converge to a negative multiple of $W$ (here, $N_n$ denotes the unit normal
of $l_{f(p_n)^{-1}} \circ f$). This contradicts the previously explained
invariance of the sign of the function
$\esiz W, \nabla (z\circ f)\esde$ under left translations. Hence, $\g_f$ is a simple closed
curve, what proves Assertion~\ref{ass6.3} if case~(O1) holds.

In order to analyze case (O2), we will prove the following general
property to be used later on.

\begin{claim}[Slab Property in $\R^2\rtimes _A\R$]
\label{assslabprop}
Suppose $X=\R^2\rtimes _A\R $ and
$K_{\Sigma }\in \mbox{\rm Span}\{ F_1,F_2\} $.
Let $\wt{K}_{\Sigma}=f^*(K_{\Sigma })$ be the
Killing field on $\Sigma$ induced from the right invariant vector
field $K_{\Sigma}$. If there exists a sequence $\{ q_n\} _n\subset \Sigma $
such that the angle between $f(\Sigma )$ and
$\R^2\rtimes _A\{ z(f(q_n))\} $ at $f(q_n)$ tends to zero as $n\to
\infty $ (i.e., case (O2) holds), then:
 \begin{enumerate}[(1)]
\item $f(\Sigma )$ is contained in a smallest horizontal slab $\R^2\rtimes_A
[a,b]$ with $a\leq 0\leq b$, $a\neq b$, and $f(\Sigma)$ is tangent to both $\R^2\rtimes_A \{a\}$ and $ \R^2\rtimes_A\{b\}$.
 \item If $q\in \Sigma$ satisfies that
$f(\Sigma )$ is tangent to $\R^2\rtimes_A \{z(f(q))\}$,
then $z(f(q))\in \{ a,b\} $.
 \item If $\Sigma$ is simply
connected, then the points $q_n$ can be taken to lie on distinct
integral curves of $\wt{K}_{\Sigma}$ such that
\[
z(f(q_n))=\left\{ \begin{array}{cl}
 a & \mbox{ if $n$ is odd,}
 \\
 b & \mbox{ if $n$ is even.}
\end{array}
\right.
\]
\item  $\g_f$ is a simple closed curve that passes through the North
and South poles of $\S^2$.
\end{enumerate}
\end{claim}
\begin{proof}[Proof of Claim~\ref{assslabprop}]
By the uniform curvature estimates in item 5 of Proposition~\ref{propsu2},
after passing to a subsequence, the sequence of left translated
pointed immersions
$l_{f(q_n)^{-1}}\circ f\colon (\Sigma ,p)\la (X,e)$ converges to
a pointed immersion $[\hat{f}\colon (\hat{\Sigma},\hat{p} )\looparrowright (X,e)]
\in \Delta(f)$. Since the angle between $f (\Sigma )$ and $\R^2 \rtimes_A \{z(f(q_n))\} $ at
$f(q_n)$ is not bounded away from zero, then $\hat{f}(\hat{\Sigma
})$ is tangent to $\R^2\rtimes _A\{ 0\} $ at $e=(0,0,0)$. By
Corollary~\ref{cor:unique}, $\wh{f}(\hat{\Sigma })$ is contained in
$\R^2\rtimes _A(-\infty ,0]$ or $\R^2\rtimes _A[0, \infty )$. From
now on, we will assume that $\wh{f}(\wh{\Sigma })\subset
\R^2\rtimes _A[0,\infty )$ (this does not affect the arguments that follow).

Note that $\hat{f}$ is invariant under a nonzero right invariant vector field
$K_{\hat{\Sigma}}$, which is the limit as $n\to \infty $ of the right invariant vector fields
$K(n)$ given by equation~(\ref{kan}) after replacing $p_n$ by $q_n$;
in particular, both $K(n)$ and $K_{\hat{\Sigma }}$ lie in
$\mbox{\rm Span}\{ \partial _x,\partial _y\} $.
Without loss of generality, we can assume that $K_{\hat{\Sigma}}=\partial_x$ .

Therefore, $\hat{f}(\hat{\Sigma} )$ is ruled by straight lines in
the direction of $\partial _x$. If $\hat{f}(\hat{\Sigma
})=\R^2\rtimes _A\{ 0\} $, then $\hat{f}$ has constant left invariant Gauss map,
which contradicts our assumption stated before Theorem~\ref{g_f closed}.
Since $\hat{f}(\hat{\Sigma })$ is analytic
and $\hat{f}(\hat{\Sigma })\neq \R^2\rtimes _A\{ 0\} $, then there
exists $\ve >0$ such that if $B_{\hat{\Sigma }}(\hat{p},\ve )$ denotes the intrinsic
closed disk in $\wh{\Sigma }$ centered at the base point $\hat{p}$ with radius
$\ve$, then $D=\hat{f}\left( B_{\hat{\Sigma }}(\hat{p},\ve )\right) $
is a small graphical disk over its vertical projection to
$\R^2\rtimes _A\{ 0\} $, $D$ is foliated by line segments in the
direction of $\partial _x$, exactly one of these line segments lies
in the $x$-axis, and all the other ones lie in
$\R^2\rtimes _A(0,\infty )$.

As $l_{f(q_n)^{-1}}\circ f$ converges to $\widehat{f}$,
then the sequence of disks $D_n=(l_{f (q_n)^{-1}}\circ f )(B_{\Sigma  }(q_n,\ve ))$
converges smoothly to $D$, where $B_{\Sigma  }(q_n,\ve )$ stands for the intrinsic
closed disk in $\Sigma $ centered at $q_n$ with radius $\ve $.
Consider the integral curves of $K(n)$ (which are
horizontal straight lines not necessarily parallel to the $x$-axis)
which intersect $D_n$. For each $n$, this set of
parallel lines forms a smooth surface $D'_n$
contained in $(l_{f (q_n)^{-1}}\circ f) (\Sigma)$, with boundary
consisting of two horizontal lines, and the $D_n'$ converge smoothly on
compact sets to the ruled surface $D'\subset \wh{f}(\wh{\Sigma })$ consisting of the
integral curves of $\partial _x$ that intersect $D$; note that $D'$
contains the $x$-axis and its boundary consists of two lines
that both lie in $\R^2\rtimes _A(0,\infty )$. In
particular, for $n$ large, the boundary curves of $D'_n$ lie
strictly above the lowest straight line in $D'_n$. Therefore,
along this lowest straight line in $D_n'$, the tangent plane to this surface is horizontal.
Hence, $f(\Sigma )$ is also tangent to some
horizontal plane and by Corollary~\ref{cor:unique} we see that $f(\Sigma)$
lies above this plane.

After possibly replacing $f$ by a left translation of it and
changing the base point $p$, we may assume that $f(\Sigma)$ is
tangent to the plane $P_0=\R^2\rtimes_A \{0\}$ at $f(p)=(0,0,0)$ and it lies
above $P_0$.

By analyticity of $f$, we can choose an open strip
$\mathcal{S}\subset \Sigma $ containing $p$ and such that:
\begin{itemize}
\item $\mathcal{S}$ is invariant under the flow of $\wt{K}_{\Sigma}$.
\item  $\nabla (z\circ f)$ only vanishes in $\mathcal{S}$ along the integral curve $\be_1$ of $\wt{K}_{\Sigma}$
that passes through $p$.
\item $z\circ f$ attains  the same value at the two boundary components of $\mathcal{S}$.
\end{itemize}
Let $C_1\subset \Sigma$ be one of the (at most two) components of $\Sigma -\mathcal{S}$, and let $P'$ be the
horizontal plane that contains $f(\partial C_1)$.
We claim
that  the angle that $f(C_1)$ makes with the
foliation by horizontal planes is not bounded away from zero: otherwise there
exists a sequence of points $r_n\in \Sigma $ such that $|(z\circ
f)(r_n)|\to \infty $ as $n\to \infty $. After left translating $f$
by $f(r_n)^{-1}$ and extracting a subsequence, we find a new limit
$f'\in \Delta(f)$ such that its image makes angles bounded
away from zero with the foliation of horizontal planes $\{ \R^2
\rtimes_A \{t\}\ | \ t\in \R\} $. Since $\g _f$ is a leaf of the lamination
$\overline{\g _f}$ and $\overline{\g _f}=\overline{\g _{f'}}$
(recall that we were assuming that $\overline{\g_f}$ has no proper sublaminations),
 then item~2 of Lemma~\ref{lam:gauss} implies that $f$ is also a limit of $f'$ under
left translations. This clearly contradicts that $f(\Sigma )$ is
tangent to $P_0$. Hence, our claim is proved.

By repeating the above arguments, there is an integral curve $\be_2\subset C_1$
of $\wt{K}_{\Sigma}$ such that $\nabla (z\circ f)$  vanishes along $\be_2$ and
being the closest such integral curve to $\be_1$ in $C_1$; let
$b:=(z\circ f)(\be_2) > 0$. As before, Corollary~\ref{cor:unique}
implies that $f(\Sigma)$ is contained on one side of the plane
$P_b=\R^2\rtimes_A \{b\}$; by connectedness of $f(\Sigma )$, this side
must be $\R^2\rtimes_A (-\infty ,b]$. This implies that the immersed surface $f(\Sigma)$
is contained in the slab between the planes $P_0$ and $P_b$; also,
the strip $Y_1\subset \Sigma$ between $\be_1$ and $\be_2$ has an embedded
image contained in that slab (because $z\circ f$ has no
 critical points in the interior of $Y_1$). Repeating these arguments
in the domain $C_2:=C_1-Y_1$, one finds an integral curve $\be_3$ of
$\wt{K}_{\Sigma}$ in $C_2$ closest to $\be_2$ where $z\circ f$ has a
critical value and $f(\Sigma)$ lies on one side of the corresponding
horizontal plane. Again by connectedness of $f(\Sigma)$, this critical value must be
0. Continuing inductively with this argument, in the case that
$\Sigma$ is simply connected, one produces the desired sequence of
points $\{q_n\}_{n\in \N }$, where $q_n \in \G_n$ and for $n$ odd
(resp. even) one has the related  integral curves $\be_n$ contained in
$\R^2\rtimes _A\{ 0\} $ (resp. $\R^2\rtimes _A\{ b\} $).
The remainder of the statements listed in Claim~\ref{assslabprop} follow
immediately from this discussion.
\end{proof}
Finally, Assertion~\ref{ass6.3} follows from the Slab Property (Claim~\ref{assslabprop})
and our previous discussion of case (O1).
\end{proof}

In the last step of the proof of Theorem~\ref{g_f closed} we will study the case in
which $X$ is isomorphic to $\sl$ and $K_{\Sigma }$ is hyperbolic, or equivalently, the
1-parameter subgroup $\G_{\Sigma }$ of $\sl $ generated by $K_{\Sigma }$ is hyperbolic.
Recall from Section~\ref{subsecsl} that the image of $\G_{\Sigma}$ through the projection
$\Pi \colon \sl \to \Hip ^2$ defined by (\ref{eq:PI}) is an arc of constant geodesic
curvature (possibly zero) in the standard metric of $\Hip ^2$, that
passes through the origin of the Poincar\'e disk
and its two extrema $\t _1,\t _2\in \partial _{\infty }\Hip ^2$ determine the two
two-dimensional subgroups $\Hip ^2_{\t _1}$, $\Hip ^2_{\t _2}$ of $\sl $
that contain $\G_{\Sigma }$.
 We will also use in the next statement the pair
of antipodal, simple closed curves
$\Upsilon \cup (-\Upsilon )\subset \S^2$ defined in Lemma~\ref{lemma5.3}.

\begin{assertion}
\label{ass6.5}
If $X$ is isomorphic to $\sl$ and $K_{\Sigma }$
is hyperbolic, then $\g_f$ is a simple closed curve. Moreover, if
$\g_f$ does not intersect $\Upsilon \cup (-\Upsilon )$, and $\H$ is either
$\H^2_{\theta_1}$ or $\H^2_{\theta_2}$ (where $K_\Sigma$ is tangent to both
$\H^2_{\theta_1}$ and $\H^2_{\theta_2}$), then $f(\Sigma )$ is an entire Killing graph with
respect to the parabolic right invariant vector field which is
tangent at $e$ to $\H$.
\end{assertion}
\begin{proof}[Proof of Assertion~\ref{ass6.5}]
We denote by $\G_3$ the 1-parameter elliptic
subgroup of $X$ defined by $\G_3'(0)=(E_3)_e$ (see equation~\ref{eq:lframesl})).
Let $\H $ be one of the subgroups $\H^2_{\theta_1}$ or $\H^2_{\theta_2}$.
Let $\mathcal{F}=\{ r_t(\Hip)=\Hip\ t\ | \ t\in \sl \} $
be the foliation of $X$ whose leaves are the right cosets of $\Hip$. Since
$\G_3$ intersects $\Hip$ only at the identity element, it is easily seen that
every element of $X$ lies in a unique right coset of $\Hip$ of the form
$\Hip\, t$, for some $t\in \G_3$.

Consider the oriented distance function
$\wt{d}
\colon X \to \R$ to the closed set $\Hip$, which is positive on
$\H \, t$, for $t>0$, where we have naturally parameterized
$\G_3$ by $\R $.
As each right coset of a codimension one connected subgroup of a metric Lie
group  is at a constant distance from the subgroup (see~\cite[Lemma~3.9]{mpe11})
and $\G_3$ is isomorphic to
$\R$, then the level sets of $\wt{d}
$ are the leaves of
$\mathcal{F}$. Hence, we have an induced function $d
\colon \cF \to \R$ so that the absolute value of $d
$ at a given element $\Hip \, t$ of $\cF$ is the constant distance  from $\Hip $
to $\Hip\, t$; thus $\mathcal{F}$ can be parameterized by $\G _3$, i.e.,
$\mathcal{F}=\{ \Hip\, t \ | \  t\in \G_3\}$, and we may assume that
 $\langle \nabla \wt{d}
, E_3\rangle >0$, where here $\nabla $
stands for gradient in $X$.

The approach to prove Assertion~\ref{ass6.5} is to follow, when possible, the
arguments in the proof of Assertion~\ref{ass6.3} by exchanging the
former foliation $\{ \R^2\rtimes _A\{ t\} \mid t\in \R \} $ by the
current one $\mathcal{F}$. Notice that the ambient function $z$ in
Assertion~\ref{ass6.3} is now replaced by the oriented distance function
$\wt{d}
\colon X \to \R$. For the proof, it is also important
to take into account the following list of elementary observations:
\ben[(P1)]
\item Each of the right cosets $\Hip\, t$  with $t\in \R \equiv \G_3$
is a  left coset of some two-dimensional subgroup $\Hip _t$ which is
a conjugate subgroup of $\Hip$;  see Section~\ref{backgroundCMC}.
If $t\in \G _3$ lies in the center of $X$ (which is isomorphic to $\Z $), then $\Hip\, t
=t\, \Hip$. In this sense, $\mathcal{F}$ is a {\it periodic}
foliation invariant under the left action of $\Z =\mbox{center} ( X
)\subset \G_3$.

\item Limits of $f$ after left translations
by the inverses of points on the image surface $f(\Sigma)$ give rise
to limit surfaces $\wh{f}(\wh{\Sigma})$ that are invariant under
some other \emph{hyperbolic} right invariant vector field and such
that $\g_{\wh{f}}$ is a leaf of $\ov{\g_f}$.  Otherwise,
Lemma~\ref{lemma3.6} implies that
$\wh{f}(\wh{\Sigma })$
would be tangent to a right invariant parabolic vector
field (because the limit of a normalized sequence of
hyperbolic right invariant vector fields is either hyperbolic or parabolic).
As the set of parabolic vector fields normalized to have
length 1 at $e$ is compact, then every element in $\Delta (\wh{f})$
is tangent to a parabolic right invariant vector field. As $\g_f$ is a leaf of
the lamination $\ov{\g _f}= \ov{\g_{\wh{f}}}$ (here we are using
$\overline{\g _f}$ has no proper sublaminations),
then $f$ is itself a limit of
$\wh{f}$, i.e., $f\in \Delta(\wh{f})$, and so $f$ is tangent to a right invariant parabolic
vector field. This is a contradiction since $K_{\Sigma }$ is
hyperbolic and $f(\Sigma )$ is not a two-dimensional subgroup of $X$.

\item Given $a\in X$, the family $a\, \mathcal{F}=\{ a\, L=l_a(L)\ | \ L\in \mathcal{F}\} $
is a foliation of $X$ by the right cosets of some conjugate subgroup of $\Hip $.

\item As the foliation $\mathcal{F}$ is periodic in the sense of item~(P1),
then for every sequence $\{ a_n\} _n\subset X$, the sequence
of foliations $\{ a_n\, \mathcal{F}\} _{n\in \N}$ defined as in (P3)
has a convergent subsequence to another foliation of $X$ by the right
cosets of some conjugate subgroup of $\Hip$.

\item Let $\G ^P$ be the unique 1-parameter parabolic subgroup
 contained in $\Hip$. As $\G _{\Sigma }\, \Hip=\Hip$ and $\G ^P\, \Hip=\Hip $, we deduce that
the left actions of $\G _{\Sigma},\G^P$ on $X$ leave invariant each of the right cosets of $\Hip$.
In particular, $K_{\Sigma}$ and the parabolic right
invariant vector field $P$ on $X$ generated by $\Gamma ^P$ are everywhere tangent
to each of the leaves of $\mathcal{F}$.
\item
Given a leaf $\Hip\, t$ of $\cF$, each integral curve of  $K_{\Sigma}$ restricted to  $\Hip\, t$ intersects
 each integral curve of $P$ restricted to $\Hip\, t$ in exactly one point.
\een

To prove Assertion~\ref{ass6.5}, let us start by assuming that
the Gaussian image $\gamma_f$ is at a positive distance from the set
$\Upsilon \cup (-\Upsilon )$. In particular, the angle between
$f(\Sigma)$ and the leaves of $\cF$ is bounded away from zero.
Using observations (P5) and (P6) above, it follows from simple modifications of the
arguments in the proof of case~(O1) of  Assertion~\ref{ass6.3} that
in these conditions,
$f(\Sigma)$ is an entire Killing graph with respect to $P$, and that
$\gamma_f$ is closed. Details are left to the reader.

\begin{claim}[Slab Property in $\sl$]
\label{assslabprop1}
Suppose $X=\sl$ and $K_{\Sigma }$ is hyperbolic.
Let $\wt{K}_{\Sigma}=f^*(K_{\Sigma })$ be the
Killing field on $\Sigma$ induced by the right invariant vector
field $K_{\Sigma}$.
If there exists a sequence $\{ q_n\} _n\subset \Sigma $
such that the angle between $f(\Sigma )$ and
the leaf of $\mathcal{F}$ passing through $f(q_n)$
tends to zero as $n\to \infty $, then:
 \begin{enumerate}[(1)]
\item $f(\Sigma )$ is contained in a smallest topological slab bounded by two different
leaves $L_a,L_b$ of $\mathcal{F}$,  and $f(\Sigma)$ is tangent to both leaves.
 \item If $f(\Sigma )$ is tangent to a leaf $L$ of $\mathcal{F}$, then
 $L=L_a$ or $L=L_b$.
\item $\g_f$ is a simple closed curve that intersects both $\Upsilon$ and $-\Upsilon $.
\end{enumerate}
\end{claim}
\begin{proof}[Proof of Claim~\ref{assslabprop1}]
Consider for each $n\in \N$ the left translated pointed immersion
$f_n:=l_{f(q_n)^{-1}}\circ f\colon (\Sigma,q_n)\la (X,e)$, which is
everywhere tangent to the hyperbolic right invariant vector field
$K_n:=(l_{f(q_n)^{-1}})_* (K_{\Sigma})$;  see observation (P2). Note that
$K_n$ is everywhere tangent to the foliation by right cosets $\cF_n:=\{ \H_n \, t \ | \  t\in
\R\}$, where $\H_n$ is the two-dimensional subgroup of $X$ given by
$\H_n=f(q_n)^{-1}\, \H \, f(q_n)$ . After extracting a
subsequence, the $f_n$ converge on compact sets to a pointed limit
immersion $\wh{f}\colon (\wh{\Sigma},\wh{p})\la (X,e)$ and
by observation (P4), $\mathcal{F}_n$ converges
to some foliation $\wh{\mathcal{F}}$ of $X$ by the right cosets of some two-dimensional
subgroup $\wh{\Hip }$.
Observe that by hypothesis, the angle that $f_n(\Sigma )$ makes with $\mathcal{F}_n$
at $f_n(q_n)=e$ tends to zero as $n\to \infty$.
Thus, $\wh{f}(\wh{\Sigma })$ is tangent to $\wh{\H }$ at $
\wh{f}(\wh{p})=e$.
In particular, the left invariant Gauss map of $\wh{f}$ at $\wh{p}$ lies in
$\Upsilon \cup (-\Upsilon )$.
Furthermore, $\wh{f}$ is everywhere tangent to some right invariant hyperbolic
Killing field $\wh{K}$, by observation~(P2) ($\wh{K}$ is the limit
of appropriate rescalings of the $K_n$ after
extracting a subsequence, as we did in \eqref{kan}).  In particular, $\wh{K}$ is everywhere
tangent to $\wh{\H}$, which implies that $\wh{K}$ is generated by a
1-parameter subgroup of $\wh{\H}$. Let
$\wh{d}\colon X\flecha \R$ stand for the oriented distance
function associated to the foliation $\wh{\cF}$. Note that $\wh{d}$ is constant along
each of the integral curves of $\wh{K}$.
By Corollary~\ref{cor:unique}, $\wh{f}(\wh{\Sigma})$ is
contained in one side of $\wh{\H}$, i.e., the restriction of $\wh{d}$ to $\wh{f}(\wh{\Sigma})$ is
nonpositive or nonnegative. Moreover, arguing as in the proof of Claim~\ref{assslabprop}
there exists a small compact topological disk $\wh{D}=
\hat{f}\left( B_{\hat{\Sigma }}(\hat{p},\ve )\right)$
which is foliated by compact arcs of integral curves of $\wh{K}$, such that:
\begin{enumerate}[(Q)]
\item Exactly one of these compact arcs of integral curves of $\wh{K}$ in $\wh{D}$
lies in the zero set of $\wh{d}$, while all the other ones lie in
$\wh{d}^{-1}(0,\infty )$ (or in $\wh{d}^{-1}(-\infty ,0)$).
\end{enumerate}

Since $f_n$ converges to $\widehat{f}$,
then the sequence of disks $D_n:=f_n(B_{\Sigma  }(q_n,\ve ))$
converges smoothly to $\wh{D}$.
As 
$\{ \Hip _n\} _n\to \wh{\Hip }$, then we can assume that
 the distance function $d_n\colon \sl \to \R $ associated to the foliation
$\mathcal{F}_n$ converges on compact sets of $\sl $ to $\wh{d}$.
As $K_n$ is everywhere tangent to the leaves of $\cF_n$,
we conclude that $d_n$ is constant along the integral curves of $K_n$. Furthermore,
item~(Q) and the previous convergence properties ensure that for $n$ large enough,
$D_n$ contains an integral curve segment of $K_n$ along which $d_n$ has a local
extremum. Hence, there are points $r_n\in B_{\Sigma }(q_n,\ve )$ such that $f_n$ is
tangent at $f_n(r_n)$ to some leaf of $\mathcal{F}_n$. Therefore, $f$ is tangent at $f(r_n)$
to some leaf $L_a$ of $\mathcal{F}$. By Corollary~\ref{cor:unique}, we deduce that
$f(\Sigma )$ lies on one side of $L_a$. Once here, and using similar ideas, we can adapt the last four paragraphs
of the proof of Claim~\ref{assslabprop} to our $\sl$ setting, and conclude
that $f(\Sigma )$ is tangent to another
leaf $L_b$ of $\mathcal{F}$, and lies in the topological slab of $\sl$ bounded by $L_a\cup L_b$.
This proves item~1 of Claim~\ref{assslabprop1}, and item~2 is a direct consequence of item 1,
Corollary~\ref{cor:unique} and the connectedness of $\Sigma$ that we
leave to the reader.

As for item~3 of Claim~\ref{assslabprop1}, since the leaves of $\mathcal{F}$ are minimal,
 the mean curvature comparison principle applied to $f(\Sigma )$ and $L_a,
L_b$ implies that the mean curvature vector of $f(\Sigma )$ points towards the interior of the
previously defined topological slab at the tangency points between $f(\Sigma )$ and $L_a\cup L_b$. From here,
a straightforward continuity argument applied to the
 normal field to $\mathcal{F}$ gives that $\g _f$ intersects both $\Upsilon $ and $-\Upsilon $.
If $\Sigma $ is an annulus, then $\g _f$ is clearly a closed curve. Otherwise, $\Sigma $
is simply connected and $f(\Sigma )$ is tangent to $L_a$ along at least two (in fact,
infinitely many) distinct integral curves of $\wt{K}_{\Sigma }=f^*(K_\Sigma)$.
This implies that $\g _f$ takes the same
value at these integral curves of $\wt{K}_{\Sigma }$. As $\g $ is embedded, we now conclude that $\g_f$ is
a closed curve.
\end{proof}

Assertion~\ref{ass6.5} follows from the Slab Property (Claim~\ref{assslabprop1})
and our discussion in the paragraph just before Claim~\ref{assslabprop1}.
\end{proof}

We now complete the proof of Theorem~\ref{g_f closed}.
Since $X$ is not isomorphic to
$\su$, it is either isomorphic to a semidirect product or it is
isomorphic to $\sl$.  By Assertions~\ref{ass6.2}, \ref{ass6.3} and~\ref{ass6.5},
$\gamma_f$ is a simple
closed curve, which completes the proof of Theorem~\ref{g_f closed}.
\end{proof}

\begin{corollary}
\label{thm6.11}
Let $f\colon (\Sigma ,p)\looparrowright
(X,e)$ be a pointed limit immersion that satisfies the hypotheses stated just before Theorem~\ref{g_f closed}.
\ben[(1)]
\item If $\Sigma $ is diffeomorphic to an annulus, then it has linear area growth.
\item
\label{item:invariant}
If $\Sigma$ is simply connected, then one of the two following possibilities holds:
\begin{enumerate}[(a)]
\item $f$ factors through a pointed immersion $\widehat{f}\colon
(\widehat{\Sigma},\widehat{p}) \looparrowright (X,e)$ of an annulus. In this
case, $\Sigma$ has quadratic area growth and the constant mean curvature
immersion $\wh{f}$ is stable.
\item $f$ is periodic, in the sense that there exists an element
$a_{\Sigma}\in X-\G_{\Sigma}$ such that the left translation
by $a_{\Sigma}$ leaves $f(\Sigma)$ invariant.
\end{enumerate}
\een
\end{corollary}
\begin{proof}
Recall that $f(\Sigma )$ is foliated by
integral curves of $K_{\Sigma }$. Since $X$ is a simply connected Lie group
diffeomorphic to $\R^3$, then the integral curves of $K_{\Sigma }$ are proper
Jordan arcs.
Let $\wt{K}_{\Sigma}=f^*(K_{\Sigma })$
denote the Killing field induced on $\Sigma$ by $K_{\Sigma}$
after pulling back $K_{\Sigma }$ through $f$.

Assume that $\Sigma $ is diffeomorphic to an annulus and we will
show that item~1 holds. Consider the natural projection $\Pi _{\G_{\Sigma }}\colon
X\to X/\G _{\Sigma }$, where $X/\G _{\Sigma }$ denotes the quotient
surface whose points are the integral curves of $K_{\Sigma }$ in $X$.
Since $\Sigma $ is an annulus, then $(\Pi _{\G_{\Sigma }} \circ f)(\Sigma )$ is an immersed closed
curve $\be \subset X/\G _{\Sigma }$ and $f\colon \Sigma \to \be $ is a
trivial $\R$-bundle over $\be $, whose fibers are the integral curves
of $\wt{K}_{\Sigma }$. Take a global
section of this trivial bundle. The image set of this global section
is an embedded closed curve $\alfa \subset \Sigma $ which intersects transversely each of the
integral curves of $\wt{K}_{\Sigma }$. Since $\wt{K}_{\Sigma }$ is
a Killing vector field, then the image $\psi _t(\alfa )$ of $\alfa $ through
the flow $\{ \psi _t\ | \ t\in \R \} $ by isometries
of $\Sigma $ associated to $\wt{K}_{\Sigma }$
produces a foliation of $\Sigma $ by curves isometric to $\alfa $, so in particular
all of these curves $\psi _t(\alfa)$ have the same length. From here it
is straightforward to show that $\Sigma $ has linear area growth,
so item~1 is proved.

Next suppose that $\Sigma$ is simply connected. Consider
the quotient space $\Sigma / \wt{K}_{\Sigma} $ of integral
curves of $\wt{K}_{\Sigma}$. Thus, $\Sigma / \wt{K}_{\Sigma} $ is a 1-dimensional
manifold diffeomorphic to $\R$, and there is a natural
projection $\Pi\colon\Sigma \to \Sigma / \wt{K}_{\Sigma} $.
We can identify $\Sigma / \wt{K}_{\Sigma} =\R $,
so that $\Pi(p)=0$ for the base point $p$ of
$\Sigma$ satisfying $f(p)=e$. Let $\wh{G}\colon \R=\Sigma /
\wt{K}_{\Sigma} \to \gamma_f \subset \esf^2$ be the mapping such
that $G=\wh{G}\circ \Pi$, where $G$ stands for the left invariant Gauss map
of $f$. Since $\g _f$ is a simple closed regular curve by Theorem~\ref{g_f closed},
it is easy to check that the map $\wh{G}$ is the
universal cover of $\gamma_f$. Let
$\tau\colon \Sigma / \wt{K}_{\Sigma}\to \Sigma
/ \wt{K}_{\Sigma}$ be one of the two generators of the
group of automorphisms of the covering
$\wh{G}\colon \Sigma /\wt{K}_{\Sigma} \to \gamma_f \subset \esf^2$.
Since the left invariant Gauss map
determines the surface up to left translations by~\cite[Theorem~3.7]{mmpr4},
given a point $q\in \Pi^{-1}(\tau(0))$,
the left translation $l_{f(q)}\colon X \to X$ induces a
nontrivial isometry $\wt{l_{q}}\colon \Sigma \to \Sigma$ that is a
lift of $\tau$ to $\Sigma$, i.e., $\tau \circ \Pi =\Pi \circ \wt{l_{q}}$,
and satisfies $\wt{l_{q}}(p)=q$.
Note that $\wt{l_{q}}$
acts freely on $\Sigma$ and the quotient $\wh{\Sigma}=\Sigma / \wt{l_{q}}$
is an annulus. We now have two possibilities.
\begin{enumerate}[(R1)]
\item $f(\Pi^{-1}(0))=f(\Pi^{-1}(\tau (0)))$, in which case
$f$ factors through the quotient  annulus
$\wh{\Sigma}$ to an immersion $\wh{f}\colon \wh{\Sigma}\la X$.
\item $f(\Pi^{-1}(0))\neq f(\Pi^{-1}(\tau (0)))$, in which case
the integral curves $f(\Pi^{-1}(0))$, $f(\Pi^{-1}(\tau(0)))$
of $K_{\Sigma }$ are disjoint.
\end{enumerate}
In any of the two possibilities above, the isometry group Iso$(\Sigma)$ of the induced
 metric on $\Sigma $ by $f$ contains an $\R $-type subgroup
 (which corresponds to left translations by any element in $\G _{\Sigma }$) and a
$\Z $-type subgroup (which corresponds to the subgroup generated by $\wt{l_{q}}$),
such that these two subgroups
generate a subgroup $\Delta $ of Iso$(\Sigma)$, although
$\Delta $ is not necessarily isomorphic to $\R \times \Z $,
because elements in the $\R$-type subgroup do not necessarily
commute with those in the $\Z$-subgroup; also note that
$\Sigma $ would have at most quadratic area growth provided that
$\Delta $ is isomorphic to $\R \times \Z $.

If case~(R1) holds, then the quotient metric on $\wh{\Sigma }$ (which is the induced
metric by $\wh{f}$) has linear area growth,
and so $\Sigma$ has quadratic area growth. Since $f$ is stable and the cover
 $\Sigma \to \wh{\Sigma }$ is cyclic, then the quotient
immersion $\wh{f}\colon \wh{\Sigma }\la X$ is also stable (see
  for instance~\cite[Proposition~2.5]{mpr19}), which proves that item~2a
  in the statement of Corollary~\ref{thm6.11} holds. Finally, if case~(R2) holds,
  then we define $a_{\Sigma }=f(q)$, which lies in $X-\G_{\Sigma }$ since
$f(p)=e\in \G _{\Sigma }$ and the integral curves
$f(\Pi^{-1}(0))$, $f(\Pi^{-1}(\tau(0)))$
of $K_{\Sigma }$ are disjoint. Then, the left translation by $a_{\Sigma }$
leaves $f(\Sigma )$ invariant and we have item~2b of the corollary. Now the proof is
complete.
\end{proof}

\section{Proof of Theorem~\ref{topR3case} when $X$ is a semidirect product}
\label{sec:semidirect}
In this section we will prove Theorem~\ref{topR3case} in the case that $X$ is a semidirect product
$\R^2\rtimes_A \R$ endowed with its canonical metric.

Let $S_n$ be a sequence of  constant mean curvature spheres in $X$ with ${\rm Area}(S_n)>n$ for all $n$. In order to
prove Theorem~\ref{topR3case} in $X$ it suffices to consider the
case where all the spheres $S_n$ lie in the connected component $\mathcal{C}$ of the space of index-one
spheres in $X$ with $H>h_0(X)$ given in Proposition~\ref{propsu2}.
Indeed, if Theorem~\ref{topR3case} holds for any such sequence of spheres $(S_n)_n\subset \mathcal{C}$,
the discussion after Definition~\ref{defH(X)} ensures that any $H$-sphere $\Sigma$ in $X$ satisfies that
$H>h_0(X)$, and so Theorem~\ref{topR3case} holds in full generality.

So, assuming these conditions (in particular, that $S_n\in \mathcal{C}$ for every $n$),
Proposition~\ref{rank:limits}, Theorem~\ref{g_f closed} and Corollary~\ref{thm6.11}
imply that there exists a pointed limit immersion
$f\colon (\Sigma,p) \la (X,e)$ of the spheres $S_n$ 
with the following properties:
\begin{enumerate}[(S1)]
\item $f$ is complete, has constant mean curvature $h_0(X)$, it is stable and its
left invariant Gauss map image $\g_f$ is either a point or a regular, simple closed curve in $\esf^2\subset T_eX$.
 \item
$f$ is everywhere tangent to a nonzero, right invariant Killing vector field $K_{\Sigma}$ in $X$. Moreover,
exactly one of the following three situations happens:
\begin{enumerate}[(S2.1)]
\item $\Sigma $ is diffeomorphic to an annulus, and it has linear area growth.
\item $\Sigma $ is simply connected and $f\colon \Sigma \la X$ factors through an
immersion $\wh{f}\colon \wh{\Sigma }\la X$ of an annulus. In this case, $\Sigma $ has
quadratic area growth and $\wh{f}$ is a stable immersion with constant mean curvature.
\item There exists $a\in X-\G_{\Sigma} $ such that the left translation by $a$ leaves
$f(\Sigma )$ invariant. Here, $\G_{\Sigma}$ is the $1$-parameter subgroup of $X$ that generates $K_{\Sigma}$.
\end{enumerate}
\end{enumerate}

Recall that if $\gamma_f$ is a point, then $f(\Sigma)$ is a two-dimensional subgroup of $X$.
In particular $f(\Sigma)$ is an entire Killing graph with respect to some right invariant vector field in $X$,
and so Theorem~\ref{topR3case} holds. Hence, {\bf from now on we will assume that $f(\Sigma)$ is not
a two-dimensional subgroup}; in particular, $\gamma_f$ is a regular Jordan curve in $\S^2$,
and $K_{\Sigma}$ is unique (up to scaling) among right invariant vector fields on $X$
everywhere tangent to $f(\Sigma )$.

First assume  that the matrix $A$ is singular. By Lemma~\ref{lem:asin}, $X$ is isometric to $\R^3$
or to some $\E (\kappa,\tau )$ space with $\kappa \leq 0$.
By the classification theorems of Hopf, Chern and
Abresch-Rosenberg cited in the Introduction, constant mean curvature spheres in these
spaces are embedded and rotational, and the moduli space of such spheres can be parameterized analytically
by the constant values of their mean curvatures, which take all possible values in $(H(X),\8)$.
In particular, $h_0(X)=H(X)$, which is the mean curvature of the
subgroup $\R^2\rtimes _A\{ 0\}$ with respect to its upward-pointing unit normal vector.
If $\g _f$ passes through the North or South pole of $\esf^2$, then Claim~\ref{assslabprop}
ensures that $f(\Sigma )$ in contained in a smallest horizontal slab $\R^2\rtimes_A[a,b]$
with $a\leq 0\leq b$, $a\neq b$, and $f(\Sigma)$ is tangent to both $\R^2\rtimes _A\{ a\} $ and $\R^2\rtimes _A\{ b\}$.
By the maximum principle, $f(\Sigma)=\R^2\rtimes_A\{ a\}$, which is a contradiction. Hence,
$\g_f$ does not pass through the  North or South poles of $\esf^2$. As $\g_f$ is a closed curve,
then $\g_f$ is at positive distance from the North and South poles. In this situation, Assertion~\ref{ass6.3}
implies that $f(\Sigma)$ is an entire Killing graph with respect to some nonzero Killing
vector field in $X$. Thus, Theorem~\ref{topR3case} holds in this situation.

So, from now on, we will assume that the matrix $A$ is regular. By the previous discussion, the next
result directly implies that Theorem~\ref{topR3case} holds for any metric semidirect
product $X$, as desired.

\begin{theorem}
\label{semigraph}
In the above conditions, $K_{\Sigma}$ is horizontal, i.e., $K_{\Sigma}\in {\rm Span}\{ F_1,F_2\} $, and
$f(\Sigma)$ is an entire Killing graph in $X$ with respect to
any nonzero right invariant Killing vector field $V\in {\rm Span}\{ F_1,F_2\} $
that is linearly independent from $K_{\Sigma }$
(with the notation of (\ref{eq:6})).
\end{theorem}

Observe that if
Theorem~\ref{semigraph} holds, then case (S2.3) occurs for $f$ (because in cases
(S2.1) and (S2.2) the image surface $f(\Sigma)$ is an immersed annulus in $X$, which
contradicts that $f(\Sigma)$ is an entire graph).

We will divide the proof of Theorem~\ref{semigraph} into four steps.
In Section~\ref{sec:hori} we will prove that Theorem~\ref{semigraph} holds when $K_{\Sigma}\in {\rm Span}\{ F_1,F_2\} $;
thus, Theorem~\ref{semigraph} will be proved provided that we find a contradiction whenever $K_{\Sigma }$
is not horizontal. Assuming that $K_{\Sigma}\notin {\rm Span}\{ F_1,F_2\} $,
in Sections~\ref{sec:F3}, \ref{sec:nos}, \ref{sec:noa}
we will respectively show that cases (S2.3), (S2.2), (S2.1) above cannot occur. This will complete the proof
of Theorem~\ref{semigraph} (and thus of Theorem~\ref{topR3case} for $X$ a metric semidirect product).

\subsection{Proof of Theorem~\ref{semigraph} when $K_{\Sigma}$ is horizontal}
\label{sec:hori}
\begin{lemma}
\label{lem:horizontal}
If $K_{\Sigma }\in {\rm Span}\{ F_1,F_2\} $, then $f$ is an entire Killing graph
with respect to any right invariant vector field $V\in {\rm Span}\{ F_1,F_2\} $
that is linearly independent from $K_{\Sigma }$.
\end{lemma}
\begin{proof}
In this case, $K_{\Sigma }=\l \partial_x +\mu \partial_y$ for some $\l,\mu \in \R$.
Since $f(\Sigma )$ is foliated by integral curves of $K_{\Sigma }$, then
$f(\Sigma )$ is ruled by straight lines in the direction of $K_{\Sigma }$.
By Assertion~\ref{ass6.3},
we conclude that the lemma will hold provided that $\g _f$
does not pass through the North or South pole of $\esf^2$.
Arguing by contradiction, assume that $\g _f$ passes through the North or
South pole of $\esf^2$. Then, $f$ is in the conditions of Claim~\ref{assslabprop}, which
ensures that $f(\Sigma )$ is contained in a smallest slab-type region $\R^2\rtimes_A
[a,b]$ with $a\leq 0\leq b$, $a\neq b$, $f(\Sigma)$ intersects each of the planes $\R^2\rtimes_A
\{a\}$, $\R^2\rtimes_A
\{b\}$ in nonempty sets whose components
are integral curves of $K_{\Sigma}$, and $f(\Sigma )$ transversely
intersects every intermediate plane $\R^2\rtimes _A\{ z\} $ with $a<z<b$.
Let $Y$ be the flat Riemannian product $\R^2\times [a,b]$. Thus, we can consider
$f\colon \Sigma \la Y$ to be a complete
flat immersion, which is ruled by straight lines parallel
to $K_{\Sigma}$. In particular, $\Sigma $ has quadratic area growth
with respect to the pullback metric of $Y$ through $f$. Since the coefficients
of the canonical metric on $\R^2\rtimes _A\R $ with respect to the $(x,y,z)$-coordinates
only depend on $z$ by equation~(\ref{eq:13}), we conclude that
the identity map $\mbox{Id}\colon Y\to \R^2\rtimes_A
[a,b]$ is a quasi-isometry.  From here,
we deduce that
$\Sigma $ has at most quadratic area growth with respect to the
pullback metric of $X$ through $f$.

Let $V$ be any horizontal right invariant vector field on $X$ that is linearly independent from
$K_{\Sigma}$. As the coefficients of the canonical metric are bounded  in every
horizontal slab of finite width by equation (\ref{eq:13}), then $V$ is
bounded in $\R^2\rtimes_A[a,b]$. Denoting by $N$ the unit normal vector to $f$, we deduce that
$J:=\langle N,V \rangle$ is a bounded Jacobi function on $\Sigma $ that changes sign
(changes of sign of $J$ occur at points of $f^{-1}(\R^2\rtimes_A \{ a,b\} )$).
As $f\colon \Sigma \la X$ is a complete stable $h_0(X)$-surface
with at most quadratic area growth, then the conformal structure of $\Sigma $ is parabolic,
and Corollary~1 in Manzano, P\'erez and Rodr\'\i guez~\cite{mper1} implies 
that $J$ cannot change sign, which is a contradiction.
This contradiction completes the proof of the lemma.
\end{proof}

\subsection{Case (S2.3) is impossible when $K_{\Sigma}$ is not horizontal}
\label{sec:F3}
Assume that $K_{\Sigma }\notin {\rm Span}\{ F_1,F_2\} $. The next lemma reduces this case
to the specific nonhorizontal, right invariant vector field
$K_{\Sigma}=F_3$ given by (\ref{eq:6}),
which is generated by the 1-parameter subgroup of $X$ corresponding to the $z$-axis.
In proving the next lemma, we will be making use of the assumption that $A$ is a regular matrix,
as discussed just before the statement of Theorem~\ref{semigraph}.

\begin{lemma}
\label{lem:F3}
If $K_{\Sigma }\notin {\rm Span}\{ F_1,F_2\} $, then after
a horizontal left translation of $f$ and a scaling of $K_{\Sigma}$,
the following statements hold:
\ben[(1)]
\item $K_{\Sigma}=F_3$.
\item Case  (S2.3) does not happen.
\een
\end{lemma}
\begin{proof}
Since $K_{\Sigma}\notin {\rm Span}\{ F_1,F_2\} $, then after scaling and using
(\ref{equationgenA}) and (\ref{eq:6}),
\[
K_{\Sigma}=F_3+s\partial_x +t\partial_y
=(ax+by+s)\partial_x +(cx+dy+t)\partial_y +\partial_z\equiv \left( \begin{array}{c}
A{\bf p}+{\bf q}
\\
1
\end{array}\right) ,
\]
for some $s,t\in \R$, where ${\bf p}=(x,y)$, ${\bf q}=(s,t)$.
Given ${\bf P}_0=({\bf p}_0,0)\in \R^2\rtimes _A\{ 0\} $, the right invariant vector field
$(l_{{\bf P}_0})_*(K_{\Sigma })$ is everywhere tangent to the
left translation of $f$ by ${\bf P}_0$. Hence, to prove item~1 of the lemma
it suffices to find ${\bf p}_0\in \R^2$ such that $(l_{{\bf P}_0})_*(K_{\Sigma })=F_3$.
By direct computation using the group operation in $\R^2\rtimes_A \R$, see Section~\ref{subsecsemdirprod},
we have that $l_{{\bf P}_0}({\bf p},z)=({\bf p}_0+{\bf p},z)$
is an Euclidean translation, hence its differential is the identity after identifying the tangent spaces of $X$
at both points $({\bf p},z)$, $({\bf p}_0+{\bf p},z)$ with $\R^3$ in the basis $\partial _x,\partial _y,\partial_z$.
Therefore,
the value of  $(l_{{\bf P}_0})_*(K_{\Sigma })$ at $l_{{\bf P}_0}({\bf p},z)$ is given by
\[
(dl_{{\bf P}_0})_{\binom{\bf p}{z}}\left( \begin{array}{c}
A{\bf p}+{\bf q}
\\
1
\end{array}\right) =\left( \begin{array}{c}
A{\bf p}+{\bf q}
\\
1
\end{array}\right) ,
\]
while
\begin{equation}
\label{eq:F3a}
(F_3)_{l_{{\bf P}_0}\binom{\bf p}{z}}=\left( \begin{array}{c}
A({\bf p}_0+{\bf p})
\\
1
\end{array}\right) ;
\end{equation}
hence we only need to define ${\bf p}_0$ such that $A{\bf p}_0={\bf q}$,
which can be done since $A$ is regular.
 This proves the first statement in the lemma.

To prove the second assertion, assume that case (S2.3) holds with $K_{\Sigma }=F_3$.
Hence, there exists  $a\in X-\G_{\Sigma}=X-\{ z\mbox{-axis}\} $
 such that the left translation by $a$ leaves
$f(\Sigma )$ invariant. Writing $a=({\bf p}_0,z_0)\in \R^2\rtimes _A\R $ as
the product $({\bf p}_0,0)*({\bf 0},z_0)$ and using that $({\bf 0},z_0)\in \G_{\Sigma }$, we
conclude that $a$ can be chosen so that $z_0=0$. As $l_a$ leaves $f(\Sigma )$ invariant
and $F_3$ is unique up to rescaling among right invariant vector fields
which are nonzero and everywhere tangent to $f(\Sigma )$,
then $(l_a)_*(F_3)=\l \, F_3$ for some $\l \in \R $. But
 the value at $l_{a}({\bf p},z)=({\bf p}_0+{\bf p},z)$
of $(l_{a})_*(F_3)$ is
\[
(dl_{a})_{\binom{\bf p}{z}}\left( \begin{array}{c}
A{\bf p}
\\
1
\end{array}\right) =\left( \begin{array}{c}
A{\bf p}
\\
1
\end{array}\right) ,
\]
while $(F_3)_{l_{a}\binom{\bf p}{z}}$ is given by (\ref{eq:F3a}).
Since $A$ is a regular matrix, then we deduce that $(l_a)_*(F_3)$ cannot be
a multiple of $F_3$. This completes the proof.
\end{proof}

\subsection{Case  (S2.2) is impossible when $K_{\Sigma}$ is not horizontal}
\label{sec:nos}
Assume that $K_{\Sigma }\notin {\rm Span}\{ F_1,F_2\} $
and that case (S2.2) above holds. By item~1 of Lemma~\ref{lem:F3},
we can assume after a horizontal left translation that $K_{\Sigma}=F_3$. This
of course may produce a change of base point for $f$, of the form $f(p):=Q\in \R^2\rtimes_A \{0\}$.
Recall from the beginning of Section~\ref{sec:semidirect} that
$f\colon (\Sigma,p)\la (X,e)$ is obtained as a limit of pointed limit immersions
$f_n\colon (S_n, p_n)\la (X,e)$ of spheres $S_n\in \mathcal{C}$.
Applying to the $f_n$ the same horizontal left translation that we applied
to $f$ in order to have $K_{\Sigma }=F_3$, we have
$f_n(p_n)=Q$ for all $n$.

By Proposition~\ref{propsu2}, each sphere $S_n\in \mathcal{C}$ is Alexandrov embedded.
In this way, for each $n\in \N$ there exists a three-dimensional Riemannian manifold
$(B_n,g_n)$ which is topologically a closed ball and a Riemannian submersion
$\wt{f}_n\colon (B_n,g_n)\flecha X$ such that $\parc B_n=S_n$, $\wt{f}_n |_{S_n}=f_n$ and
$\parc B_n$ is mean convex.

\begin{lemma}\label{c2no}
In the conditions above, there is some $R^*>0$ and points $q_n\in S_n$ such that the
volumes of the Riemannian metric balls $B_{g_n} (q_n,R^*)$ in $(B_n,g_n)$ of radius
$R^*$ centered at $q_n$ tend to $\8$ as $n \to \8$.
\end{lemma}
\begin{proof}
As explained in Remark~\ref{rem7.1}, the surface
$f(\Sigma )$ can be obtained by pulling back via
$\Pi _{K_{\Sigma }}$ a closed immersed curve $\a $ contained in $X/K_{\Sigma }$.
Since in our case $K=K_{\Sigma }=F_3$ and every integral curve of $F_3$ intersects any
horizontal plane $\R^2\rtimes_A \{z_0\}$ transversely at a single point,
we can identify $X/K_{\Sigma}$ with $\R^2\rtimes_A \{0\}$.

Consider the vertical geodesic of $X$
\begin{equation}
\label{eq:z-axis}
\Gamma=\{ \G (t)=(0,0,t)\mid t\in \R\}.
\end{equation}
Given $r>0$,
let $\mathcal{W}(\Gamma,r)$ denote the solid metric cylinder of radius $r$ around $\G $,
i.e., the set of points of $X$ whose distance to $\Gamma$ is at most $r$.
Since $f(\Sigma)$ is an annulus invariant under the flow of $F_3$, and
since all cylinders $\mathcal{W}(\Gamma,r)$ are also invariant under this flow, it
is clear that $f(\Sigma)$ is contained in a solid metric cylinder
$\mathcal{W}_0 := \mathcal{W}(\Gamma,r_0)$ for some $r_0>0$,
such that $\parc \mathcal{W}_0$ is at a positive distance from $f(\Sigma)$.
By~\cite[Theorem~1.1]{mmp1}, there exists some $R=R(r_0)>0$ associated to
$\mathcal{W}_0$ with the following property:
\begin{enumerate}[$(\star )$]
\item If $M$ is any compact, stable minimal surface in $X$ whose boundary
is contained in $\mathcal{W}_0$, then the
intrinsic radius\footnote{The radius of a
compact Riemannian manifold $M$ with boundary is the maximum distance of points in $M$
to its boundary $\partial M$.} of $M$
is less than $R$.
\end{enumerate}																																	
In particular, since the mean curvature of the complete
stable surface $f(\Sigma)$ is $h_0(X)$ and $f(\Sigma )$ is contained in $\mathcal{W}_0$,
it follows that $f(\Sigma)$ cannot be minimal, thus, $h_0(X)>0$.

Let us also consider the positive number
$D_0=D_0(r_0)$ defined in the following way.
Observe that $\mathcal{W}_0$ is foliated by integral curves of $F_3$. We define
$D_0$ as the maximum length
of all the arcs of integral curves of $F_3$ that are contained in the compact piece of
$\mathcal{W}_0$ that lies in the slab determined by the planes $\R^2\rtimes_A \{R+2\}$
and $\R^2\rtimes_A \{-R-2\}$; note that $D_0$ exists since $\mathcal{W}_0\cap
[\R^2\rtimes _A\{ 0\} ]$ is compact.

As $\Sigma $ is simply connected and
$f\colon \Sigma \la X$ is invariant under the flow of $F_3$, then by the discussion in Remark~\ref{rem7.1}
we may parameterize $f$ as
\begin{equation}
\label{eq:fst}
f(s,t)=l_{\Gamma(t)} (\alfa(s)),\quad (s,t)\in \R^2,
\end{equation}
where $\alfa=\a (s)\colon\R\flecha \R^2\rtimes _A\{ 0\}$
is a curve parameterized by arc length in $\R^2\rtimes _A\{ 0\}$, and so that under the natural identification
$\Sigma\equiv \{(s,t)\in \R^2\}$, the base point $p\in \Sigma$ for $f$ corresponds to $(0,0)$. In particular,
$f(0,0)=\alfa(0)=Q$. Note that since we are in the conditions of case (S2.2), the image $\alfa(\R)$
is an immersed closed curve and $\alfa(s)$ is $L$-periodic, where $L$ is the length of $\alfa(\R)$.

For each $s_0\in \R$ we may define numbers $c^-(s_0)<0$, $c^+(s_0)>0$ such that
\begin{equation}
\label{arki}
A(s_0):=\{l_{\Gamma(t)} (\alfa(s_0)) \ | \  c^-(s_0)\leq t\leq c^+(s_0)\}
\end{equation}
is the arc of the integral curve of $F_3$ that passes through $\alfa(s_0)$ whose
 endpoints lie in $\R^2\rtimes_A \{R+2\}$ and $\R^2\rtimes_A \{-R-2\}$. Note that $A(s_0)\subset f(\Sigma)$,
and that the length of $A(s_0)$ is at most $D_0$.

Take some $a\in (0,L)$ and consider for each $j\in \N$ the compact region
 of $\Sigma $ (viewed as $\R^2$ with coordinates $(s,t)$)
given by
\begin{equation}
\label{Kj}
K^j =\{ (s,t) \ | \ 0\leq s \leq jL +a, \ c^-(s)\leq t\leq c^+(s)\} ,
\end{equation}
see Figure~\ref{fig16}.
\begin{figure}
\begin{center}
\includegraphics[height=6cm]{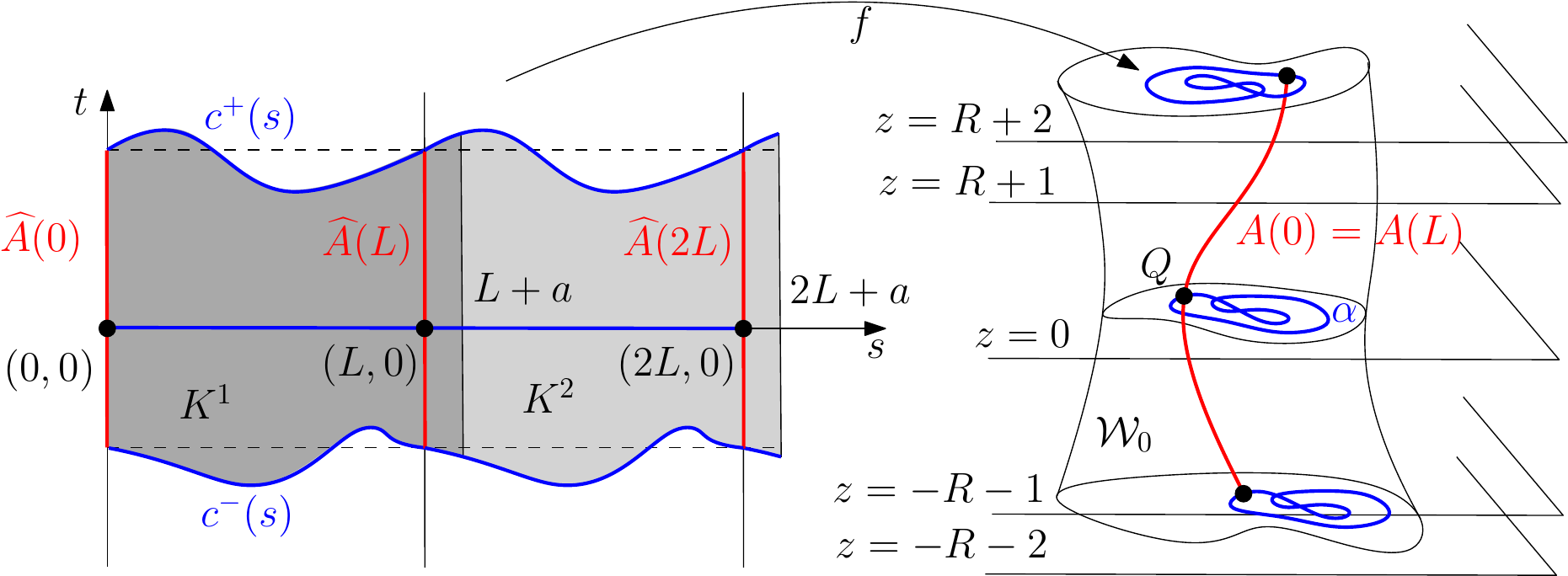}
\caption{Left: the compact regions $K^1\subset K^2$ in the $(s,t)$-plane.
The vertical segments $\widehat{A}(0)$,
$\widehat{A}(L),\widehat{A}(2L)$ apply through $f$ into the compact arc $A(0)$,
which is part of the integral curve of $F_3$ passing through $Q$.
Each $K^j$ wraps $j$ times through $f$ around its image, due to the $L$-periodicity of $f$ in the $s$-variable.}
\label{fig16}
\end{center}
\end{figure}

Since $f$ is a limit surface of the $f_n$, it follows that for $n$ large enough, there exist
compact simply connected domains $K_n^j\subset S_n$ such that
 $f_n(K_n^j)\subset \mathcal{W}_0$ and the sequence $\{ f_n |_{K_n^j}\} _n$ converges to
 $f|_{K^j}$ uniformly as $n\to \8$. For $\ve >0$ sufficiently small and
 less than the injectivity radius of $\Sigma $, denote by $\cU^{k}$ the closed geodesic disk in $\Sigma $
 of radius $\ve$ centered  at $(kL,0)$ with $k\in \{1,\dots, j\}$. After choosing $\ve>0$ sufficiently small,
 we may assume that $\{ \cU^1,\ldots ,\cU^j\} $
 forms a pairwise disjoint collection, and the intrinsic distance
 between $\cU^{i}$, $\cU^k$  for  $i,k\in \{1,\dots, j\}$, $i\neq k$, is greater than $2\de $ for some $\de >0$
 independent of $j$.
Also, there exist closed geodesic disks of radius $\ve $, $\cU_n^{k} \subset K_n^j$,
such that $\{ f_n |_{\cU_n^k} \} _n$ converges to
$f|_{\cU^{k}}$ uniformly as $n\to \8$.
Observe that for each $j$, there exists $n_0(j)\in \N$ such that if $n\geq n_0(j)$,
 the related compact disks
 $\cU_n^k\subset K_n^j$ satisfy that the intrinsic
 distance between $\cU_n^i$ and $\cU_n^k$ inside $K_n^j$ is greater than $\delta$
 for every $i,k\in \{1,\dots, j\}$, $i\neq k$.
It is important to notice that $\delta$ does not depend on $j$, since $f(s,t)$ is $L$-periodic with
respect to $s$.

Let $I$ denote a small compact segment in $\R^2\rtimes _A\{ 0\}$ centered at $Q$
and transversal to $f(\Sigma )$ at~$Q$.
Let $\alfa_n \subset S_n$ denote the set $f_n^{-1} (\R^2\rtimes _A\{ 0\})$.
Since $f$ is a limit of the $f_n$ and $f(\Sigma)$ intersects $\R^2\rtimes_A \{0\}$ transversely, then
by the Transversality Lemma (Lemma~3.1 in~\cite{mmp2}),
$\alfa_n$ is a closed simple curve in $S_n$.
As $f(s,t)$ is $L$-periodic in the variable $s$, the previous convergence properties imply that for $n=n(j)$
large enough there exist $j$ points $q_n^1,\dots, q_n^j \in \alpha_n\cap K_n^j$ such that, for every $k=1,\dots, j$, we have
$q_n^k\in \cU_n^k$ and $f_n(q_n^k)\in I$. In particular, we note for later use the following property:
\begin{enumerate}[$(\star \star)$]
\item For $n$ large enough, the intrinsic compact metric disks $D_{S_n} (q_n^k,\delta/2)$
in $S_n$ of radius $\de /2$ and center $q_n^k$,
$k=1,\dots, j$, are pairwise disjoint.
\end{enumerate}		

Let $J_n^k$ be the compact connected arc of $\alfa_n$ contained
in $K_n^j$ whose endpoints are $q_n^1$ and $q_n^k$, for each $k\in \{1,\dots, j\}$.
Let $W_n$ be the vector field on $S_n$ obtained by pulling back via $f_n$ the
tangent part to $f_n(S_n)$ of the right invariant vector field $K_{\Sigma}=F_3$; clearly,
$W_n$ has no zeros on $K_n^j\subset S_n$ for $n$ large enough, since $K_{\Sigma}$ has no zeros, is
 everywhere tangent to $f(\Sigma)$ and $\{ f_n |_{K_n^j}\} _n$ converges to $f|_{K^j}$
 uniformly as $n\to \8$. Moreover, again for $n$ large enough,
the integral curves of $W_n$ that start at any point of  any of the arcs $J_n^k$ satisfy that their images
by $f_n$ intersect both planes $\R^2\rtimes_A \{-R-1\}$ and $\R^2\rtimes_A \{R+1\}$.
In particular, we can define the compact disks $D_n^k\subset S_n$
obtained by letting the arc $J_n^k$ flow under $W_n$ between the
sets $f_n^{-1}(\R^2\rtimes_A \{-R-1\})$ and $f_n^{-1}(\R^2\rtimes_A \{R+1\})$ of $S_n$.
In fact, $D_n^k\subset K_n^j$ for all $k\in \{ 1,\ldots ,j\} $ and for $n$ large.

Let $\gamma_n^k :=\parc D_n^k$, which is a Jordan curve in $S_n$. Observe that
$\gamma_n^k$ can be written as a union $A_1^n\cup A_2^n \cup B_1^n \cup B_2^n$, where:
\begin{enumerate}[(T1)]
\item
$A_1^n,A_2^n$ are compact arcs of integral curves of $W_n$ that pass through $q_n^1$ and
$q_n^k$, respectively. In particular, the endpoints of both $f_n(A_1^n)$ and $f_n(A_2^n)$ lie
on the planes $\R^2\rtimes _A\{ -R-1\} $, $\R^2\rtimes _A\{ R+1\} $.
Moreover, since $f_n(A_1^n)$ and $f_n(A_2^n)$ converge uniformly as $n\to \8$
to proper subarcs of the compact arc $A(0)$  defined by equation \eqref{arki},
then the lengths of $A_1^n$ and $A_2^n$ are smaller than the constant $D_0$ for $n$ large enough.
\item
The arc $f_n(B_1^n)$ (resp. $f_n(B_2^n)$) lies in the plane $\R^2\rtimes_A \{ -R-1\}$ (resp. $\R^2\rtimes_A \{R+1\}$).
\end{enumerate}

Recall that $(B_n,g_n)$ is the abstract Riemannian three-ball that $S_n$ bounds, and
$\wt{f}_n\colon (B_n,g_n)\flecha X$ is a Riemannian submersion with mean convex boundary $\wt{f}_n |_{\parc B_n}=f_n$.
In particular, $S_n$ is a good barrier to solve Plateau problems in $B_n$ in the sense of
Meeks-Yau~\cite{my2}. More precisely,
for each $k\in \{ 1,\ldots, j\} $ (and $n$ large enough),
there exists a compact stable embedded minimal disk $M_n^k\subset B_n$ of least
area, with $\parc M_n^k =\gamma_n^k$. Furthermore, also by~\cite{my2}, $M_n^k$ is an immersion
transverse to $\parc B_n$ at the regular points of its boundary $\gamma_n^k$. In particular, $M_n^k$ is an
immersion transverse to $\parc B_n$ at the interior points of the arcs $A_1^n$ and $A_2^n$.
Therefore, if we denote by $\Lambda_n\subset B_n$ the embedded surface $\{q\in B_n \mid \wt{f}_n(q)\in \R^2\rtimes _A\{ 0\}\}$
and $\beta_n^k:=M_n^k\cap \Lambda_n$, then $\beta_n^k\cap \parc M_n^k=\{q_n^1,q_n^k\}$, and the intersection
of $M_n^k$ with $\Lambda_n$ is clearly transverse at $q_n^1,q_n^k$ since the arcs $A_1^n,A_2^n$ are transverse to~$\Lambda_n$.
Moreover, up to a small perturbation of $\Lambda_n$ in $B_n$ that fixes its boundary, we may assume that
the intersection of $M_n^k$ with $\Lambda_n$ is transverse.
Therefore, $\beta_n^k$ is a (possibly disconnected) compact $1$-dimensional manifold with boundary, and a
connected component of $\beta_n^k$ is a compact arc in $B_n$ joining $q_n^1$ with $q_n^k$.
We will keep denoting this compact arc by $\beta_n^k$, see Figure~\ref{fig17}.
\begin{figure}
\begin{center}
\includegraphics[height=10cm]{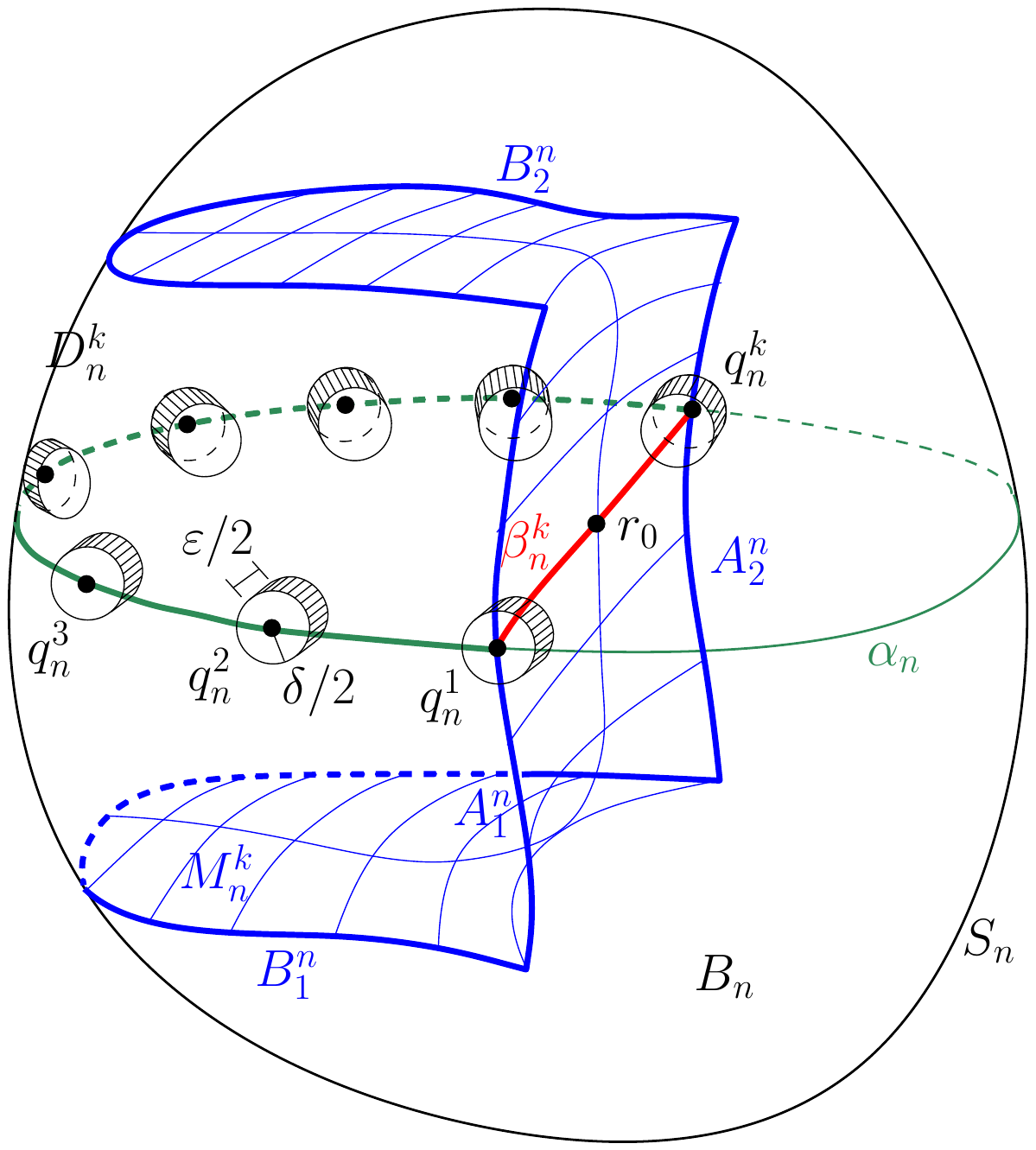}
\caption{The least area disk $M_n^k$ inside the abstract Riemannian ball $B_n$
has the same boundary $\g _n^k$ as the disk $D_n^k\subset S_n=\partial B_n$.
$\g _n^k$ consists of four consecutive arcs $A_1^n,B_1^n,A_2^n,B_2^n$ satisfying properties (T1), (T2).}
\label{fig17}
\end{center}
\end{figure}

Take $r\in \be _n^k$. We next give a lower bound of
$d_X(\wt{f}_n(r) ,f_n(B_1^n\cup B_2^n))$,
the distance in $X$ from $\wt{f}_n(r)$ to $f_n(B_1^n\cup
B_2^n)$. Since $f_n(B_1^n\cup B_2^n)$ is contained in the union of
planes $\R^2\rtimes _A\{ \pm (R+1)\} $, then
$d_X(\wt{f}_n(r) ,f_n(B_1^n\cup B_2^n))\geq d_X(\wt{f}_n(r), \R^2\rtimes _A\{ \pm R_1\})$.
As $\wt{f}_n(r)$ lies at an arbitrarily small distance of $\R^2\rtimes _A\{ 0\}$ and
$d_X(\R^2\rtimes _A\{ 0\} , \R^2\rtimes _A\{ \pm (R+1)\})=R+1$
 (this follows, for instance, from~\cite[Lemma~3.9]{mpe11}, since
$\R^2\rtimes _A\{ \pm (R+1)\}$ are right cosets of
$\R^2\rtimes _A\{ 0\}$), then $d_X(\wt{f}_n(r) ,\R^2\rtimes _A\{ \pm (R+1)\} )$
 can be taken arbitrarily close to $R+1$. In particular,
\begin{equation}
\label{eq:d>R}
d_X(\wt{f}_n(r),f_n(B_1^n\cup B_2^n))>R,\quad \mbox{for all $r\in \beta_n^k$.}
\end{equation}

As $\wt{f}_n|_{M_n^k}\colon M_n^k \la X$ is
by construction a compact immersed stable minimal disk in $X$
whose boundary is contained in the solid metric cylinder $\mathcal{W}_0$,
we deduce from the radius estimate $(\star )$ above
that for any $r\in \beta_n^k$
the distance $d_{M_n^k}(r,\g _n^k)$ in $M_n^k$ from
$r$ to $\gamma_n^k=\parc M_n^k$ is not greater than $R$.
Note that $d_{M_n^k}(r,B_1^n\cup B_2^n)\geq d_X(\wt{f}_n(r),f_n(B_1^n\cup B_2^n))$,
which by (\ref{eq:d>R}) is greater than $R$. Therefore,
\begin{equation}
\label{eq:7.4}
d_{M_n^k} (r,\gamma_n^k) = d_{M_n^k} (r, A_1^n\cup A_2^n),
\quad \mbox{for all $r\in \beta_n^k$.}
\end{equation}
Moreover, since $q_n^1\in A_1^n$ and $q_n^k\in A_2^n$, we deduce by the continuity
and connectedness of the arc $\beta_n^k$ that there exists some \emph{midpoint}
$r_0\in \beta_n^k$ such that $d_{M_n^k} (r_0,A_1^n) = d_{M_n^k} (r_0, A_2^n)$.
Thus, we can estimate the distance in $B_n$ from $r_0$ to $A_i^n$, $i=1,2$, by
 \begin{equation}
 \label{eq:7.5}
 \def\arraystretch{1.3}
 \begin{array}{llll}
d_{B_n} (r_0,A_i^n) & \leq & d_{M_n^k} (r_0,A_i^n)& \\
  & = & d_{M_n^k} (r_0,A_1^n\cup A_2^n) & \mbox{($r_0$ is a midpoint)}
  \\ & = &d_{M_n^k} (r_0,\gamma_n^k) & \mbox{(by equation (\ref{eq:7.4}))}
  \\ & \leq  &R. &
\end{array}
\end{equation}
 Therefore,
 \begin{equation}
 \label{tridi}
\def\arraystretch{1.3}
\begin{array}{llll}
d_{B_n} (q_n^1,q_n^k) & \leq & d_{B_n} (q_n^1,r_0)+d_{B_n} (r_0,q_n^k)&
 \mbox{(triangle inequality)}\\
& \leq & D_0 + d_{B_n} (r_0,A_1^n) + D_0 + d_{B_n} (r_0,A_2^n) & \mbox{(triangle inequality and (T1))}\\
& \leq & 2D_0 + 2R.& \mbox{(inequality (\ref{eq:7.5}))}
\end{array}\end{equation}

Recall that, by item 5 of Proposition~\ref{propsu2}, the norms of the second fundamental
forms of the spheres $S_n$ are uniformly bounded by some constant $C>0$. Also, note that
since we have already proved that $h_0(X)>0$ in this case, the values $H_n$ of the
mean curvatures of $S_n$ are uniformly bounded away from zero (since $H_n>h_0(X)$ for
every $n$). Hence, by~\cite{mt3}, there exists some $\ep_1>0$ smaller than the injectivity
radius of $B_n$ such that $S_n=\parc B_n$ has a regular neighborhood
$\mathcal{V}_n^{\ep_1}$ in $B_n$ of fixed size $\ep_1$ (independent of $n$).
In other words, we have a diffeomorphism
\[
\Phi_n \colon S_n \times [0,\ve _1)\to \mathcal{V}_n^{\ep_1},
\quad\Phi (x,t)=\exp ^n_{x} (t N_n(x))
\]
where $N_n$ is the inward pointing unit normal in $(B_n,g_n)$ of the mean convex
sphere $S_n=\parc B_n$ and $\exp^n$ denotes the exponential map in $(B_n,g_n)$.

By Property $(\star\star )$ above, for $n$ large enough,
the solid cylinders $C_n^1,\dots, C_n^j$ in $\mathcal{V}_n^{\ep_1}$ given by the image under
$\Phi _n$ of the compact cylinders $D_{S_n} (q_n^k, \delta/2)\times [0,\ep_1/2]$
are mutually disjoint, see Figure~\ref{fig17}.
Also, note that the volume of each $C_n^k$ in $(B_n,g_n)$
is greater than some $V>0$,
independent of $n,j$ and $k$ (again this follows from the uniform bound of the second fundamental
forms of the $S_n$, as well as from the fact $\ve_1$ does not depend on $n$).
In particular, for $n$ large enough, the total volume in $(B_n,g_n)$
of the (disjoint) union $C_n^1\cup \cdots \cup C_n^j$ is at least $jV$. Note that this volume
can be made arbitrarily large, as $j\in \N$ was chosen fixed but arbitrary.

Also, by the triangle inequality we have for any $x\in C_n^k$, $k\in \{ 1,\ldots ,j\} $, that
$d_{B_n}(q_n^k,x)\leq \frac{\ep_1+\delta}{2}$. This inequality and (\ref{tridi}) imply that
\[
d_{B_n}(q_n^1,x) \leq d_{B_n} (q_n^1,q_n^k) + d_{B_n}( q_n^k,x)
\leq 2(D_0+R) +\frac{\ep_1 +\delta}{2},
\]
from where we deduce that $C_n^1\cup \cdots \cup C_n^j$ is contained in the Riemannian ball
of $(B_n,g_n)$ centered at $q_n^1$ and of radius $R^*:= 2(D_0+R) +\frac{\ep_1+\delta}{2}$.

In particular, we have proved the following: given a fixed but arbitrary $j\in \N$,
then for all $n\in \N$ large enough, the metric ball $B_{g_n}(q_n^1,R^*)$ in the Riemannian
manifold $(B_n,g_n)$ centered at $q_n^1\in \parc B_n$ of radius $R^*$ has volume at least $jV$,
where $V,R^*>0$ do not depend on $j,n$.
This clearly proves Lemma~\ref{c2no}.
\end{proof}

We will next obtain a contradiction with Lemma~\ref{c2no}, that will prove that case (S2.2) is
impossible when $K_{\Sigma }\notin {\rm Span}\{ F_1,F_2\} $, which was the objective of the present
Section~\ref{sec:nos}. First, observe that as the second fundamental
forms of the $f_n$ are uniformly bounded, there exists $\de >0$ such that, for all $n$, the following properties hold:
\begin{enumerate}[(U1)]
\item $(B_n,g_n)$ can be extended to a compact Riemannian three-manifold with
boundary $(\wt{B}_n,\wt{g}_n)$, in the sense that $B_n$ is a compact subset of $\wt{B}_n$
and $\wt{g}_n|_{B_n}=g_n$.
\item The map $(x,t)\in S_n\times [-\de ,0]\to  \exp ^n_{x}(tN_n(x))\in \wt{B}_n-\mbox{Int}(B_n)$
is a diffeomorphism, where $\exp ^n$ denotes the exponential map of
$(\wt{B}_n,\wt{g}_n)$ and $N_n$ is the unit normal vector to $S_n$ at $x$ with respect to
$\wt{g}_n$ such that $H_nN_n$ is the mean curvature vector of $S_n$.
\item The metrics $\wt{g}_n$ are uniformly bounded in the $C^2$-topology.
\end{enumerate}

 Let  $\varphi \colon (-\de ,0]\to (0,\infty )$ be a smooth positive function such that
$\varphi =1$ in $[-\de /3,0]$ and $\varphi (t)=\frac{1}{t+\de }$  in $(-\de ,-2\de /3]$.
We define the complete Riemannian metric
$\wh{g}_n$ on $\wt{B}_n-\partial \wt{B}_n$ by
\[
\wh{g}_n=\left\{ \begin{array}{ll}
g_n & \mbox{in $B_n$,}
\\
\varphi ^2\wt{g}_n & \mbox{in $S_n\times (-\de ,0]$}
\end{array}\right.
\]
(we are identifying $S_n\times [-\de ,0]$ with $\wt{B}_n-\mbox{Int}(B_n)$ through
the diffeomorphism that appears in property (U2)).
Observe that
\begin{equation}
\label{Bishop1}
\wh{g}_n=\frac{1}{(t+\de )^2}\wt{g}_n \hspace{0.2cm} \mbox{in } \hspace{0.2cm} S_n\times (-\de ,-2\de/3].
\end{equation}

As $\wt{g}_n|_{\partial \wt{B}_n}$ is a Riemannian metric on the closed surface
$\partial \wt{B}_n$, then (\ref{Bishop1}) implies that given $\ve >0$,
there exists $\de _1=\de _1(\ve )\in (2\de /3,\de )$ such that the restriction of $\wh{g}_n$
to $S_n\times (-\de ,-\de _1]$ is $\ve $-close in the $C^2$-topology to the hyperbolic
metric of constant sectional curvature $-1$. In particular, the Ricci curvature of $\wh{g}_n$ satisfies
\begin{equation}
\label{Bishop2}
\left| \mbox{Ric}_{\wh{g}_n}+2\right| _{L^{\infty }(S_n\times (-\de ,-\de _1])}<b(\ve ),
\end{equation}
where $b(\ve )$ can be made arbitrarily small for $\ve $ small.

As the coefficients of the Ricci tensor of $\varphi ^2\wt{g}_n$ are smooth expressions of
the coefficients of $\wt{g}_n$, of those of the Ricci tensor of $\wt{g}_n$ and of
$\varphi,\varphi',\varphi''$, then property (U3) and the smoothness of $\varphi $ imply
that  the Ricci curvature of $\varphi ^2\wt{g}_n$ satisfies for all $n$:
 \begin{equation}
\label{Bishop3}
\mbox{Ric}_{\varphi ^2\wt{g}_n} \mbox{ is uniformly bounded in } S_n\times [-\de _1,0].
\end{equation}

As $g_n$ is locally homogeneous, (\ref{Bishop2}) and (\ref{Bishop3}) imply that
there exist $c<0$ independent of~$n$
such that Ric$_{\wh{g}_n}\geq 2c$ in $\wt{B}_n-\partial \wt{B}_n$ for all $n$.
Since $(\wt{B}_n-\partial \wt{B}_n,\wh{g}_n)$ is a complete Riemannian manifold, then
Bishop's comparison theorem gives the following upper bound for the
volume of the metric ball $B_{\wh{g}_n}(q_n,R^*)$ in
$(\wt{B}_n-\partial \wt{B}_n,\wh{g}_n)$ of radius $R^*>0$ centered at $q_n$:
\[
\mbox{Vol}\left( B_{\wh{g}_n}(p_n,R^*) \right) \leq V(c,R^*),
\]
where $V(c,R^*)$ denotes the volume of any metric ball of radius $R^*$ in the three-dimensional
space form of constant sectional curvature $c$. This is a contradiction with Lemma~\ref{c2no}, as desired.

\subsection{Case (S2.1) is impossible when $K_{\Sigma}$ is not horizontal}
\label{sec:noa}

Assume that case (S2.1) holds and $K_{\Sigma }\notin \mbox{\rm Span}\{ F_1,F_2\} $.
We start with the same normalizations as in the previous Section~\ref{sec:nos}. So, $K_{\Sigma}=F_3$,
and $f\colon(\Sigma,p)\la (X,Q)$ (here $Q\in \R^2\rtimes _A\{ 0\} $)
is obtained as a limit of pointed immersions $f_n\colon(S_n, p_n)\la (X,Q)$
of spheres $S_n\in \mathcal{C}$. 
Also, we will consider the compact Riemannian balls $(B_n,g_n)$ with $\parc B_n=S_n$, and
Riemannian submersions $\wt{f}_n\colon (B_n,g_n)\flecha X$ such that $\parc B_n$ is mean convex
and $\wt{f}_n |_{\parc B_n}=f_n$. Let
$\a_n:=f_n^{-1}(\R^2\rtimes _A\{ 0\})\subset S_n$, which is a simple closed curve in $S_n$ by the
Transversality Lemma (Lemma~3.1 in~\cite{mmp2}).
Also, let $K_{\Sigma}^n$ be the Killing field on $B_n$ that is induced by $K_{\Sigma}$
via the pullback by the submersion $\wt{f}_n$.

Recall that in this case (S2.1), $\Sigma $ is an annulus.
Observe that $f(\Sigma )$ is the immersed annulus obtained by letting the immersed
closed curve $f(\Sigma )\cap (\R^2\rtimes _A\{ 0\} )$ flow under $K_{\Sigma}$. 
In particular, $\a _{\infty }=f^{-1}(\R^2\rtimes _A\{ 0\} )$ is a  simple closed curve in $\Sigma $, and
the sequence $\{ f_n|_{\a _n}\} _n$ converges uniformly to the immersion
$f|_{\a _{\infty }}$.

By Theorems~1 and 2 in~\cite{my2}, $\a_n$ is the boundary of a smooth, embedded
least-area disk $M(n)$ in $B_n$ and this disk is transverse to $S_n=\partial  B_n$ along
its boundary $\a_n$. Note that the intrinsic radii of the disks $M(n)$ are bounded by
a universal constant $R>0$ by Theorem 1.1 in~\cite{mmp1},
since their compact image boundary curves $f_n|_{\partial M(n)}=f_n|_{\a_n}$ converge
uniformly as $n\to \infty $ to $f|_{\alpha _{\infty }}$, and $f(\alpha _{\infty })$
is contained in a solid metric cylinder around $\Gamma=\{ (0,0,t)\ | \ t\in \R \} $ (so the same holds
for  $f_n(\partial M(n))$ for all $n\in \N$, for a slightly larger solid metric cylinder around $\Gamma$).
It follows that a subsequence of the
immersed stable minimal disks $\wt{f}_n(M(n))$, which lie locally in the mean convex side of
$f_n(S_n)$ near $f_n(\a_n)$, converge to an immersed stable minimal disk $M(\infty)$ in $X$
with boundary $f(\a _{\infty })$.

We claim that $M(n)$ lies in the piecewise smooth,
closed upper half-ball $B_n^+=\wt{f}_n^{-1}(\R^2\rtimes_A [0,\8))$ determined
by the disk $\wt{f}_n^{-1}(\R^2\rtimes _A\{ 0\})$. To see this, first recall that
for every $z_0\in \R$, the halfspace of the form $\R^2\rtimes _A[z_0,\infty )$ is mean convex in $X$,
and the constant mean curvature of its boundary is equal to $H(X)$. 
Since $ f_n(\a_n)\subset \R^2\rtimes _A\{ 0\}$ is the boundary of the immersed minimal disk
$f_n(M(n))$, then the maximum principle (or the mean curvature comparison principle)
applied to $M(n)$ and to the smooth surfaces $\wt{f}_n^{-1}(\R^2\rtimes_A \{z\})$, $z\in \R$, implies
that $M(n)\subset B_n^+$ as desired.

We next claim that  there exists an $\eta\in(0,\pi/4)$ such that the
angle that $M(n)$ makes with $S_n$ along $\a _n$ is greater than $\eta$ for all $n\in \N$ large enough.
Assume that the claim does not hold; in this case, 
$M(\infty )$ is tangent to $f(\Sigma )$ at some boundary point $Q'\in \partial M(\infty )$.
As $M(\infty )$ lies locally in the mean convex side of $f(\Sigma )$ along $f(\a _{\infty })$,
then the Hopf boundary maximum principle implies that $f(\Sigma )$ coincides with $M(\infty )$ nearby $Q'$.
This is a contradiction since we know that $h_0(X)>0$ in this case, as it was explained just after
property $(\star )$ in the previous section.
Thus, the claim at the beginning of this paragraph holds.

Let $N_{M(n)}$ denote the unit normal vector field to
$M(n)$ in $(B_n,g_n)$. We claim that for $n$ sufficiently large,
$g_n(N_{M(n)},K_{\Sigma}^n)$ has no zeros along $\a_n$. Arguing by
contradiction, suppose that after passing to a subsequence,
 $K_{\Sigma}^n$ is tangent to $M(n)$ at some point $q_n\in \a _n$ for all $n$.
 Then $K_{\Sigma}^n(q_n)$ and the unit tangent vector $w_n$ to $\a _n$ at $q_n$ generate the
 tangent space $T_{q_n}M(n)$. By the last paragraph, the angle between $T_{q_n}M(n)$
 and $T_{q_n}S_n$ is greater than $\eta $. Since $K_{\Sigma}$ 
is everywhere tangent to
$f(\Sigma)$, then the angle that $K_{\Sigma}^n(q_n)$ makes with $T_{q_n}S_n$ tends
to zero as $n\to \infty $. Since $K_{\Sigma}^n(q_n)$ is tangent to $M(n)$ at $q_n$,
then the angle that $K_{\Sigma}^n(q_n)$ makes with $w_n$
becomes arbitrarily small if $n$ is large enough
(note that $w_n\in T_{q_n} S_n\cap T_{q_n} M(n)$). Applying the differential of $\wt{f}_n$ at $q_n$,
we deduce that the angle between $(K_{\Sigma})_{f_n(q_n)}=(F_3)_{f_n(q_n)}$
and $(df_n)_{q_n}(w_n)$ tends to zero as
$n\to \infty $. Since $(df_n)_{q_n}(w_n)$ is a nonzero horizontal vector, this only can occur
provided that $(F_3)_{f_n(q_n)}$ becomes horizontal as $n\to \8$, or equivalently, the ratio between
the component of $(F_3)_{f_n(q_n)}$ in the direction of
$\partial _z$ and
its horizontal component tends to zero. By equation (\ref{eq:6}), this only can occur
if $f_n(q_n)$ diverges in $\R^2\rtimes _A\{ 0\}$. But this contradicts the fact that
$f(\a _{\infty })$ lies at a finite distance from the origin.
Hence the claim follows.

Therefore, after replacing by a subsequence and choosing $N_{M(n)}$ so that the
Jacobi function $J_n=g_n(N_{M(n)},K_{\Sigma}^n)$ is  positive along $\a_n$, we
deduce from the stability of the minimal disk $M(n)$ that $J_n$ is positive on $M(n)$.

Consider the Jacobi function $J_\infty=\langle N_{M(\infty )},K_{\Sigma}\rangle $
defined on the stable minimal disk $M(\infty)$. Observe that $J_{\infty }$ does not vanish
at any point of $\partial M(\infty )=f(\a _{\infty})$ (otherwise $M(\infty )$ would become tangent
to $f(\Sigma )$ at some boundary point of $M(\infty )$, which we have seen that contradicts
the Hopf boundary maximum principle). Once we know that $J_{\infty }$ does not vanish
at any point of $\partial M(\infty )$, then $J_{\infty }|_{\partial M(\infty )}$ is positive (because
$J_n|_{\a_n}$ is positive for $n$ large), and thus, $J_{\infty }>0$ on $M(\infty )$ because
$M(\infty )$ is stable.

Consider the mapping $T_n\colon M(n)\times \R \to X$ given by
\begin{equation}
\label{eq:Tn}
T_n(q,t)=l_{\G(t)}(\wt{f}_n(q)),
\end{equation}
where $\G(t)$ is given by (\ref{eq:z-axis}). As $J_n$ is positive  on $M(n)$,  $T_n$
is a proper submersion of $M(n)\times \R $ into $X$ (note that the projection $(x,t)\mapsto t$
is proper). 
Our next goal is to take limits of these immersions $T_n$.
To do this, we use a three-point condition. Let $\overline{\D }$ be the closed unit disk in $\R^2$ and parameterize each
$M(n)$ by a conformal diffeomorphism $\phi_n\colon \ov{\D} \to M(n)$ in such a way
that the points $(1,0), (0,1), (-1,0)$ of $\partial \D$ have as images by $f_n\circ \phi _n$
three ordered points in $f_n(\a _n)$ which converge to three ordered points in the limit curve
$f(\a_\infty)=\partial M(\infty )$.
In this case, after passing to a subsequence, the conformal immersions
$\wt{f}_n\circ \phi_n \colon \ov{\D} \la X$ converge smoothly to a limit conformal immersion
$\wt{f}_\infty\colon \ov{\D}\la X$ that parameterizes $M(\infty )$; for this standard type of limit result,
see~\cite{mor2} and~\cite{my1} for details.

Note that as $n\to \infty $, the proper immersions {$T_n\circ(\phi _n\times 1_{\R }) \colon
\ov{\D}\times \R \to X$} converge to the submersion
$T_{{\infty }}\colon \ov{\D}\times \R \to X$ given by
\begin{equation}
\label{eq:Tinfty}
T_{\infty }(x,t)=l_{\G(t)}(\wt{f}_{\infty }(x)),
\end{equation}
which again is a proper submersion since $(x,t)\in \ov{\D}\times \R\to \R \mapsto
t\in \R $ is proper (because $J_{\infty }$ is positive on $M(\infty )$).
Consider the codimension-one foliation of $\ov{\D}\times \R$ given by
the family of compact disks
\begin{equation}
\label{fol}
\cF =\{ D_t=T_\infty^{-1}(\R^2\rtimes_A \{ t\} ) \ | \ t\in \R\} .
\end{equation}
Let $K_1$
be the nowhere zero, smooth vector field on $\ov{\D}\times \R$ obtained
after pulling back $K_{\Sigma }$ through $T_{\infty }$. Clearly,  each of the
integral curves of $K_1$ in $\ov{\D}\times \R $ 
intersects the disk $D_0$ in a single point.
Thus, we can view $\ov{\D}\times \R$ in a natural manner as being the
topological product $W:=D_0 \times \R$, and we
endow this three-manifold with boundary with the pulled-back metric by $T_{\infty }$. We
will find the desired contradiction as an application of Theorem~\ref{flux} below, as we explain next.

Observe that the boundary $\partial W=(\partial D_0)\times \R$ of $W$ satisfies
\begin{equation}
\label{eq:bdry}
T_{\infty}(\partial W)=\{ l_{\G (t)}(T_{\infty }(x))\ | \ x\in \partial D_0, \ t\in \R\}
=\{ l_{\G (t)}(f(x))\ | \ x\in \a _{\infty }, \ t\in \R\} =f(\Sigma ).
\end{equation}
Recall that $f(\Sigma )$ has constant mean curvature $h_0(X)>0$ in $X$ with respect
to the unit normal field obtained as the limit of the inward pointing unit normal fields
to the spheres $S_n$; see the paragraph after property $(\star)$ in Section~\ref{sec:nos}. Also note that the inward
pointing unit normal field on $\partial W$  has positive inner product with the inner
conormal to $M(\infty)$ by the Hopf boundary maximum principle,
and that we can view $M(\8)$ as an embedded minimal disk in $W$.
Since $M(\infty)$ can be viewed as the limit of the disks $M(n)$, each of which
lies in the mean convex ball $B_n$ bounded by $S_n$, then we conclude that $\partial W$ has
constant mean curvature $h_0(X)$ with respect to the inward pointing normal field to its boundary
(thus, $W$ has mean convex boundary of constant mean curvature).

The vector field $K_1$ 
is clearly a Killing field on $W$ endowed with
the pulled-back metric by $T_{\infty }$. Note that $K_1$ 
arises from the proper action of $\G=\G (t)$, which we can view as a $1$-parameter
subgroup of isometries of $W$.

We claim that there exists a nowhere zero Killing field $K_2$ in {$W$ which is tangent to
the foliation $\cF$ defined in \eqref{fol}, bounded on the end of $W$ given by the end representative
$D_0\times [0,\infty )$} 
and such that the closure in $W$ of every integral curve of $K_2$ in $\Int(W)$
is a compact arc with end points in $\partial W$.   To see this,
first note that by Lemma~\ref{cor:bounded}
there exists a nonzero horizontal
right invariant vector field $V$  in $X$ that is bounded in the halfspace
$\R^2\rtimes _A[0,\infty )$. Let $K_2$ be the Killing field in $W$ defined as the
pullback by $T_\infty$ of $V$. Since $V$ is {everywhere horizontal and}
bounded in $\R^2\rtimes _A[0,\infty )$, then $K_2$ is {tangent
to $\cF$ and} bounded on the end representative $T_\infty ^{-1}(\R^2\rtimes _A[0,\infty ))$ of $W$.
Clearly, both noncompact domains $D_0\times [0,\infty )$, $T_\infty ^{-1}(\R^2\rtimes _A[0,\infty ))$
represent the same end of the solid cylinder $W$. Finally,
since $T_\infty $ is proper, then the closure of every integral curve of $K_2$ in $\Int(W)$
is a compact arc in a disk $D_t$ for some $t\in \R$, with end points in
$\partial W$. Thus our claim in this paragraph is proved.

By Theorem~\ref{flux}, the CMC flux (see Definition~\ref{def:flux})
\[
\mbox{Flux}(\partial W,\partial D_0
\times \{ 0\} ,K_1)\neq 0.
\]
We will arrive to a contradiction by showing that the above CMC flux is actually zero. Note
that $\a _n$ is homologous to zero in $S_n$, and so, the homological invariance of
the CMC flux gives that $\mbox{Flux}(\partial B_n,\a _n,K_{\Sigma}^n)=0$ for all $n$. But this
sequence of numbers converges as $n\to \infty $
to $\mbox{Flux}(\partial W,\partial D_0
\times \{ 0\} ,K_1)$, which gives the desired contradiction.

\section{Proof of Theorem~\ref{topR3case} when $X$ is isomorphic to $\sl $}
 \label{sec:sl2R}
In this section we will prove Theorem~\ref{topR3case} when $X$ is a metric Lie group isomorphic to $\sl$.
As explained at the beginning of Section~\ref{sec:semidirect}, to prove Theorem~\ref{topR3case}
it suffices to prove that if $S_n$ is a sequence of constant mean curvature spheres in $X$ that lie
in the connected component $\mathcal{C}$ of the space of index-one spheres in $X$ with $H>h_0(X)$
defined in Proposition~\ref{propsu2}, and for
which ${\rm Area}(S_n)\to \8$ as $n\to \8$, then there exists a pointed limit immersion
$f\colon (\Sigma,p)\la (X,e)$ of this sequence such that $f(\Sigma)$ is an entire Killing graph in $X$
with respect to a right invariant Killing vector field in $X$.
Also by the discussion at the beginning of Section~\ref{sec:semidirect}, we can also assume that
there exists a pointed limit immersion $f\colon (\Sigma,p)\la (X,e)$ of the spheres $S_n$ such that
properties (S1), (S2) hold.

By item~5 of Lemma~\ref{lemma4.1}, $H(X)>0$.
Therefore, by \eqref{hoche} we have $h_0(X)>0$, which implies that
no immersions in $\Delta(f)$ have constant left invariant Gauss map
(because  two-dimensional subgroups of $X$ are minimal, see Corollary 3.17 in~\cite{mpe11}).
In particular, $f(\Sigma )$ is not a two-dimensional subgroup of $X$
and the right invariant vector field $K_{\Sigma}$ to which $f(\Sigma)$ is everywhere tangent is unique up to scaling.

Theorem~\ref{topR3case} when $X$ is isomorphic to $\sl$ will follow directly from the next result.

\begin{theorem}
\label{th:aes}
In the above conditions, $f(\Sigma)$ is an entire Killing graph with respect to a nonzero right invariant vector field on $X$.
Moreover, $f$ satisfies the periodicity condition (S2.3) stated at the beginning of Section~\ref{sec:semidirect}.
\end{theorem}

As explained in the paragraph just after the statement of
Theorem~\ref{semigraph}, case (S2.3) must occur for $f$ provided that $f(\Sigma )$ is
an entire Killing graph. So we only need to prove that $f(\Sigma)$ is an entire Killing graph
with respect to some nonzero right invariant vector field.
The proof of this property will be divided into four steps, depending on the character of
the right invariant vector field $K_{\Sigma}$ (elliptic, hyperbolic or parabolic), and on the situations
(S2.1), (S2.2) and (S2.3). In Section~\ref{sec:notel} we will show that $K_{\Sigma}$
cannot be elliptic. In Section~\ref{paraok} we will prove that Theorem~\ref{th:aes} holds if
$K_{\Sigma}$ is parabolic. In Section~\ref{sec:c3ok} we will prove that
Theorem~\ref{th:aes}
holds if $K_{\Sigma}$ is hyperbolic and case (S2.3) occurs. In Sections~\ref{sec:c2no}
and~\ref{sec:c1no} we will respectively show that cases (S2.2) and  (S2.1) are impossible
when $K_{\Sigma }$ is hyperbolic. This discussion exhausts all possible cases,
and thus proves Theorem~\ref{th:aes}, and therefore
Theorem~\ref{topR3case} when $X$ is isomorphic to $\sl$.

\subsection{$K_{\Sigma}$ cannot be elliptic}
\label{sec:notel}
\begin{lemma}
\label{elliptic}
 $K_{\Sigma}$ is not generated by an elliptic
$1$-parameter  subgroup of $\sl$.
\end{lemma}
\begin{proof}
Let $\G^E$ be any elliptic 1-parameter subgroup of
$\sl$ and let $K^E$ be the right invariant, elliptic vector
field on $\sl $ generated by the left action of $\G^E$
on $\sl $.  Each elliptic
1-parameter subgroup of $\sl $ consists of the set of
liftings to $\sl $ of all rotations of $\Hip ^2$ around a given point; see Section \ref{subsecsl}.
If $a\in \sl $ satisfies $(l_a)_* (K^E)=\l \, K^E$ for some $\l \in \R-\{ 0\} $, then $a \G^E a^{-1}=\G^E$,
see e.g., the proof of Lemma~\ref{lemma3.6}. Once here,
it is well known that on an elliptic subgroup this
is only possible if $a \in \G^E$.

Suppose now that $K_{\Sigma}$ is generated by the left action on $\sl $
of an elliptic 1-parameter subgroup $\G_{\Sigma }=\G ^E$ of $\sl $ and we will obtain a contradiction.
If case (S2.3) holds, then there exists $a\in X-\G_{\Sigma}$
 such that 
 $l_{a}(f(\Sigma ))=f(\Sigma )$.
As the right invariant vector field
$(l_{a})_*( K_{\Sigma })$ is everywhere tangent to $l_{a}(f(\Sigma ))$
and $K_{\Sigma }$ is unique up to scaling among nonzero right invariant vector
fields that  are everywhere tangent  to $f(\Sigma )$, then we conclude that
$(l_{a})_*( K_{\Sigma })=\l \, K_{\Sigma }$ for some $\l \in \R $, $\l \neq 0$.
Thus, the first paragraph
in this proof ensures that $a\in \G^E$, a contradiction. Therefore,
case (S2.3) cannot hold.

In both cases (S2.1) and (S2.2),
Remark~\ref{rem7.1} implies that
$f(\Sigma )$ is obtained by the flow under the elliptic right invariant vector
field $K_{\Sigma }$ of an immersed closed curve contained in $\sl /K_{\Sigma }$.
In particular, since $K_{\Sigma}$ is Killing, the distance in $X$ of $f(\Sigma )$ to $\G ^E$ is bounded,
or equivalently,
$f({\Sigma})$ is contained in a solid metric cylinder in $X$ of fixed radius $r>0$
centered along  $\Gamma^E$,
\[
\mathcal{W}(\G^E,r)=\{ x\in X \ | \ d_X(x,\G^E )\leq r\} ,
\]
where $d_X$ denotes distance in $X$. As every elliptic 1-parameter subgroup of
$X$ contains the center $Z$ of $\sl $, then so does $\Gamma^E$, from where
we deduce that $\mathcal{W}(\G^E,r) $
is invariant under every left translation in $Z$.
Also, it is clear that every right invariant vector field of $\sl $ descends to
the quotient group $\sl /Z$ given by the action on $\sl$ of the group of left translations by elements of $Z$.
As $\mathcal{W}(\G^E,r) /Z$ is compact,
the restriction to $\mathcal{W}(\G^E,r)$ of every right
invariant vector field in $\sl $ is bounded in $\mathcal{W}(\G^E,r)$.

 Let $V$ be any nonzero right invariant 
vector field in $X$
tangent to $f(\Sigma)$ at ${f}({p})=e$ and linearly independent from
$K_{\Sigma }$. As $V$ is not
everywhere tangent  to $f(\Sigma )$ and $V$
is bounded in $\mathcal{W}(\G^E,r)$, then the Jacobi function
$J=\langle V, N \rangle $ is not identically zero, vanishes at $p$
 and it is bounded on $\Sigma $ (here $N$ is the unit
normal vector field to $f$). Note that $J$ changes sign on
$\Sigma$ by the maximum principle.
As $f:\Sigma \mapsto X$ is a complete stable $h_0(X)$-surface
with at most quadratic area growth, the existence of the function $J$ contradicts~\cite[Corollary~1]{mper1}
(see the last paragraph of the proof of Lemma~\ref{lem:horizontal} for a similar argument).
This contradiction proves the lemma.
\end{proof}

\subsection{Proof of Theorem~\ref{th:aes} when $K_{\Sigma}$ is parabolic}
\label{paraok}
In order to prove Theorem~\ref{th:aes} if  $K_{\Sigma}$ is parabolic, we will use the following auxiliary result:

\begin{lemma}
\label{hyper-para}
Let $\G ^P,\G ^H$ be two 1-parameter subgroups of $\sl $ such that $\G^P$ is parabolic,
$\G ^H$ is hyperbolic and they generate a two-dimensional  subgroup  of
$\sl $. Let $K^P,K^H$ denote the right invariant vector fields on $\sl $ generated by
$\G ^P,\G ^H$. Then, there exists a function $h\in C^{\infty }(\sl )$ such that
the vector field $K^H-h\, K^P$ satisfies:
\begin{enumerate}[(1)]
\item $K^H-h\, K^P$ is bounded in $X=(\sl,\esiz,\esde)$ (regardless of the left invariant metric on $\sl $).
\item If the isometry group of $X$ has dimension four,
then $K^H-h\, K^P$ has constant length.
\end{enumerate}
\end{lemma}
\begin{proof}
First note that if we decompose $K^H$ as $K^H=h\, K^P+V$ for some $h\in
C^{\infty }(\sl )$ and some smooth vector field $V$ on $\sl $, then the
boundedness of $V$ in $X=(\sl,\esiz,\esde)$ does not depend on the left invariant metric $\esiz,\esde$ on $\sl $,
because all such metrics are quasi-isometric by the identity map in $\sl$.
Also, note that the particular bound for $V$ on each specific $X$ might depend on the metric $\esiz,\esde$.
Thus, it suffices to prove item~2 of the lemma.

We next prove item~2.
Let $\Hip ^2=\Hip ^2_{\theta }$ be the 2-dimensional subgroup of $\sl $ that contains
$\G^P,\G^H$ (see the description of $\Hip ^2_{\theta }$ before Lemma~\ref{lemma5.3}).
Consider the Lie group given by the semidirect product $\R \rtimes _{(1)}\R =
\{ (x,y)\ | \ x,y\in \R \} $, endowed with the group operation
\begin{equation}\label{gop}
(x,y)*(x',y')=(x+e^yx',y+y').
\end{equation}
Let $\Phi_1 \colon \Hip^2 \to \R \rtimes _{(1)}\R$ be a group isomorphism, which exists
since both $\Hip^2,\R \rtimes _{(1)}\R $ are noncommutative Lie groups of dimension two.
As $\Phi _1$ is an isomorphism, the push-forward vector fields $\wt{K}^P=(\Phi _1)_*(K^P)$,
 $\wt{K}^H=(\Phi _1)_*(K^H)$ are nowhere zero and right invariant on $\R \rtimes _{(1)}\R $.
Furthermore, as the unique parabolic one-dimensional subgroup of $\R \rtimes _ {(1)}\R $ is $\{ (x,0)\ | \ x\in \R \} $,
then we conclude that there exists $\l \in \R -\{ 0\} $ such that $\wt{K}^P=\l \,\partial _x$.
Since $\{ \partial _x, x\, \partial _x+\partial _y\} $ is a basis of right invariant vector fields on
$\R \rtimes _{(1)}\R$ and $\l\neq 0$, there exist $\mu _1,\mu _2\in \R $ such that
\begin{equation}
\label{eq:lem7.3}
\wt{K}^H=\mu _1\wt{K}^P+\mu _2\, (x\, \partial _x+\partial _y)=
\left( \mu _1+\frac{\mu _2}{\l }x\right) \wt{K}^P+\mu _2\partial _y.
\end{equation}
Consider the left invariant metric $\langle ,\rangle $ on $\R\rtimes _{(1)}\R $ that makes
$\Phi_1$ an isometry (recall that $\H^2$ has the induced metric from $\sl $).
As $\partial _y$ is 
left invariant on $\R \rtimes _{(1)}\R$, then the length of $\partial _y$ with
respect to $\langle ,\rangle $ is constant (nonzero).
Hence, equation (\ref{eq:lem7.3}) implies that  the length of
$\wt{K}^H-\wt{h}\, \wt{K}^P$ with respect to $\langle,\rangle $
is 
constant, where $\wt{h}= \mu _1+\frac{\mu _2}{\l }x
\in C^{\infty }(\R \rtimes _{(1)}\R)$.
This implies that if we call $h=\wt{h}\circ \Phi _1\in C^{\infty }(\Hip^2)$, then the vector field
$(K^H)|_{\Hip^2}-h\, (K^P)|_{\Hip^2}$ has constant length on $\Hip^2$.

Parameterize by $\{ c(t)\ | \ t\in \R\} $ the 1-parameter subgroup $\G_3$
of $\sl $ generated by $E_3$.
As the family $\{ r_{c(t)}(\Hip^2)\ | \ t\in \R \} $ of right cosets of $\Hip^2$
is a foliation of $\sl $, we can
extend $h$ to a smooth function, also called $h\in C^{\infty }(\sl)$, so that
$h$ is constant along the orbits of the right action by elements in $\G_3$.
Suppose now that the left invariant metric on
$X$ has isometry group of dimension four. Thus, the right translations
$r_{c(t)}$ by elements in $\G_3$ are isometries of the ambient metric
(because $\G_3'(0)=(E_3)_e$ given by (\ref{eq:lframesl}) and $E_3$
is a left invariant Killing vector field of every left invariant metric $\sl $ whose
isometry group has dimension four).
This property and the right invariance of $K^H,K^P$
imply that the length of $K^H-hK^P$ is constant along the orbits of
the right action by elements of $\G_3$. As
$(K^H)|_{\Hip^2}-h\, (K^P)|_{\Hip^2}$ has constant length on $\Hip^2$,
we conclude that $K^H-h\, K^P$ has constant length on $X$
and item~2 of Lemma~\ref{hyper-para} is proved.
\end{proof}

The next lemma implies that Theorem~\ref{th:aes} holds in the case that $K_{\Sigma}$ is parabolic,
which is the objective of the present Section~\ref{paraok}.
\begin{lemma}
\label{parabolic}
Suppose that $K_{\Sigma }$ is parabolic. Let $\Hip ^2_{\t }$
be the two-dimensional subgroup of $\sl $ that contains the 1-parameter subgroup $\G _{\Sigma }$
generated by $K_{\Sigma }$.
Let $\G^H$ be any hyperbolic $1$-parameter subgroup of $\Hip ^2_{\t }$ and
$K^H$ the nonzero right invariant vector field on $\sl $ generated by $\G^H$.
Then, $f(\Sigma )$ is a smooth entire Killing $K^H$-graph in $X=(\sl ,\langle,\rangle )$.
\end{lemma}
\begin{proof}
Consider the foliation $\mathcal{F}$
of $\sl $ by right
cosets of $\Hip ^2_{\t }$. As $\G ^H$ (resp. $\G_{\Sigma }$) is a 1-parameter
subgroup of $\Hip ^2_{\t }$, then
the integral curves of $K^H$ (resp. $K_{\Sigma }$)
are entirely contained in right cosets of $\Hip ^2_{\t }$; see property (P5) just before Claim~\ref{assslabprop1}.
Since $f(\Sigma )$ is ruled by integral curves of $K_{\Sigma }$,
we have that if $f(\Sigma )$ intersects a leaf $r_x(\Hip ^2_{\theta })$ of $\mathcal{F}$ at a point $x$,
then the integral curve $r_x(\G _{\Sigma })$ of $K_{\Sigma }$ passing through $x$ is also contained in
$f(\Sigma )\cap r_x(\Hip ^2_{\theta })$, and the angle that $f(\Sigma )$ makes
with $r_x(\Hip ^2_{\theta })$ along $r_x(\G _{\Sigma })$ is constant,
since left translations by elements of $\Gamma_{\Sigma}$ are isometries of $X$.

Fix one of the leaves of $\mathcal{F}$, say $r_x(\Hip ^2_{\t })$, $x\in \sl $.
Since each of the integral curves of $K^H$ contained
in $r_x(\Hip ^2_{\t })$ intersects exactly once the unique integral curve
$r_x(\G _{\Sigma })$ of $K_{\Sigma }$ contained in $r_x(\Hip ^2_{\t })$, then it
follows that Lemma~\ref{parabolic} will hold provided that we show that
$f(\Sigma )$ intersects each of the right cosets $r_x(\Hip ^2_{\t })$ transversely.

By straightforward modifications of the arguments in Assertion~\ref{ass6.5}
(note that in that assertion, $K_{\Sigma }$
was assumed to be {\it hyperbolic} while we are now assuming that
$K_{\Sigma }$ is parabolic; nevertheless
the arguments in Assertion~\ref{ass6.5}
still hold since they will be applied to the foliation
by right cosets of $\Hip ^2_{\t }$, which are tangent to both $K_{\Sigma }$ and $K^H$),
we deduce that one of the two following possibilities holds for $f$:
\begin{enumerate}[(V1)]
\item $f(\Sigma)$ makes angles with the leaves of $\mathcal{F}$
that are bounded away from zero.
\item $f(\Sigma )$ is contained in a smallest slab-type region between two different
leaves $L:= r_x(\Hip^2_{\t })$, $L':=r_y(\Hip ^2_{\t })$ of $\cF$,
for some $x,y\in \sl$, and $f(\Sigma)$ is tangent to $L$ (resp. $L'$)
along an integral curve $\G_x$ (resp. $\G_y)$ of $K_{\Sigma}$
(in other words, the conclusions of Claim~\ref{assslabprop1} hold).
\end{enumerate}
A straightforward modification of the arguments
in the paragraph just before Claim~\ref{assslabprop1} implies
that if case (V1) holds, then Lemma~\ref{parabolic} also holds.
Therefore, we will assume that case~(V2) holds and we will find a contradiction.

Consider the Jacobi function $J=\langle K^H, N\rangle$ on $\Sigma$, where $N$ stands for
the unit normal vector to $f$.
As $K_{\Sigma}$ is parabolic and everywhere tangent to $f(\Sigma )$, then
Lemma~\ref{hyper-para} applied to $K^P:=K_{\Sigma}$ and $K^H$
ensures that $J$ is bounded on $\Sigma $.
Also, note that $J$ vanishes along the curves in $f^{-1}(L\cup L')$,
which exist by condition (V2). Since each nodal domain $\Omega \subset \Sigma $ of $J$
is a smooth infinite strip invariant under a $1$-parameter group of isometries 
(the ones generated by the flow of $K_{\Sigma}$), then $\Omega $ has linear area growth.
Hence, the closure $\ov{\Omega}$ of $\Omega $ is a parabolic Riemannian surface with boundary, 
in the sense that there exists a sequence
$\{ \varphi _j\} _j\subset C^{\infty }_0(\ov{\Omega })$ such that $0\leq \varphi _j\leq 1$, the
$\varphi _j^{-1}(1)$  form an increasing exhaustion of $\ov{\Omega}$
and $\int _{\Omega }|\nabla \varphi _j|^2\leq \frac{1}{j}$ for each $j$. As $f$ is stable,
there exists $v\in C^{\infty }(\Sigma )$ such that $v>0$ and $Lv=0$ in $\Sigma $, where
$L$ is the Jacobi operator of $f$. Since $J$ is bounded in $\Sigma$ and 
changes sign, then we have a contradiction with Lemma~\ref{lemmanew}
below.
This contradiction proves the lemma.
\end{proof}

\begin{lemma}
\label{lemmanew}
Suppose $\Sigma $ is a complete Riemannian surface without boundary. Suppose that 
$L=\Delta +P$ is a Schr\"{o}dinger operator with potential $P\in C^{\infty }(\Sigma )$.
Let $u,v\in C^{\infty }(\Sigma )$ satisfying $Lu=Lv=0$, with $u$ not identically zero and $v>0$ in $\Sigma $. If 
$\Omega \subset \Sigma$ is a nodal domain of $u$ such that 
$\ov{\Omega }$ is parabolic and $u$ is bounded in $\Omega$, then $u$ is a nonzero multiple of $v$
and $\Omega =\Sigma $.
\end{lemma}
\begin{proof}
The arguments in the proof of Theorem~1 in~\cite{mper1} are still valid in our current setting
and we leave the details to the reader.
\end{proof}

\subsection{Proof of Theorem~\ref{th:aes} when $K_{\Sigma}$ is hyperbolic and case
(S2.3) holds}
\label{sec:c3ok}
\begin{lemma}
\label{hyperbolicnew}
Suppose that $K_{\Sigma}$ is generated by a
1-parameter hyperbolic subgroup $\G^H$ of $\sl$, and that case
(S2.3) holds. Let $\Hip ^2_{\t }$
be one of the two two-dimensional subgroups of $\sl $ containing $\G^H$.
Let  $\G^P, K^P$ be the parabolic subgroup of $\H^2_{\t }$ and
 the right invariant vector field on $\sl $ generated by $\G^P$, respectively.
Then, $f(\Sigma )$ is a smooth entire Killing $K^P$-graph.
\end{lemma}
\begin{proof}
We argue in a similar way as in the proof of Lemma~\ref{parabolic}.
Consider the foliation $\mathcal{F}=\{ r_t(\Hip _{\t }^2)\ | \ t\in \G_3\} $ of $\sl$ by
right cosets of $\Hip _{\t }^2$, where $\G_3$ is the elliptic $1$-parameter subgroup associated to $E_3$
given by equation~(\ref{eq:lframesl}). Recall that
the leaves of $\mathcal{F}$ are equidistant surfaces to $\Hip _{\t }^2$.

Suppose first that the angles that $f(\Sigma)$ makes with $\cF$
are bounded away from zero. Then, $f(\Sigma )$ is an entire Killing graph with respect to $K^P$
as explained just before Claim~\ref{assslabprop1},  and so Lemma~\ref{sec:c3ok} holds in this case.
Otherwise, we can apply the Slab Property in $\sl $ (Claim~\ref{assslabprop1}) to conclude that
$f(\Sigma )$ is contained in a smallest slab-type region $S$ between two different
leaves $L,L'$ of $\cF$.
As we are in the conditions of case (S2.3), there is an element
$a\in \sl -\G^H$ such that $l_{a}(f(\Sigma ))=f(\Sigma )$.
As $f(\Sigma )$ is everywhere tangent
to the right invariant vector field $K_{\Sigma }$ and $f(\Sigma )$ is not a two-dimensional
 subgroup of $\sl $, then there exists $\l \in \R -\{ 0\} $ such that
$(l_a)_*(K_{\Sigma })=\l \, K_{\Sigma }$.
Observe that the 1-parameter subgroup of $\sl$ generated by $(l_a)_*(K_{\Sigma })$ is
$a\G_{\Sigma}a^{-1}=(l_a\circ r_{a^{-1}})(\G_{\Sigma})$
(see e.g., the proof of Lemma~\ref{lemma3.6}),
hence this last 1-parameter subgroup must be a reparameterization of
$\G_{\Sigma}$. By Lemma~\ref{lemmanew2} below, either $a^2$ lies in the center $Z$ of $\sl $,
or $a$ is the product of an element $h\in \G^H$ with an element $b\in Z$. In the first case,
$a^2=e$ since otherwise the left translation by a large power $a^{2k}$ of $a^2$ 
sends the topological slab $S$ into another topological slab disjoint from $S$, which contradicts that
$f(\Sigma )$ is invariant under the left translation by $a^{2k}$ and $f(\Sigma )$ is contained in $S$.
But $a^2=e$ implies $a=e$, which contradicts our hypothesis that $a\in \sl -\G^H$. Therefore,
$a=bh$ with $h\in \G^H$ and $b\in Z$. Observe that $b\neq e$ since $a\notin \G^H$. Therefore,
a large power $a^n=b^nh^n$ of $a$ sends the topological slab $S$ into another topological slab disjoint from $S$
and we find a contradiction as before.
This proves Lemma~\ref{hyperbolicnew}.
\end{proof}

\begin{lemma}
\label{lemmanew2}
Suppose $a\in  \sl $ satisfies $a\G^Ha^{-1}=\G^H$ for some hyperbolic 1-parameter subgroup $\G^H$ of $\sl $.
Let ${\bf P}\colon \sl \to \psl=\sl/Z$ be the quotient homomorphism where $Z$ is the center of $\sl$.
Then ${\bf P}(a)\in {\bf P}(\G^H)$ or $a^2\in Z$.
\end{lemma}
\begin{proof}
We identify ${\bf P}(\G^H)$ with the set of hyperbolic translations of the Poincar\'e disk $\D$ along a geodesic $\g $.
Since ${\bf P}(a){\bf P}(\G^H){\bf P}(a^{-1})={\bf P}(\G^H)$, we conclude that ${\bf P}(a)$ lies in the
normalizer subgroup of ${\bf P}(\G^H)$ in $\psl$. In this case, it is well known that ${\bf P}(a)\in {\bf P}(\G^H)$ or
${\bf P}(a)$ is an elliptic isometry of order two around some point in $\g$. In the second case, ${\bf P}(a)^2$
is the identity transformation of $\D$, or equivalently $a^2$ lies in the kernel of ${\bf P}$ which is $Z$.
\end{proof}

\subsection{Case (S2.2) is impossible if $K_{\Sigma }$ is hyperbolic}
\label{sec:c2no}
The proof that $f(\Sigma)$ cannot be in the conditions of case (S2.2) if $K_{\Sigma }$ is
hyperbolic  will be an adaptation to
the $\sl$ setting of our arguments in Section~\ref{sec:nos}. 
In order to do this, we prove first a radius estimate (Lemma~\ref{ass:box}) for compact stable
minimal surfaces with boundary contained in some special regions of $\sl$, similar in spirit to
Theorem 1.1 in~\cite{mmp1}. To start, we next associate to every
hyperbolic 1-parameter subgroup $\G ^H$ of $\sl $ a cylindrical-type
neighborhood of $\G ^H$ in $X$, which depends on the left invariant metric
and on a number $r>0$. We will call this neighborhood of $\G^H$ an {\it infinite box.}

\begin{definition}
\label{def12.7}
  {\rm
  Let $\G^H$ be a hyperbolic 1-parameter subgroup of $\sl $, and let
$\g \subset \Hip ^2(-4)$ be the geodesic with the same extrema as the
circle arc $\Pi (\G ^H)\subset \Hip ^2(-4)$ (recall that $\Pi $ was defined in~(\ref{eq:PI})).
Note that $\Pi (\G ^H)$ is at constant hyperbolic distance
$r_0$ from $\g $, for some $r_0\geq 0$. Given $r>r_0$, let
$\g _r^+,\g _r^-\subset \Hip ^2(-4)$ be the two arcs at distance $r$ from
$\Pi (\G ^H)$ and let $P^+(r)=\Pi ^{-1}(\g _r^+)$, $P^-(r)=\Pi ^{-1}(\g _r^-)$
be the topological planes inside $\sl $ obtained after lifting these arcs through
$\Pi $. Let $\H=\Hip _{\theta}^2$ be one of the two two-dimensional subgroups of $\sl $ that
 contain $\G^H$, and call $\Hip \, a_r^+, \Hip \, a_r^-$ the right cosets of $\Hip $ 
at distance $r$ from~$\Hip $ (here $a_r^+,a_r^-$ are some elements in
$\sl $). Observe that each of the surfaces $P^+(r),P^-(r)$ intersects transversely
both $\Hip \, a_r^+$ and $\Hip \, a_r^-$, along integral curves of the hyperbolic
right invariant vector field $K^H$ generated by the left action of $\G^H$ on $\sl $.

With these ingredients,
we define the {\it infinite box} $\mathcal{B}(\G^H,r)$ to be the closure of the component of
$\sl -[P^+(r)\cup P^-(r)\cup (\Hip \, a_r^+)\cup (\Hip \, a_r^-)]$
that has portions of  all four surfaces in its boundary.
Observe that $\G^H$ is contained in the interior of $\mathcal{B}(\G^H,r)$.
}
\end{definition}

\begin{lemma}
\label{ass:box}
Consider an infinite box $\mathcal{B}(\G^H,r)$ as before. Then:
\begin{enumerate}[(1)]
\item
$\mathcal{B}(\G^H,r)$ is invariant under the left translations by elements in $\G^H$.
\item
Every compact, immersed minimal surface $M\subset X$ whose boundary lies in $\mathcal{B}(\G^H,r)$
  satisfies that $M\subset \mathcal{B}(\G^H,r)$.
\item
There exists a constant $R(r)>0$ such that for every compact, stable, immersed minimal surface
$h\colon M \la X$ with $h(\partial M)\subset \mathcal{B}(\G^H,r)$, the intrinsic radius of $M$ is less than $R(r)$.
\end{enumerate}
\end{lemma}
\begin{proof}
Item 1 follows from the definitions of the surfaces in the boundary of
$\mathcal{B}(\G^H,r)$, all of which are invariant under left translations by elements in $\G ^H$.
Item 2 is a direct consequence of the maximum principle and the following facts:
 \begin{itemize}
 \item For every $r'\geq r$, $\mathcal{B}(\G^H,r')$ has piecewise smooth
mean convex boundary (this follows from item~4 of
 Lemma~\ref{lemma4.1} and the fact that every right coset $\Hip \, a$
 is a left translation of some two-dimensional
 subgroup of $X$, hence minimal), and the inner angles between adjacent
smooth surfaces in this boundary are less than $\pi $. Furthermore,
$\mathcal{B}(\G^H,r)\subset \mathcal{B}(\G^H,r')$.
\item $\bigcup _{r'>r}\mathcal{B}(\G ^H,r')=X$.
\end{itemize}

Suppose item~3 of the lemma fails to hold. Then, there exists a sequence of
compact, stable, immersed minimal surfaces $h_n\colon M_n \la X$ such
that $h_n(\partial M_n)\subset \mathcal{B}(\G^H,r)$ and points $p_n\in
M_n$ such that the intrinsic distances from $p_n$ to the boundary of
 $M_n$ satisfy $d_{M_n}(p_n,\partial M_n)>n$ for all $n\in \N $.
Next we show that after appropriate left translations and passing to
a subsequence, we can take limits of the $h_n$. Let $\be $ be the
geodesic in $\Hip ^2(-4)$ that passes through $\Pi(e)$ and is orthogonal to $\g$. Note that each
integral curve in $\sl $ of the hyperbolic right invariant vector field $K^H$ generated
by $\G ^H$  intersects the surface $\Pi^{-1}(\be )$ exactly
once; in particular, the integral curve $\G ^H\, h_n(p_n)=r_{h_n(p_n)}(\G ^H)$ of $K^H$ passing
through the point $h_n(p_n)$ intersects the compact surface $Y=\Pi^{-1}(\be )\cap
\mathcal{B}(\G^H,r)$ just once (to see this, observe that by the already proven items 1, 2 of this lemma,
$\cB(\G^H,r)$ is invariant under the left action of $\G^H$
and $h_n(M_n)\subset \mathcal{B}(\G ^H,r)$). Hence, after replacing $h_n$ by a
left translation by some element in $\G^H$, we will have that
$h_n(p_n)\in Y$ and $h_n(M_n)$ still lies in $\mathcal{B}(\G ^H,r)$.
As $Y$ is compact, after choosing a subsequence we may assume that
$h_n(p_n)$ converges to a point $q_{\infty}\in Y$.

By curvature estimates for stable minimal surfaces (Schoen~\cite{sc3}, Ros~\cite{ros9}),
the immersions $h_n$ restricted to the intrinsic balls
\[
B_{M_n}(p_n,n/2)=\{ x\in M_n\ | \ d_{M_n}(p_n,x)<{\textstyle \frac{n}{2}}\}
\]
have uniformly bounded second fundamental forms.
Hence, after choosing a subsequence, we
obtain a complete minimal immersion
$h_{\infty}\colon M_{\infty} \la \mathcal{B}(\G^H,r)$
with bounded second fundamental form that is a limit of the restriction of
the $h_n$ to compact domains in $M_n$, and such that $h_{\infty }(p_{\infty })=
q_{\infty }$ for some point $p_{\infty }\in M_{\infty }$.
Consider the set
\[
\cM=\ov{\{ a \, h_{\infty }(M_{\infty})=l_a\left( h_{\infty }(M_{\infty})\right) \mid a \in \G^H\}} \subset \; \mathcal{B}(\G^H,r),
\]
where for any $A\subset X$, $\ov{A}$ is the closure of $A$.  Since
$h_{\infty }(M_{\infty})$ has bounded second fundamental form, then
given any point $q\in \cM $, there exists a compact, embedded minimal
disk $D(q) \subset \cM$ with $q \in \Int(D(q))$; specifically, the disk $D(q)$ is
a limit of embedded geodesic disks of fixed small geodesic radius
centered at points $q_n$ in some left translate $l_{a_n}\left( h_{\infty}
(M_{\infty})\right) $ such that $\lim_{n\to \infty} q_n=q$, where $a_n$ is
some element in $\G^H$.  Since $\cM \cap Y$ is a compact set, there
exists a largest positive $r'\leq r$ such that the set
\[
W=[\Hip \, a_{r'}^+\cup\Hip \, a_{r'}^-]\cap \cM \cap Y
\]
is nonempty. Note that the $\G ^H$-invariance of $[\Hip \,
a_{r'}^+\cup\Hip \, a_{r'}^-]\cap \cM $ implies that $r'$ is also the
largest positive number such that $[\Hip \, a_{r'}^+\cup\Hip
\, a_{r'}^-]\cap \cM $ is nonempty. Let $q\in W$. Then by the maximum
principle for minimal surfaces, the set  $\Hip \, a_{r'}^+\cup\Hip \,
a_{r'}^-$ must contain the aforementioned embedded minimal disk
$D(q)$.  But in this case, $\mathcal{M}$ would have to contain one
of the surfaces  $\Hip \, a_{r'}^+,\, \Hip \, a_{r'}^-$ which is false
since neither of these surfaces is entirely contained in
$\mathcal{B}(\G^H,r)$. This contradiction proves the lemma.
\end{proof}

With Lemma~\ref{ass:box} at hand, we can now rule out case (S2.2).

\begin{proposition}
\label{notc2}
If $K_{\Sigma}$ is hyperbolic, then case (S2.2) does not occur for $f$.
\end{proposition}
\begin{proof}
We will adapt the arguments of Section~\ref{sec:nos} to our current setting, focusing
on the differences with the case of $X$ being a semidirect product. Suppose
that $K_{\Sigma }=K^H$ is hyperbolic and that case (S2.2) holds.
Recall that $f\colon (\Sigma,p)\la (X,e)$ is a complete stable  pointed immersion
with constant mean curvature $h_0(X)$, that is
obtained as a limit of pointed immersions
$f_n\colon (S_n, p_n)\la (X,e)$ of spheres $S_n\in \mathcal{C}$.
Since each sphere $S_n\in \mathcal{C}$ is Alexandrov embedded,
for each $n\in \N$ there exists a three-dimensional Riemannian manifold
$(B_n,g_n)$ which is topologically a closed ball, together with a Riemannian submersion
$\wt{f}_n\colon (B_n,g_n)\flecha X$ such that $\parc B_n=S_n$, $\wt{f}_n |_{S_n}=f_n$ and
$\parc B_n$ is mean convex. The key step in the proof of Proposition~\ref{notc2}
will be proving that Lemma~\ref{c2no} holds in our situation.
Observe that if we admit the validity of Lemma~\ref{c2no} in our current setting, then
all the arguments in Section~\ref{sec:noa} after the proof of Lemma~\ref{c2no} remain
valid in $\sl $ and the proof of Proposition~\ref{notc2} will be finished.
Therefore, we will only concentrate on the proof of the statement of Lemma~\ref{c2no} in our current $\sl$-setting.

Consider the hyperbolic $1$-parameter subgroup $\G^H$ of $\sl $ generated by $K_{\Sigma}$.
Since $f(\Sigma)$ is an annulus invariant under the flow of $K^{\Sigma }$, we can find
$r_0>0$ such that $f(\Sigma)$ is contained in an infinite box $\mathcal{B}(\G^H,r_0)$
and $\parc \mathcal{B}(\G^H,r_0)$ is at a positive distance from $f(\Sigma)$; in what follows,
$\mathcal{B}(\G^H,r_0)$ will play the role of the solid cylinder $\mathcal{W}(\G ,r_0)$ of
the original proof of Lemma~\ref{c2no} for metric semidirect products.

By item 3 of Lemma~\ref{ass:box}, there exists some $R=R(r_0)>0$ with the following
property:
\begin{enumerate}[$(\star )'$]
\item If $M$ is any compact, stable immersed minimal surface in $X$ whose boundary $\parc M$
is contained in $\mathcal{B}(\G ^H,r_0)$, then the intrinsic radius of $M$
is less than $R$.
\end{enumerate}	

The role of the horizontal planes in the proof of the semidirect product case of  Lemma~\ref{c2no}
will be now played by the leaves of the foliation $\cF' =\{P_t= l_t(\Pi^{-1}(\be))\mid  \ t\in \G^H \equiv \R\} $,
where  $\be $ is the
geodesic in $\Hip ^2(-4)$ that passes through $\Pi(e)$ and is orthogonal to $\g$, where both
$\g$ and $\be$ were previously defined in the proof of Lemma~\ref{ass:box}.
Notice that $\cF'$ is also the foliation of $X$ whose projection by $\Pi$ to $\Hip(-4)$ is the foliation
of $\Hip(-4)$ by the set of geodesics  orthogonal to $\g$, and the leaves of $\cF'$ are minimal topological
planes by Lemma~\ref{lemma4.1}.

As $\Sigma $ is simply connected and $f\colon \Sigma \la X$ is invariant under the
flow of $K_{\Sigma}$, then by the discussion in Remark~\ref{rem7.1}
we may parameterize $f$ as in equation~(\ref{eq:fst}) where
$\G $ is replaced by $\G^H$, and $\a =\a (s)\colon \R\to P_0$
is an immersed $L$-periodic curve parameterized by arc length in the plane $P_0$,
so that $f(0,0)=\alfa(0)=e$, $L>0$ being the length of $\alfa(\R)$.

Observe that $\mathcal{B}(\G ^H,r_0)$ is foliated by integral curves of $K^H$.
As $\mathcal{B}(\G ^H,r_0)\cap P_0$ is compact,
there exists the maximum length $D_0$ of all the arcs of integral curves of $K^H$
that are contained in the compact
piece of $\mathcal{B}(\G ^H,r_0)$ that lies in the topological slab of $\sl $
determined by the leaves $P_{-R-2}$ and $P_{R+2}$ of $\mathcal{F}'$.
Given $s_0\in \R$, the definition
(\ref{arki}) of the compact arc $A(s_0)\subset f(\Sigma)$ of the integral curve of $\G^H$
that passes through $\alfa(s_0)$ whose endpoints lie in $P_{-R-2}$ and $P_{R+2}$, remains valid
after replacing $\G $ by $\G^H$, as well as the property that the length of $A(s_0)$ is at most $D_0$.

For each $j\in \N$, consider the compact region $K^j\subset \Sigma $ defined by
equation (\ref{Kj}). Repeating the arguments in the proof of Lemma~\ref{c2no} verbatim,
each $K^j$ produces compact simply connected domains $K_n^j\subset S_n$
for $n$ large, such that $f_n(K_n^j)\subset \mathcal{B}(\G ^H,r_0)$ 
 and the sequence $\{ f_n |_{K_n^j}\} _n$ converges to $f|_{K^j}$ uniformly as $n\to \8$,
and we can also
find for each $k\in \{1,\dots, j\}$ a small compact disk $\cU^{k}$ centered at $(kL,0)$ in $K^j$
and related compact disks $\cU_n^{k} \subset K_n^j$ such that $\{ f_n |_{\cU_n^k} \} _n$
converges to $f|_{\cU^{k}}$ uniformly as $n\to \8$. Furthermore, the $\cU_n^k$ satisfy
that the intrinsic distance between $\cU_n^i$ and $\cU_n^k$ inside $K_n^j$
is greater than $\delta$ for some $\delta>0$ independent of $n,j,k$.  Notice however, that in the present
setting, while the plane $P_0$ intersects $f(\Sigma)$ transversely along the closed immersed curve $\a(\R)$,
we no longer know that $f_n^{-1}(P_0)$ is a simple closed curve (since $P_0$ is not a two dimensional 
subgroup). Instead,
for any fixed $j$ and for $n$ sufficiently large,
we have that $f_n^{-1}(P_0)\cap K_n^j$ is a connected compact embedded
arc which intersects each of the disks $\cU_n^{k} \subset K_n^j$, where $k\in \{1,\ldots,j\}$.
Let $I$ denote a small compact segment in $P_0$ centered at $\a(0)=e$ and
transversal to $f(\Sigma )$ at this point. The $L$-periodicity of $f(s,t)$
in the variable $s$ and the previous convergence properties allow to find for $n=n(j)$
large enough, $j$ points $q_n^k\in \cU_n^k$ with $f_n(q_n^k)\in I$, for every $k=1,\dots, j$.
In particular, property $(\star \star)$ in the proof of Lemma~\ref{c2no} holds.

Let $J_n^k$ be the compact connected arc of $f_n^{-1}(P_0)\cap K_n^j$
whose endpoints are $q_n^1$ and $q_n^k$, for each $k\in \{1,\dots, j\}$.
Pulling back via $f_n$ the tangent part to $f_n(S_n)$ of $K_{\Sigma}$, we obtain a
nowhere zero vector field $W_n$ on $K_n^j$ for $n$ large enough.
Letting the arc $J_n^k$ flow under $W_n$ between the sets $f_n^{-1}(P_{-R-1})$ and $f_n^{-1}(P_{R+1})$
we obtain a compact disk $D_n^k\subset K_n^j$ which is foliated by  compact arcs
of integral curves of $W_n$, each of which joins points in $(\partial D_n^k)\cap f_n^{-1}(P_{-R-1})$ and
$(\partial D_n^k)\cap f_n^{-1}(P_{R+1})$, for $n$ large. Then, the Jordan curve
$\gamma_n^k :=\parc D_n^k$, can be decomposed
as $\gamma_n^k=A_1^n\cup A_2^n \cup B_1^n \cup B_2^n$,
where properties (T1), (T2) hold after replacing the planes $\R^2\rtimes _A\{ \pm(R+1)\} $
by the topological planes $P_{\pm (R+1)}$.

We next solve for $k\in \{ 1,\ldots, j\} $ fixed (and $n$ large enough) the Plateau problem with boundary $\g _n^k$
in $(B_n,g_n)$, finding a compact stable embedded minimal disk $M_n^k\subset B_n$ of least
area, with $\parc M_n^k =\gamma_n^k$. Furthermore, $M_n^k$ is an immersion
transverse to $\parc B_n$ at the interior points of $A_1^n$ and $A_2^n$. Let
$\Lambda_n$ be the embedded surface $\{q\in B_n \mid \wt{f}_n(q)\in P_0\}$ and $\beta_n^k:=
M_n^k\cap \Lambda_n$. Following the same reasoning as in the proof of Lemma~\ref{c2no} we deduce that
after a small perturbation of $\Lambda_n$ in $B_n$ that fixes its boundary, we may assume that
$\beta_n^k$ contains a compact arc joining $q_n^1$ with $q_n^k$ along which
$M_n^k$ intersects $\Lambda_n$ transversally. Equations (\ref{eq:d>R}), (\ref{eq:7.4}),
(\ref{eq:7.5}) and (\ref{tridi}) hold in our current situation with the only changes of
$\R^2\rtimes _A\{ \pm (R+1)\}$
by $P_{\pm (R+1)}$, $\R^2\rtimes _A\{ 0\} $ by $P_0$ and Property~$(\star)$ by Property~$(\star)'$.
Finally, the arguments in the last five paragraphs of the proof Lemma~\ref{c2no} remain valid without changes,
by noticing that in our $\sl $-setting, the mean curvature $h_0(X)$ of $f$ is positive as
explained just before the statement of Theorem~\ref{th:aes}.
This finishes the sketch of proof of
Lemma~\ref{c2no} when $X$ is isomorphic to $\sl $ and $K_{\Sigma}$ is hyperbolic, and also finishes the proof of
Proposition~\ref{notc2}.
\end{proof}

\subsection{Case (S2.1) is impossible if $K_{\Sigma}$ is hyperbolic}
\label{sec:c1no}
\begin{proposition}
\label{notc1}
If $K_{\Sigma}$ is hyperbolic, then $f(\Sigma)$ cannot be in the conditions of case (S2.1).
\end{proposition}
\begin{proof}
Again the proof of this proposition is basically an adaptation of the contents of Section~\ref{sec:noa} to the setting
of $\sl$. As we did when proving Proposition~\ref{notc2}, we will only focus on the differences with the case
of $X$ being a semidirect product, and we will also use part of the notation established in the proof of
Proposition~\ref{notc2}.

Suppose that case (S2.1) holds and $K_{\Sigma }$ is hyperbolic, with associated 1-parameter
subgroup $\G ^H\colon \R \to \sl $. We use the same notation $f\colon (\Sigma ,p) \la (X,e)$,
$f_n\colon (S_n,p_n)\la (X,e)$ and $\wt{f}_n\colon B_n\to X$ as in the first paragraph of
the proof of Proposition~\ref{notc2}. The desired contradiction will come from application of a
flux argument  (Theorem~\ref{flux}) that we explain next.

Let $\Hip =\Hip ^2_{\t }$ be one of the two two-dimensional subgroups of $\sl $ that contain
$\G^H$. Associated to $\Hip $ and to $\G^H$, we have the infinite boxes
$\mathcal{B}(\G^H,r)$ for all $r>0$.
Since $f(\Sigma)$ is an annulus invariant under the flow of $K_{\Sigma }$, we can find
$r_0>0$ such that $f(\Sigma)$ is contained in $\mathcal{B}(\G^H,r_0)$ and
$\parc \mathcal{B}(\G^H,r_0)$ is at a positive distance from $f(\Sigma)$.
Consider the related foliation $\cF'=\{ P_t\ | \ t\in  \G^H\equiv \R \} $ by topological minimal planes
defined in the proof of Proposition~\ref{notc2}.

For $n$ sufficiently large, consider the simple closed curve component
$\a_n$ of $f_n^{-1}(P_0)\subset S_n$ containing $p_n$ and let
$K_{\Sigma }^n$ be the Killing vector field on $B_n$ obtained by pulling-back $K_{\Sigma }$
 through $\wt{f}_n$. As we are assuming that case (S2.1) holds, $\Sigma $ is an annulus
and $f(\Sigma )$ is the immersed annulus obtained by letting the
immersed closed curve $f(\Sigma )\cap P_0$ flow under $K_{\Sigma}$. Let $\a _{\infty }
=f^{-1}(P_0)$, which is a simple closed curve in $\Sigma $ such that the sequence
$\{ f_n|_{\a _n}\} _n$ converges uniformly to $f|_{\a _{\infty }}$.

As $\cF'$ is a foliation by minimal planes, then the maximum principle implies that the smooth, embedded
least-area disk $M(n)$ obtained by solving the Plateau problem in $B_n$ with boundary $\a _n$,
satisfies that $\wt{f}_n$ restricts to $M(n)$ producing an immersion of $M(n)$ into $P_0$.
As the mean curvatures of the $S_n$ are positive,
then $M(n)$ is transverse to $\partial B_n$ along $\a _n$.

The same argument as in the fifth paragraph of Section~\ref{sec:noa} proves that
there exists an $\eta\in(0,\pi/4)$ such that the angle that $M(n)$ makes with $S_n$ along
$\a _n$ is greater than $\eta$ for all $n\in \N$ large enough, after replacement of the intrinsic
radius estimate Theorem 1.1 in~\cite{mmp1} by item~3 of Lemma~\ref{ass:box}, changing
$\G $ by $\G^H$, and noticing that $h_0(X)\geq H(X)>0$.

The next step consists of proving that if $N_{M(n)}$ stands for the unit normal field to
$M(n)$ in $(B_n,g_n)$, then for $n$ sufficiently large the function $J_n=
g_n(N_{M(n)},K_{\Sigma }^n)\colon M(n)\to \R $ has no zeros along $\a_n$. The proof of
the same property in Section~\ref{sec:noa} does not work now, since it uses
equation~(\ref{eq:6}) which is only valid in a semidirect product.
Instead, we will argue in $\sl $ as follows. Observe that 
\begin{enumerate}[(W)]
\item $K_{\Sigma }$ is everywhere transversal to $P_0$.
\end{enumerate}
Now, if $J_n$ vanishes at some point $q_n\in \a _n$, then $K_{\Sigma }^n(q_n)$ is
tangent to $M(n)$ at $q_n$. Applying the differential of $\wt{f}_n$ at $q_n$, we deduce that
$K_{\Sigma }(f_n(q_n))$ is tangent to $P_0$ at $f_n(q_n)$, which contradicts~(W).
Therefore, $J_n$ has no zeros along $\a _n$ for all $n\in \N$.

Once we know that $J_n$ has no zeros along $\a_n$ for all $n$, we can follow the
arguments in Section~\ref{sec:noa} with the following changes (we are using the same notation as in Section~\ref{sec:noa}):
\begin{itemize}
\item Replace $\G $ by $\G^H$ in equations~(\ref{eq:Tn}) and (\ref{eq:Tinfty}). In this way, we can
obtain a submersion $T_{\infty }\colon \ov{\D}\times \R\to X$.

\item Replace $\R^2\rtimes _A\{ t\} $ in equation~(\ref{fol})
by the horocylinder $\Pi^{-1}(c_t)$ where $c_t$ is the horocircle in $\Hip^2(-4)$
whose end point at $\partial _{\infty }\Hip ^2$ is $\theta $ defined by $\Hip=\Hip^2_{\theta }$,
such that $c_t$ passes through $\Pi (t)$, $t\in \G^H$.
That is, define $D_t=T_{\infty}^{-1}(\Pi ^{-1}(c_t))$,
and then define the foliation $\cF$ as in equation~(\ref{fol}).

\item Define $K_1$ as the nowhere zero, smooth vector field on $\ov{\D}\times \R $ obtained after pulling
$K_{\Sigma}$ by the submersion $T_{\infty }~$. Observe that  each of the integral curves of $K_1$ in $\ov{\D}\times \R $
intersects $D_0$ in a single point.

\item Define $K_2$ as  the nowhere zero vector field on $W:=D_0\times \R$ which is the pullback by
$T_{\infty }$ of a parabolic right invariant vector field $V$ on $X$ associated to $\H$ ($V$ is defined
up to a multiplicative nonzero constant). Observe that $K_2$ is tangent to the foliation $\cF$
and that once we endow $W$ with the pulled-back metric by $T_{\infty }$, $K_2$ is a Killing field. We claim that
$K_2$ is bounded on one of the two ends of $W$: to do this, first note that to prove this
boundedness property, it suffices to assume that $X$ has isometry group of dimension four (as every
left invariant metric on $\sl $ is quasi-isometric to any fixed left invariant metric on $\sl$).
Also notice that
$\Pi ^{-1}(c_0)\cap \H$ is the 1-parameter parabolic subgroup of $\H$, and that 
the restriction of $V$ to $\Pi ^{-1}(A_0)\cap \H$ is bounded, where $A_0$ is the open horodisk in $\Hip ^2$ bounded by $c_0$
(in fact, if we identify $\H$ with the semidirect product
$\R\rtimes _{(1)}\R=\{ (x,y)\ | \ x,y\in \R\} $ as explained after \eqref{gop}, then $V$ identifies with $\partial _x$ up to a multiplicative
constant, $\Pi ^{-1}(A_0)\cap \H$ is the open halfspace $\R\rtimes_{(1)}(0,\infty )$ and $\partial _x$ is
bounded in $\R\rtimes_{(1)}(0,\infty )$). As the right translations by elements in the elliptic subgroup of $\sl $
generated by the vector $(E_3)_e$ given by equation (\ref{eq:lframesl}) are isometries of any $\E (\kappa,\tau )$-metric in $\sl$
and $V$ is right invariant, we conclude that $V$ is bounded in $\Pi^{-1}(A_0)$. Therefore, our claim follows.
\end{itemize}

The above changes allow us to apply Theorem~\ref{flux} to find the same contradiction
with CMC flux as in Section~\ref{sec:noa}. This contradiction
finishes the proof of Proposition~\ref{notc1}.
\end{proof}

\section{Appendix: A nonvanishing result for the CMC flux}

Let $W$ be a Riemannian $(n+1)$-manifold and $M\subset
W$ a two-sided $n$-hypersurface with constant mean curvature $H\in \R $
(throughout this appendix we define the mean curvature of $M$ to be the trace of its second fundamental form divided by $n$).
The CMC flux of $M$ associated to a Killing vector field $K$ on $W$ and to a piecewise smooth
$(n-1)$-chain $\a $ in $M$ which is the boundary of a piecewise smooth two-sided
$n$-chain $\be $ in $W$, is the real number
\begin{equation}
\label{eq:flux}
\mbox{Flux}(M,\a ,K)=\int_{\a }\langle K,\eta\rangle + nH
\int_{\beta }\langle K,N\rangle,
\end{equation}
where $\eta $ is a unit conormal vector to $M$ along $\a $ and $N$ is a
unit normal vector field along $\be $. In order for (\ref{eq:flux}) not to depend on
orientation choices, certain compatibility between $H,\eta $ and $N$ must be satisfied,
see a more precise definition in Definition~\ref{def:flux}.
A key property of Flux$(M,\a ,K)$ is that it does not depend on the homology class of $\a $ in $M$.

The goal of this section is to prove the following general result, in which the CMC flux of an
$n$-hypersurface $M=\partial \Sigma \times \R$ with respect to a
certain Killing vector field $K=K_1$ is nonzero. Theorem~\ref{flux} has been used in
Sections~\ref{sec:noa} and \ref{sec:c1no}.

\begin{theorem}
 \label{flux}
Let $\Sigma$ be a smooth, compact oriented $n$-manifold with nonempty
boundary, $n\leq 6$.  Let $W$ be $\Sigma\times \R$
equipped with a Riemannian metric such that:
\ben[(1)]
\item   The boundary $\partial \Sigma \times \R$ of $W$ is
mean convex, with constant mean curvature $h_0\geq 0$.
\item  There exists  a nowhere zero Killing field $K_1$ of $W$ arising
from a proper action\footnote{This means that if we denote by $\{ \phi _s\ | \ s\in \R \} $
the 1-parameter subgroup of isometries of $W$ corresponding to the flow of the Killing
field $K_1$, then for each $x\in W$ the map $s\in \R \mapsto \phi _s(x)$ is proper.}
of an $\R$-subgroup of the isometry group of $W$.
\een Suppose that there exists a  Killing field $K_2$ in $W$ which is tangent to
the foliation $\{\Sigma\times\{s\} \mid s \in \R\}$, bounded on
either   $\Sigma\times [0,\infty)$ or on  $\Sigma\times (-\infty,
0]$ and such that every integral curve of $K_2$ that intersects $\Int(W)$ is a compact arc with boundary in
$\partial W$. Then: $$\mbox{\rm Flux}(\partial \Sigma \times \R ,
\partial \Sigma \times \{0\},K_1)\neq 0.$$
\end{theorem}

\begin{definition}[CMC flux] \label{def:flux}
 {\rm
Suppose $M$ is a hypersurface with constant mean curvature $H\in \R $ in an oriented  Riemannian
$(n+1)$-manifold $X$, and let $K$ be a  Killing field on $X$.
For each $[\a ]\in H_{n-1}(M)$ in the kernel of the induced
inclusion homomorphism $i_*\colon H_{n-1}(M)\to H_{n-1}(X)$,
consider two homologous piecewise smooth $(n-1)$-cycles $\a,\a _1
\subset M$ representing $[\a ]$ and let $\beta,\be _1$ be piecewise
smooth $n$-chains in $X$ with boundaries $\partial \be =\a $,
 $\partial \be _1=\a _1$ (which exist since
$[\a ]\in \ker (i_*)$). We assume that if $M(\a ,\a _1)$ is the $n$-chain in $M$
whose boundary is $\a _1-\a $, then  $\be -\be _1-M(\a,\a _1)$ is the
boundary of an $(n+1)$-chain $\Omega $ in $X$.
Applying the Divergence Theorem to $K$ in
$\Omega $, we obtain
\begin{equation}
\label{eq:flux2a}
0=-\int _{\be }\langle
 K,N\rangle +\int _{\be _1}\langle K,N_1\rangle -\int _{M(\a ,\a _1)}\langle K,N^M\rangle ,
\end{equation}
where $-N,N_1$ are unit normal vector fields to
$\be ,\be _1$ (defined almost everywhere) pointing outward
$\Omega $, and $N^M$ is the unit normal vector field to
$M$ pointing inward $\Omega $ (see Figure~\ref{flux1}).
\begin{figure}
\begin{center}
\includegraphics[height=5cm]{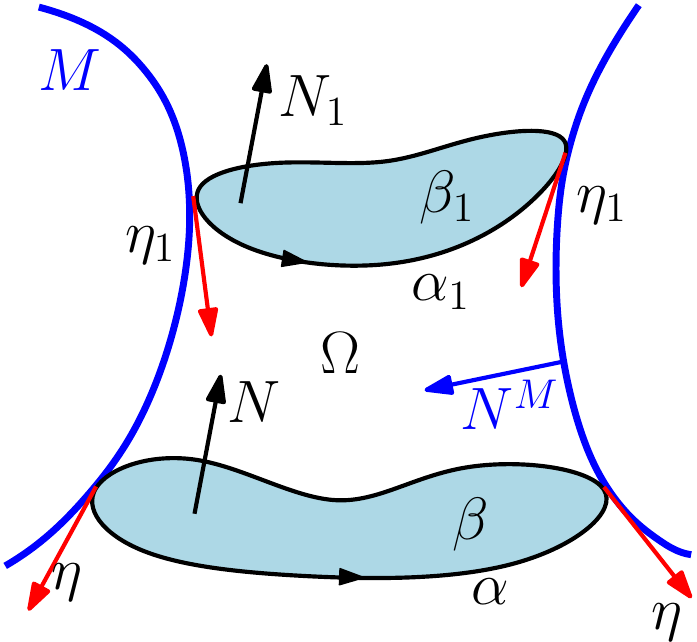}
\caption{Interpretation of the CMC flux homological invariant. }
\label{flux1}
\end{center}
\end{figure}
By the Gauss equation, the divergence in $M$ of the tangential part
$K^{T}$ of $K$ is given by div$^M(K^T)=nH\langle K,N^M\rangle $,
where $H$ is the
 mean curvature of $M$ with respect to $N^M$. Applying the
Divergence Theorem in $M(\a ,\a _1)$,
\begin{equation}
\label{eq:fluxb}
nH\int _{M(\a ,\a _1)}\langle K,N^M\rangle =\int _{\a }\langle K,\eta \rangle
-\int _{\a _1} \langle K,\eta _1\rangle ,
\end{equation}
where $\eta ,-\eta _1$ are the unit conormal vector fields to $M$
 along $\a ,\a _1$ pointing outward $M(\a ,\a _1)$, which are
 defined almost everywhere on $\a ,\a _1$. Equations~(\ref{eq:flux2a}),
 (\ref{eq:fluxb}) show that the real number
\begin{equation}
\label{eq:10.2}
\mbox{Flux}(M,\a ,K)=\int_{\a }\langle K,\eta\rangle + nH
\int_{\beta }\langle K,N\rangle
 \end{equation}
is independent of the choice of the representative $\a $ in a given homology class
$[\a ]\in \ker (i_*)\subset H_{n-1}(M)$ and of the
$n$-chain $\be $ as above. This number is usually
called  the {\it CMC flux} of $M$ along $[\a]$ associated to $K$.
Observe that the mean curvature of $M$ is computed with
respect to the unit normal vector that, along  $\a$, points to the
same closed component of $M\cup \be $ as $N$, and that $\langle N,\eta \rangle
\leq 0$ along $\a $.
The above argument shows the homological invariance of the CMC
flux; we refer the reader to the papers~\cite{hlr,kkms1,kks1,maz4}
 for applications of this CMC flux.
}
\end{definition}

\begin{proof}[Proof of Theorem~\ref{flux}]
Without loss of generality,  we will assume that the Killing field
$K_2$ is bounded on the end $\Sigma\times (-\infty,0]$ of $W$. Let
$\{\phi_s\mid s\in \R\}$ be the 1-parameter subgroup of isometries of
$X$ corresponding to the flow of the Killing field $K_1$. Since the action
$(s,x)\in \R \times W\mapsto \phi _s(x)$ is proper, then after possibly
changing $K_1$ by $-K_1$, we may suppose that $\phi _s(\Sigma
\times \{ 0\} )$ lies eventually in the upper end $\Sigma \times
[0,\infty )$ of $W$ as $s\to +\infty $.
 Consider the
quotient map $\Pi\colon W\to W/G=\wh{W}$, where $G=\{ \phi _s\
| \ s\in 2\pi \Z \} $ (note that the discrete group of isometries
$G$ acts proper discontinuously on $W$). We will denote by $\esf^1$ the related
1-parameter subgroup of isometries of $\wh{W}$.

Let $\wh{\G }$ be an area-minimizing oriented $(n-1)$-dimensional hypersurface in
$\partial \wh{W}$ that represents $[\partial \Sigma ]\in H_{n-1}(\partial \wh{W})$ ($\wh{\G}$ is an area-minimizing
integral current, which is smooth in this dimension and it exists since $\partial \wh{W}$ is a closed Riemannian manifold).
By cut-and-paste type arguments and the maximum principle for
minimal hypersurfaces, any two distinct solutions to this area-minimization  problem
in $\partial \wh{W}$ are disjoint or equal.  It follows that
$\partial \wh{W}$ is foliated by the area-minimizing oriented
$(n-1)$-dimensional submanifolds $\phi _s(\wh{\G})$, $\phi _s\in \esf^1$.
For homological reasons, $\wh{\G}$ has a lift through $\Pi$ which is
a compact $(n-1)$-dimensional hypersurface $\G$ in $\partial W$
that is area-minimizing in $\partial W$,
and when considered to lie in $W$, it is the boundary of an integral
$n$-chain.
By the properness of the $\R$-action giving rise to $K_1$ and the compactness of $\G$,
we deduce that
$K_1$ cannot be everywhere tangent to $\G$. On the other hand, since for $s\neq 0$
it holds $\phi _s (\G)\cap \G=\mbox{\rm \O}$,
then the  Jacobi function  $J_{\G }=\langle K_1,N_{\G}\rangle
\colon \G\to \R$,
is nonnegative on $\G$ (after choosing an
appropriate orientation of $\G $, or equivalently a
unit normal $N_{\G }$ along $\G $, i.e., $N_{\G }$ is tangent to $\partial W$,
orthogonal to $\G $ and unitary); note
that $J_{\G }$ is positive at
the   points of $\G$ where  $K_1$ is not tangent to $\G$.  The
maximum principle (see e.g.,~\cite[ Assertion~2.2]{mpr19}) implies that $J_{\G }>0$ on $\G $;
in particular, $K_1$ is transverse to $\G$ and each integral curve of
$K_1$ intersects $\G$ in a single point.

Let $\Delta\subset W$ be an area-minimizing hypersurface with boundary $\G$, which in
 this dimension is smooth by standard interior and boundary regularity results
 for minimizing integral varifolds; see for example, Hardt and Simon~\cite{hrs1}
 for boundary regularity results and also see~\cite{my4,my2} for
 related barrier arguments.
 By the mean curvature comparison principle,   $\Delta $ is never tangent to $\partial W$ along $\G $.
 Since $J_{\G }>0$ by the previous paragraph, then  we can choose a unit
 normal vector $N_{\Delta }$ to $\Delta $ such that
$\langle (N_{\Delta })|_{\partial \G },N_{\G }\rangle $ is positive along $\G$.
As $K_1$ is tangent to $\partial W$, $(N_{\Delta })|_{\G}$ is orthogonal to $\G $
and both $\langle K_1,N_{\G }\rangle $,
$\langle (N_{\Delta })|_{\partial \G },N_{\G }\rangle $ are positive
along $\G$, then we conclude that $\langle K_1,N_{\Delta }\rangle >0$ along $\G $.
As $\Delta $ is stable and $\langle K_1,N_{\Delta }\rangle >0$ along $\partial \Delta $, then
$\langle K_1,N_{\Delta }\rangle  $ is positive on $\Delta $ and so, $\Delta$
intersects each of the integral curves of $K_1$ transversely in a single point.
Therefore, the set $\cF=\{ \phi_s(\Delta)\mid s \in \R\}$ is a smooth minimal foliation of $W$,
where each leaf $\phi_s(\Delta)$ of this foliation intersects every integral curve of $K_1$
transversely in a single point.

We will now deform the minimal surface $\Delta$ by a 1-parameter family
of surfaces with constant mean curvature and the same boundary as $\Delta$.
By the "blowing a bubble" technique described
in~\cite{kkms1}, for $\delta>0$ sufficiently small there exists a
unique 1-parameter family of hypersurfaces $t\in [0,\de )\mapsto
\Delta(t)\subset W$ depending smoothly on $t$ such that for all
$t\in [0,\de )$:

\begin{enumerate}[(X1)]
\item $\Delta (t)$ is the graph of a smooth
function $f_t\colon \Delta\to \R$
(in the direction of $K_1$), with $\Delta (0)=\Delta $ and $f_0=0$.

\item  $\Delta (t)$ has constant mean curvature $-t$ with
respect to the unit normal vector field $N_{\Delta (t)}$
to $\Delta (t)$ that
satisfies $\langle K_1, N_{\Delta (t)}\rangle >0$.

\item $f_t=0$ on $\partial \Delta $.

\item $f_t(x)>f_{t'}(x)$ whenever $t>t'$
and $x\in {\rm Int}(\Delta)$.
\end{enumerate}

We now study the maximal half-open interval $[0, T_0)$ of $t$-values
in which the family $\Delta (t)$ can be defined (here $T_0\in (0,\infty ]$).

First consider the case in which the constant mean curvature $h_0$ of $\partial W $
 satisfies $h_0<T_0$. Then, there exists a hypersurface $\Delta (t_1)\subset X$ with
constant mean curvature $-t_1$, which is the graph of a function
$f_{t_1}\colon \Delta \to \R$ such that $f_{t_1}=0$ on $\partial \Delta $ and $t_1>h_0$. Since $\partial \Delta
(0)=\partial \Delta (t_1)$, then the CMC flux of $\Delta (t_1)$ along
$[\partial \Delta (t_1)]$ associated to $K_1$ is given by
\[
 \mbox{Flux}(\Delta (t_1),\partial \Delta (t_1),K_1)=
\int_{\partial \Delta (0)} \langle K_1,\eta_{\Delta (t_1)}\rangle
+nt_1\int_{\Delta(0)}\langle K_1, N_{\Delta (0)}\rangle,
 \]
where $\eta_{\Delta (t_1)}$ is the outward pointing unit conormal
vector field to $\Delta(t_1)$ along $\partial \Delta (t_1)$ and
$N_{\Delta (0)}$ is the unit normal vector field to $\Delta (0)$
pointing towards the compact region of $W$ bounded by $\Delta(0)$
and $\Delta(t_1)$. Note that
\[
\mbox{Flux}(\Delta (t_1),\partial
\Delta (t_1),K_1)=0
\]
 by the homological invariance of the CMC flux
 ($\partial \Delta (t_1)$ is homologically trivial in
 $\Delta (t_1)$). On the other hand, along $\partial \Delta (0)$ the
 inequality
\begin{equation}
\label{eq:FLUX}
\langle K_1,\eta _{(\partial W)^+}\rangle \leq
 \langle K_1,\eta _{\Delta (t_1)}\rangle
\end{equation}
holds, where $\eta _{(\partial W)^+}$ denotes the outward pointing
unit conormal field to the portion $(\partial W)^+$ of $\partial W$
which lies above $\partial \Delta (0)$. This last inequality
together with $h_0<t_1$ imply that
\[
\begin{array}{rcl}
\mbox{Flux}\left( (\partial W)^+,\partial \Delta (0),K_1\right) &=&
{\displaystyle \int_{\partial \Delta (0)} \langle K_1,\eta_{(\partial
W)^+}\rangle
 + nh_0\int_{\Delta(0)}\langle K_1, N_{\Delta (0)}\rangle}
 \\
& < & {\displaystyle
\int_{\partial \Delta (0)} \langle K_1,\eta_{(\partial
W)^+}\rangle
 + nt_1\int_{\Delta(0)}\langle K_1, N_{\Delta (0)}\rangle}
 \\
& \stackrel{(\ref{eq:FLUX})}{\leq } & \mbox{Flux}(\Delta (t_1),\partial \Delta (t_1),K_1)= 0.
\end{array}
 \]
Finally, the homological invariance of the CMC flux gives
\[
 \mbox{Flux}\left( \partial W,\partial \Sigma \times \{ 0\} ,K_1\right) =
 \mbox{Flux}\left( (\partial W)^+,\partial \Sigma \times \{ 0\} ,K_1\right) =
\mbox{Flux}\left( (\partial W)^+,\partial \Delta (0),K_1\right) < 0,
 \]
which completes the proof of the theorem in the case $h_0<T_0$.

Assume in the sequel that $T_0\leq h_0$ and we will find a contradiction,
which will finish the proof of the theorem. We first prove the following property.
\begin{assertion}
\label{ass11.3}
$\cup _{t\in [0,T_0)}\Delta (t)$ is not contained in any compact region of
$W$. In particular, for every $n\in \N$ there
exists a $t(n)\in [0,T_0)$ such that $\Delta(t)$ contains points at
distance at least $n$ from $\Delta(0)$ whenever $t\geq t(n)$.
\end{assertion}
\begin{proof}[Proof of the assertion.]
Suppose the assertion does not hold.
Then, the constant
mean curvature hypersurfaces $\{ \Delta (t)\ | \ t\in [0,T_0)\} $ have uniform height estimates (as $K_1$-graphs).
By the "blowing a bubble" technique
 in~\cite{kkms1}, these height estimates imply that the  $\{ \Delta (t)\ | \ t\in [0,T_0)\} $
 satisfy uniform curvature estimates,
i.e.,
they have {\it uniform}
bounds on the length of their second fundamental forms; these uniform curvature estimates
can also be seen to hold since $\Delta (t)$ minimizes the
functional $\mbox{\rm Volume}_n+nt\mbox{\rm Volume}_{n+1}$
with respect to compact hypersurfaces in $W$ with boundary
$\partial \Delta$ that are homologous to $\Delta$ relative to $\partial \Delta$ ,
where $\mbox{\rm Volume}_k$ is the volume of any
$k$-chain.
From here, standard compactness results show
that the top boundary component of $\cup_{t\in [0,T_0)} \Delta(t)$ is a
hypersurface $\Delta (T_0)$ with constant mean curvature $-T_0$.

Note that with respect to the unit normal vector field $N_{\Delta (T_0)}$ to $\Delta (T_0)$ that
is the limit of the $N_{\Delta (t)}$ as $t\nearrow T_0$, the function
$\langle K_1,N_{\Delta (T_0)}\rangle $ is nonnegative, and it is positive at some point
by previous arguments. Since $\langle K_1,N_{\Delta (T_0)}\rangle $ is a Jacobi function
on $\Delta (T_0)$, then the maximum principle (see e.g.,~\cite[Assertion~2.2]{mpr19})
implies that $\langle K_1,N_{\Delta (T_0)}\rangle >0$
at Int$(\Delta (T_0))$. In particular, $\Delta (T_0)$ is the $K_1$-graph of a continuous function $f_{T_0}\colon
\Delta \to \R $ that is  smooth at Int$(\Delta (T_0))$.
As $\Delta(T_0)$ is not contained in $\partial W$, then the mean
curvature comparison principle and the Hopf boundary maximum
principle imply that the gradient of $f_{T_0}$ is bounded as we approach $\partial \Delta(T_0)$;
hence, $\langle K_1,N_{\Delta (T_0)}\rangle $ is
positive along $\partial \Delta(T_0)$. It
follows from the maximum principle that $\langle K_1,N_{\Delta (T_0)}\rangle $ is
positive in $\Delta(T_0)$
and thus, $f_{T_0}$ is smooth in $\Delta $.
Therefore, items~(X1)-(X4) hold for $t=T_0$.
In this situation, the strict stability of $\Delta (T_0)$ implies that the
family of hypersurfaces $\{ \Delta (t)\ | \ t\in [0,T_0]\} $ can be
extended to a related family of hypersurfaces defined for $t$-values
in a larger interval $[0,T_0+\ve )$, $\ve
>0$, which contradicts the definition of the value $T_0$.
This contradiction proves the assertion.
\end{proof}

For each $t\in [0,T_0)$ fixed,
 consider the related foliation of $W$
 \[
 {\calf}(t)=\{ \phi_s(\Delta(t))\ | \ s\in \R \} ,
\]
all whose leaves have constant mean curvature $-t$, see
Figure~\ref{flux2} left.
\begin{figure}
\begin{center}
\includegraphics[width=13cm]{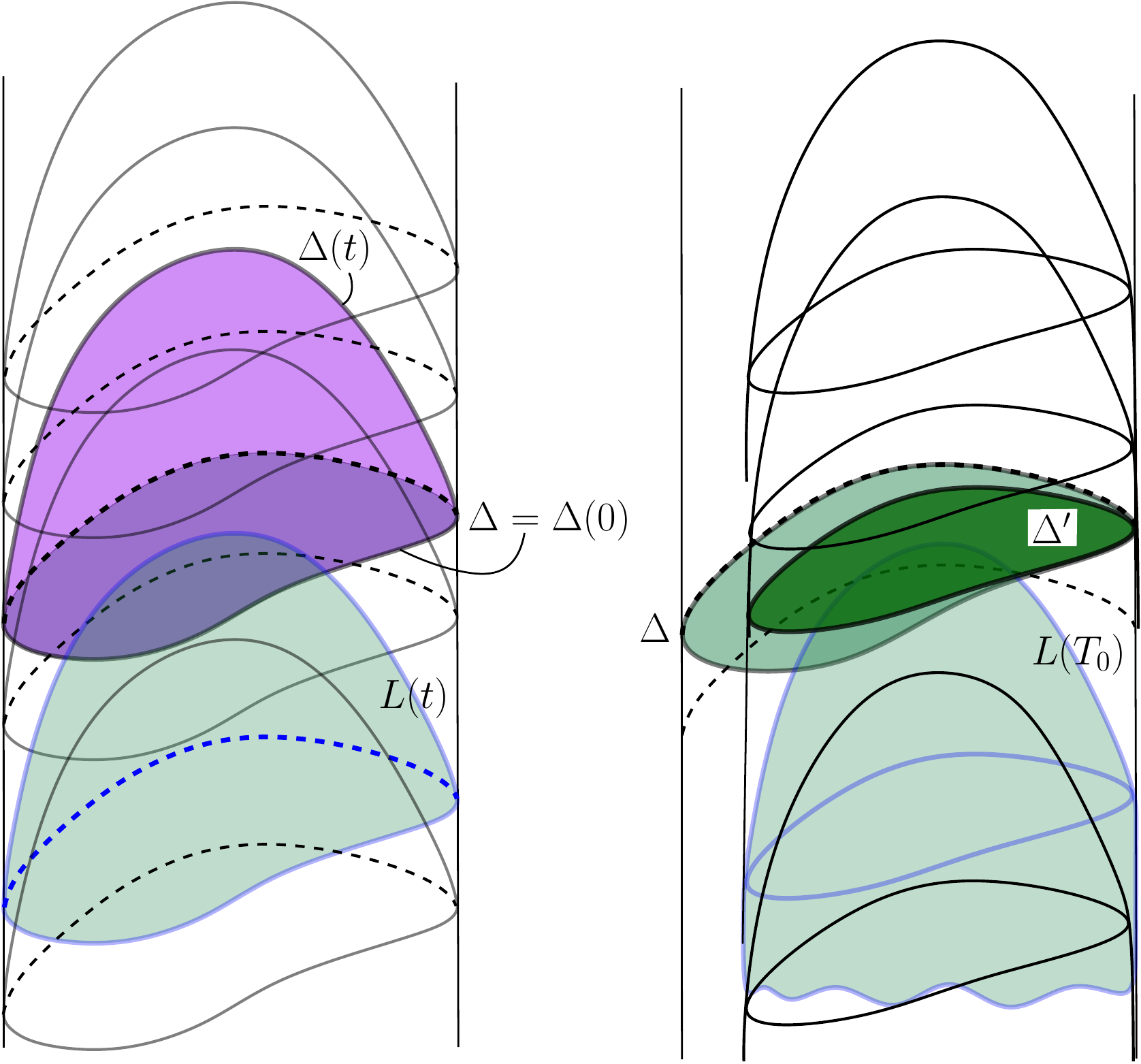}
\caption{
Left: After blowing a bubble $\Delta (t)$ with constant mean curvature $-t\in
(-T_0,0]$ (which is $K_1$-graphical over $\Delta =\Delta (0)$), we translate it with
$\{ \phi _s\} _{s\in \R}$ to produce a foliation ${\mathcal F}(t)$, and
we highlight the highest graphical leaf $L(t)$ of ${\mathcal F}(t)$ lying below
$\Delta (0)$ with $L(t)\cap \Delta (0)\neq \mbox{\O }$. Right: Since
the $\Delta (t)$ with $t\to T_0^-$ do not lie in a fixed compact
set, we produce a noncompact limit $L(T_0)=\lim _{t\to T_0^-}L(t)$
which is a graph over a domain $\Delta'$ of $\Delta $.
}
\label{flux2}
\end{center}
\end{figure}
Let
 \[
 L(t)=\phi_{s(t)}(\Delta(t))
 \]
  be the unique leaf of $\calf(t)$ such that $L(t)$ intersects $\Delta (0)$
and lies on the lower side of $\Delta (0)$. Assertion~\ref{ass11.3}
implies that $s(t)\to -\infty$
as $t\to T_0$. Let $L(T_0)$ be the limit of the hypersurfaces $L(t)$
as $t\nearrow T_0$. Such a limit exists and it is a complete graphical
hypersurface with constant mean curvature $-T_0$ by the stability of
$L(t)$ for all $t<T_0$, and the same arguments as in the proof of
Assertion~\ref{ass11.3}. Let $\Delta'$
be the open subdomain of $\Delta$ over which $L(T_0)$ is a
$K_1$-graph. Clearly, $L(T_0)$ intersects
$\Delta (0)$ and lies below $\Delta (0)$.
Define $t_0\in \R $ such that $L(T_0)$ lies in $\Sigma \times (-\infty ,t_0]$ and
$t_0$ is the smallest value with this property.
Observe that $L(T_0)$ intersects  $\Sigma \times \{ t_0\} $ tangentially at some interior point
of $\Sigma \times \{ t_0\} $.

\begin{assertion}
\label{ass11.4}
$L(T_0)$ has linear volume growth.  In particular,
$L(T_0)$ is parabolic. 
\end{assertion}
\begin{proof}[Proof of the assertion.]
A straightforward limit argument shows that $L(T_0)$ minimizes the functional
 $\mbox{\rm Volume}_n+nT_0\mbox{\rm Volume}_{n+1}$ on compact subsets of Int$(W)$.
 This property implies that $L(T_0)$ has bounded second fundamental form.
 Therefore, the gradient of the graphing function $h\colon \Delta'\to \R$ that produces $L(T_0)$
as a $K_1$-graph becomes unbounded as we approach
any point of the boundary of $\Delta'$
(otherwise we could enlarge $\Delta'$). Therefore
 the assertion holds provided that we check that the total $(n-1)$-dimensional
 volume function of the  submanifold
$\G_s:=L(T_0)\cap \phi _s(\Delta (0))$ of $L(T_0)$
is bounded as $s\to -\infty $.
To see this, for $s$ large ($s\ll -1$) let us denote by $\Omega (s)$ the portion of $L(T_0)$
lying above $\phi _s(\Delta (0))$, and let $\eta _s$ be the outward
pointing unit conormal vector field of $\Omega (s)$ along its
boundary $\G _s$. Note that for every $\ve >0$ small, there exists
$s_0(\ve )\ll -1$ such that for each $s\leq s_0$,
\begin{equation}
\label{eq:flux2}
\langle K_1,\eta _s\rangle < -\ve \quad \mbox{along }\G _s,
\end{equation}
for some $\ve >0$ independent of $s$.
On the other hand,
\begin{equation}
\label{eq:flux1}
 \mbox{Flux}(\Omega (s),\G _s,K_1)= \int _{\G
_s}\langle K_1, \eta _s\rangle +nT_0\int _{\phi _s(\Delta_s)}\langle K_1, N_{\phi_s(\Delta_s)}\rangle ,
\end{equation}
where $\Delta _s$ is the subdomain of $\Delta'$ over which
$\Omega (s)$ is a $K_1$-graph, and $N_{\phi_s(\Delta _s)}=(\phi _s)_*(N_{\Delta (0)})$.
By homological invariance of the flux, we have
$ \mbox{Flux}(\Omega (s),\G _s,K_1)=0$. As the second integral in the right-hand-side of
(\ref{eq:flux1}) is bounded above in absolute value  by $| K_1|_{\infty }$ times the
$n$-dimensional volume of $\Delta (0)$
(here $|K_1| _{\infty }=\max _{\Delta }|K_1|$, which exists since $\Delta$ is compact
and $K_1$ is generated by $\{ \phi _s\ | \ s\in \R \} $),
then the first integral in
the right-hand-side of (\ref{eq:flux1}) is bounded independently of $s\to -\infty$.
This property together with (\ref{eq:flux2}) ensure that the
$(n-1)$-dimensional volume of
$\G_s$ is bounded as $s\to -\infty $. Now the assertion is proved.
\end{proof}

Recall that by assumption, $K_2$ is bounded in $\Sigma \times
(-\infty ,0]$. Define $u=\langle K_2,N_{L(T_0)}\rangle \colon L(T_0)\to \R$,
where $N_{L(T_0)}$ denotes a unit normal vector field along $L(T_0)$.
$u$ is a bounded Jacobi function on $L(T_0)$, which vanishes at each of the
points of the nonempty set $L(T_0)\cap (\Sigma \times \{ t_0\} )$.  Since every integral curve of
$K_2$ that intersects $L(T_0)$ is a compact interval with its boundary in $\partial W$,
then $K_2$ is not everywhere tangent to $L(T_0)$, which implies $u$ is not identically
zero; by the maximum principle, $u$ must change sign in
a neighborhood of any of the points in $L(T_0)\cap (\Sigma \times \{ t_0\} )$.
As $L(T_0)$ is complete, stable, and is parabolic by Assertion~\ref{ass11.4}, we conclude that
every bounded nonzero Jacobi function on
$L(T_0)$ has constant sign~\cite[Corollary~1]{mper1},
which is a contradiction. This contradiction
finishes the proof of Theorem~\ref{flux}.
\end{proof}

\center{William H. Meeks, III at  profmeeks@gmail.com\\
Mathematics Department, University of Massachusetts, Amherst, MA 01003}
\center{Pablo Mira at  pablo.mira@upct.es\\
Departamento de Matem\'atica Aplicada y Estad\'\i stica, Universidad
Polit\'ecnica de Cartagena, E-30203 Cartagena, Murcia, Spain}
\center{Joaqu\'\i n P\'{e}rez at jperez@ugr.es \qquad\qquad Antonio Ros at aros@ugr.es\\
Department of Geometry and Topology and Institute of Mathematics
(IEMath-GR), University of Granada, 18071, Granada, Spain}

\bibliographystyle{plain}
\bibliography{bill}
\end{document}